\documentclass[]{article}
\usepackage{graphicx}
\usepackage{amsfonts}
\usepackage{amsmath}
\usepackage{amssymb}
\usepackage{fancyhdr}
\usepackage{titlesec}
\usepackage{indentfirst}
\usepackage{booktabs}
\usepackage{verbatim}
\usepackage{color}
\usepackage{amsthm}
\usepackage{subfigure}

\usepackage[page,header]{appendix}
\usepackage{titletoc}

\newcommand{\rd}{\,\mathrm{d}}
\numberwithin{equation}{section}
\newtheorem{theorem}{Theorem}[section]
\newtheorem{lemma}[theorem]{Lemma}
\newtheorem{corollary}[theorem]{Corollary}
\newtheorem{proposition}[theorem]{Proposition}

\def\bu{{\bf u}}
\def\bx{{\bf x}}
\def\by{{\bf y}}

\def\bv{{\bf v}}

\def\cT{\mathcal{T}}
\def\cR{\mathcal{R}}
\def\crr{\tilde{\mathcal{R}}}
\def\cI{\mathcal{I}}
\def\cS{\mathcal{S}}

\def\tI{\tilde{I}}
\def\km{\mathfrak{m}}
\def\kc{\mathfrak{c}}

\def\supp{\textnormal{supp}}
\def\ini{\textnormal{in}}

\begin{document}

\title{Tightness of radially-symmetric solutions to 2D aggregation-diffusion equations with weak interaction forces}
\author{Ruiwen Shu\footnote{Department of Mathematics, University of Maryland, College Park, MD 20742, USA (rshu@cscamm.umd.edu). The author's research was supported in part by NSF grants DMS-1613911 and ONR grant N00014-1812465.}}
\maketitle

\begin{abstract}
We prove the tightness of radially-symmetric solutions to 2D aggregation-diffusion equations, where the pairwise attraction force is possibly degenerate at large distance. We first reduce the problem into the finiteness of a time integral in the density on a bounded region, by introducing a new assumption on the interaction potential called the essentially radially contractive (ERC) property. Then we prove this finiteness by using the 2-Wasserstein gradient flow structure, combining with the continuous Steiner symmetrization curves and the local clustering curve. This is the first tightness result on the 2D aggregation-diffusion equations for a general class of weakly confining potentials, i.e., those $W(r)$ with $\lim_{r\rightarrow\infty}W(r)<\infty$, and serves as an important step towards the study of equilibration.
\end{abstract}

\section{Introduction}

In this paper we continue the study in~\cite{S1D} on the large time behavior of the aggregation-diffusion equation
\begin{equation}\label{eq0}
\partial_t \rho + \nabla\cdot(\rho \bu) = \Delta (\rho^m),\quad \bu(t,\bx) = -\int \nabla W(\bx-\by) \rho(t,\by)\rd{\by},\quad m>1,
\end{equation}
where $\rho(t,\bx)$ is the density distribution function of a large group of particles, $\bx\in\mathbb{R}^d$ being the spatial variable, $t\in\mathbb{R}_{\ge0}$ being the temporal variable. The term $\nabla\cdot(\rho \bu)$ describes the pairwise attraction force among particles, given by a radial interaction potential  $W=W(r),\,r=|\bx|,\,W'(r)>0$. The term $\Delta (\rho^m)$ with $m>1$ is a porous-medium type diffusion term, modeling the localized pairwise repulsion forces among particles~\cite{Oel}, making the particles less likely to concentrate. \eqref{eq0} arises naturally in the study of the collective behavior of large groups of swarming agents~\cite{MEK1,MEK2,BurCM,BCM,MCO,TBL} and the chemotaxis phenomena of bacteria~\cite{Pat,KS,JL,Hor,BDP,BCL,CC}. The 2D case ($d=2$) is of significant importance because of its application in chemotaxis: in fact, \eqref{eq0} with $m=1$ and $W$ being the Newtonian attraction is exactly the well-known Keller-Segel equation~\cite{Pat,KS}.

\eqref{eq0} is (at least formally) the 2-Wasserstein gradient flow of the total energy, as the sum of the internal energy and the interaction energy
\begin{equation}\label{E}
E[\rho] = \cS[\rho] + \cI[\rho] := \frac{1}{m-1}\int \rho(\bx)^m \rd{\bx} + \frac{1}{2}\int \int W(\bx-\by)\rho(\by)\rd{\by}\rho(\bx)\rd{\bx},
\end{equation}
in the sense that 
\begin{equation}
\frac{\rd E}{\rd{t}} = -\Big\|\frac{\delta E}{\delta \rho}\Big\|^2 = -\int 
\Big|\bu(t,\bx) - \frac{m}{m-1}\nabla(\rho(t,\bx)^{m-1})\Big|^2\rho(t,\bx)\rd{\bx},\quad E(t) := E[\rho(t,\cdot)].
\end{equation}
Therefore, it is natural to study the \emph{equilibration} of \eqref{eq0}, i.e., when $t\rightarrow\infty$, whether it is true that the solution $\rho(t,\cdot)$ converges to a steady state $\rho_\infty$, characterized by
\begin{equation}
\bu_\infty(\bx) - \frac{m}{m-1}\nabla(\rho_\infty(\bx)^{m-1}) = 0,\quad \forall \bx \in \supp \rho_\infty.
\end{equation}
We will also use the bilinear version of the interaction energy
\begin{equation}\label{bilinear}
\cI[\rho_1,\rho_2] := \int \int W(\bx-\by)\rho_1(\by)\rd{\by}\rho_2(\bx)\rd{\bx},\quad \cI[\rho] = \frac{1}{2}\cI[\rho,\rho].
\end{equation}

There has been a rich literature on the study of steady states of \eqref{eq0}, including existence, uniqueness, structures, etc., see~\cite{Bed,LY,Str,CCV,KY,CHVY,CCH,CHMV,CCH2,DYY}, and we refer to~\cite{S1D} for a discussion on them. For the case $m\ge 2$ and $W$ is no more singular than the Newtonian attraction near the origin, the results in \cite{Bed,CHVY,DYY} imply that the steady state exists, is unique, and is radially-decreasing\footnote{A density distribution $\rho(\bx)$ is radially-symmetric if $\rho$ is a function of $r=|\bx|$, and radially-decreasing if $\rho(r)$, as a function of $r$, is decreasing on $[0,\infty)$.}.
 
However, there are very few existing results on the equilibration of \eqref{eq0}. In fact, \cite{KY} proves the equilibration of radially-symmetric solutions for Newtonian attraction and its variant (to be precise, the convolution of the Newtonian potential with a radially-decreasing function with certain regularity), with exponential convergence rate. \cite{CHVY} proves the equilibration of general solutions for 2D Newtonian attraction, without an explicit convergence rate. Both results rely on the special structure of the Newtonian attraction potential.

To prove equilibration of \eqref{eq0} with general interaction potentials, the biggest difficulty is the issue of \emph{tightness}, i.e., whether there holds
\begin{equation}\label{tight}
\text{For all }\epsilon>0, \text{ there exists }R>0, \text{ such that } \int_{|\bx|> R}\rho(t,\bx)\rd{\bx} < \epsilon,\,\forall t\ge 0.
\end{equation}
Tightness could avoid a positive amount of mass from escaping to infinity, and provide compactness of the solution $\{\rho(t,\cdot)\}_{t\ge 0}$. Tightness allows one to subtract a subsequence $\{\rho(t_k,\cdot)\}_{k=1}^\infty$ with $\lim_{k\rightarrow\infty}t_k=\infty$, converging strongly, and then equilibration could follow from the energy balance \eqref{E}.

Tightness is especially hard to obtain in the case of \emph{weakly confining potentials}, defined as those $W$ with $\lim_{r\rightarrow\infty}W(r)<\infty$, because there is no hope to obtain tightness directly from the energy balance \eqref{E}.  The author's recent work \cite{S1D} gives the first result for the equilibration of \eqref{eq0} for a \emph{general class} of weakly confining potentials. In \cite{S1D}, the key assumptions include the following: the spatial dimension $d=1$; $W'(r)$ has a lower bound of the type $r^{-\alpha}$ for large $r$; $m$ is large in the sense that $m>\max\{\alpha,2\}$; the initial data is radially-symmetric and the initial energy is not too large. Under these assumptions, \cite{S1D} proved the convergence to the steady state as $t\rightarrow\infty$ with algebraic rate, in the sense that $E(t)-E[\rho_\infty] \le C(1+t)^{-\gamma},\,\gamma>0$.

The purpose of the current work is to generalize the tightness result in \cite{S1D} from 1D to 2D, which would be the key step towards the study of equilibration of \eqref{eq0} in 2D. We first recall the 1D tightness obtained in \cite{S1D}, the uniform bound of the first moment: (where $\bx$ is denoted as $x$ in 1D)
\begin{equation}\label{moment1D}
\int |x|\rho(t,x)\rd{x} < C,\quad \forall t\ge 0.
\end{equation}
Its proof includes the following major steps:
\begin{itemize}
\item Reduce \eqref{moment1D} to the finiteness of the time integral
\begin{equation}\label{int1D}
\int_0^\infty \int_{5R_1}^{6R_1}\rho(t,x)\rd{x}\rd{t},
\end{equation}
where $R_1>0$ is large. This is done by considering  the $\phi$-moment $\km_\phi(t)=\int \phi(x)\rho(t,x)\rd{x}$, where $\phi = \chi_{5R_1\le|x|\le 6R_1]}*|x|$, with $|x|$ being the fundamental solution to the Laplacian operator. In the time evolution of the $\phi$-moment, the contribution from the attraction term is always negative, due to the convexity of $\phi$, and the contribution from the diffusion term can be controlled by \eqref{int1D}, using $\Delta \phi = \chi_{5R_1\le|x|\le 6R_1}$.

\item Prove the finiteness of $\int_0^\infty \int_{5R_1}^{6R_1}\mu(t,x)\rd{x}\rd{t}$, where $\mu$ is the non-radially-decreasing part of $\rho$ (see \eqref{rhoh1} through \eqref{rhoh3} for definition). Based on the gradient flow structure of \eqref{eq0}, this is done by the \emph{continuous Steiner symmetrization} (CSS), first introduced by \cite{CHVY} (denoted as CSS1), being a curve of density distributions, along which the total energy is decreasing whenever the initial distribution $\rho$ is not radially-decreasing. A variant of it (denoted as CSS2) is designed to handle the degeneracy of $W'(r)$ at large $r$.

\item Prove the finiteness of $\int_0^\infty \int_{5R_1}^{6R_1}\rho^*(t,x)\rd{x}\rd{t}$, where $\rho^*$ is the radially-decreasing part of $\rho$. This is done by the \emph{local clustering curve}, which is a way of transporting a given density distribution $\rho$, imitating the formation of local clusters at low density regions. The latter has been numerically observed in \cite{BCH,CCWW}, and appears to be a key difficulty in the study of \eqref{eq0}. 
\end{itemize}

In the current work, we prove the tightness of 2D radially-symmetric solution\footnote{For radially-symmetric solution, one could rewrite \eqref{eq0} into a 1D PDE in the radial variable $r$. However, in this paper we still use the original PDE \eqref{eq0} because it is more convenient for most of the techniques.} to \eqref{eq0}, see Theorem \ref{thm1}, under certain assumptions which will be specified later. The proof follows the same steps as \cite{S1D}, but two major new difficulties arise:
\begin{itemize}
\item In 2D, it is natural to take the test function $\phi = \chi_{5R_1\le |\bx|\le 6R_1}*\frac{1}{2\pi}\ln|\bx|$, with $\frac{1}{2\pi}\ln|\bx|$ being the fundamental solution to the Laplacian operator, as a generalization to 1D. $\phi(\bx)\sim C\ln|\bx|$ for large $|\bx|$, and thus the uniform boundedness of the $\phi$-moment implies tightness. However, $\phi$ is \emph{no longer convex} (see Figure \ref{fig_pic10} in Section \ref{sec_main}): in this case the attraction term may give positive contribution to the time evolution of the $\phi$-moment, making it out of control. 

To handle this difficulty, we introduce a new concept called the \emph{essentially radially contractive} (ERC) property for the interaction potential $W$, see Theorem \ref{thm_erc}. The ERC property basically says that the attraction among annuli far away from the origin does not increase the $\phi$-moment, and it can be guaranteed under a clean assumption \eqref{Wassu1}. This assumption allows the attraction force $W'(r)$ to behave like $r^{-(3-\epsilon)}$, thus allowing weakly confining potentials. In fact, the Newtonian attraction is $W'(r)=r^{-1}$, and thus we allow the attraction force to be almost two orders more degenerate than the Newtonian, for large $r$.

\item In 2D, the CSS curve is ineffective in decreasing the interaction energy, for (non-radially-decreasing) level sets of the form $\{\bx: r_j\le |\bx| \le R_j\}$ with $R_j>>r_j$, see Figure \ref{fig_pic47} (left) in Section \ref{sec_finite}. Therefore the CSS curve cannot control the time integral of such `flat' level sets. 

To handle this difficulty, we manage to show that the `flat' level sets, denoted as $\rho_\flat$, behave similarly to the radially-decreasing part, in terms of its degeneracy at large radius (see \eqref{rhosmall}), as well as the potential field generated by them (see Lemma \ref{lem_ctr}). Therefore, the `sharp' non-radially-decreasing part is controlled by CSS, while the `flat' part, together with the radially-decreasing part, is controlled by the local clustering curve with an improved energy estimate.

\end{itemize}

Finally, we would like to give some intuitions about what one could expect for the tightness of general 2D solutions to \eqref{eq0}. In fact, radially-symmetric solutions are `\emph{unstable}', in the following sense: when the initial data is close to but not exactly radially-symmetric, and the attraction force is well-localized, one expects the formation of local clusters within a relatively short time scale, which breaks the radial symmetry. In other words, if (intuitively) one uses some quantity $d(t)$ to measure the distance between $\rho(t,\cdot)$ and the set of radially-symmetric functions, then $d(t) \sim d(0)e^{\lambda t},\,\lambda>0$ within some time interval $[0,T]$.

However, such `instability' should be distinguished from the concept of instability of steady states: in the former case, assuming equilibration holds, then $\rho(t,\cdot)\approx \rho_\infty$ if $t$ is large, with $\rho_\infty$ being radially-symmetric, which basically implies that $d(t)$ is uniformly small for all time if $d(0)$ is small; in the latter case, for unstable steady states, even if the initial data deviates from it by an arbitrarily small amount, this deviation will eventually grow to $O(1)$.

Therefore, we expect the general 2D solutions can be separated into two scenarios:
\begin{itemize}
\item $d(0)$ is small enough. Then $d(t)$ is small for all time, and the whole solution could be treated as a perturbation of radially-symmetric solutions studied in the current paper.
\item Otherwise, one expects the emergence of a finite number of local clusters at some time $T$. Then we are more or less back to the 1D situation, in terms of tightness: for example, if there are only two clusters, centered at $(-a,0)$ and $(a,0)$, then it is less likely to have some mass escaping to infinity towards the $x_2$-direction, and one could focus on the tightness issue for the $x_1$-direction. In other words, a test function $\phi(\bx) \approx \chi_{5R_1\le |x_1|\le 6R_1}*\frac{1}{2\pi}\ln|\bx|$, which grows like $|x_1|$ when $x_1$ is large, might work for the tightness estimates.
\end{itemize}

Therefore, despite its `instability', the study of radially-symmetric solutions still appears to be an important step towards the study of general 2D solutions to \eqref{eq0}.

The rest of this paper is organized as follows: in Section 2 we state the main result (Theorem \ref{thm1}), and give an outline of its proof, including the important intermediate results; in Section 3 we prove Theorem \ref{thm_erc} which guarantees the ERC property under the assumptions of the main result; in Section 4 we give some lemmas which will be used in later proofs; in Section 5 we use the CSS curves and the local clustering curve to prove Theorem \ref{thm_finite}, the finiteness of the time integral in $\rho$.

\section{The main result}

In this paper, all $C$ (for large constants) and $c$ (for small constants) will denote positive constants which may depend on $W$, $m$, $\rho_{\ini}$, if not stated otherwise, and they may differ from line to line.

\subsection{Assumptions}
We propose the following assumptions:
\begin{itemize}

\item {\bf (A1)} The interaction potential $W$ is radial: $W=W(r),\,r=|\bx|$. $W$ is attractive: $W'(r)>0,\,\forall r>0$, and satisfies
\begin{equation}\label{Wassu1}
-\alpha \le \frac{rW''(r)}{W'(r)} \le A,\quad \forall r>0,
\end{equation}
for some $1< \alpha< 3$ and $A> 0$.

\item {\bf (A2)} The interaction potential $W$ satisfies the following upper  bound:
\begin{equation}
W'(r) \le Cr^{-1},\quad \forall r>0.
\end{equation}

\item {\bf (A3)} $m>(\alpha+1)/2$.

\item {\bf (A4)}  The initial data is radially-symmetric: $\rho_{\ini}(\bx)=\rho_{\ini}(r)$, non-negative, having compact support, $\|\rho_{\ini}\|_{L^\infty} < \infty$, and the total mass is normalized to 1: $\int \rho_{\ini}(\bx)\rd{\bx} = 1$.

\item {\bf (A5)}  The initial data satisfies the sub-critical condition
\begin{equation}
E[\rho_{\ini}] < \frac{1}{2}\lim_{r\rightarrow\infty}W(r).
\end{equation}.

\end{itemize}

We make the following remarks on the assumptions:
\begin{itemize}
\item It is important to notice that {\bf (A1)} implies that 
\begin{equation}\label{Winc}
(W'(r)r^{\alpha})' = W''(r)r^{\alpha} + \alpha W'(r)r^{\alpha-1} \ge 0,
\end{equation}
and thus $W'(r)r^{\alpha}$ is an increasing function in $r$. In other words, there holds the lower bound
\begin{equation}\label{lambda}
W'(r) \ge c r^{-\alpha}=:\lambda(r),\quad \forall r\ge 1.
\end{equation}
One example which satisfies {\bf (A1)} and {\bf (A2)} is $W'(r) = r^{-1}(1+r)^{-(2-\epsilon)}$ for any $\epsilon>0$, which gives rise to a weakly confining potential $W(r)$, behaving like $r^{-(2-\epsilon)}$ for large $r$.

\item The strength of the assumption {\bf (A1)} is between structural and non-structural: on one hand it allows the potential $W(r)$ to change convexity as a function of $r$, but on the other hand it requires properties besides merely the size of $W'$ (the reader is invited to compare it with the assumption {\bf (A2)} in~\cite{S1D}).

In fact, the precise form of {\bf (A1)} is only used to guarantee the ERC property, as defined in Theorem \ref{thm_erc}. Other parts of the proof only rely on {\bf (A1)} through \eqref{lambda}. Therefore, the main theorem (Theorem \ref{thm1}) is still true if one replaces {\bf (A1)} by the ERC property, in case \eqref{lambda} holds. Clearly {\bf (A1)} is not a necessary condition for the ERC property, and it remains open to explore more families of ERC potentials.

\item The assumption {\bf (A2)} says that the potential $W$ is no more singular than the Newtonian attraction near the origin. In the proof, it implies Lemma \ref{lem_circle}, which is used frequently to estimate bad terms. 

\item In {\bf (A3)}, notice that it does not require $m\ge 2$, which allows situations where the steady states may not be unique, c.f.~\cite{DYY}, and the existence of steady states is not implied by~\cite{Bed}. In other words, tightness can be guaranteed without having/knowing the existence and uniqueness of steady states. This is indeed also the case in the 1D result~\cite{S1D}, but was not pointed out explicitly there.

\item The assumption {\bf (A5)} guarantees that at least a positive amount of mass always stays near the origin, see Lemma \ref{lem_crho} (similar to the assumption {\bf (A6)} in~\cite{S1D}). It looks different from the latter because the `everything runaway' situation is different in 2D and 1D: in 2D (or multi-D) radially-symmetric solutions, when mass are running away to infinity, it scatters into larger and larger annuli; while in 1D radially-symmetric solutions, it is separated into two bulks of mass moving left and right.
\end{itemize}

\subsection{The main result}\label{sec_main}

\begin{theorem}\label{thm1}
Let {\bf (A1)}-{\bf (A5)} be satisfied. Then the solution $\rho(t,\bx)$ to \eqref{eq0} has log-moment uniformly bounded in time:
\begin{equation}\label{thm1_1}
\int \rho(t,\bx) \ln(1+|\bx|)\rd{\bx} \le C,\quad \forall t\ge 0.
\end{equation}
\end{theorem}

This theorem gives the tightness \eqref{tight} of the solution $\rho(t,\bx)$. We first give the proof of this theorem by assuming the intermediate results (namely, Theorem \ref{thm_erc} and Theorem \ref{thm_finite}), and later turn to the discussion on them. 

We will frequently use the variables $\bx,\by\in\mathbb{R}^2$, and their polar coordinates are always denoted as
\begin{equation}\label{polar}
\bx = (r\cos\varphi,r\sin\varphi)^T,\quad \by = (s\cos\theta,s\sin\theta)^T.
\end{equation}

\begin{proof}[Proof of Theorem \ref{thm1}]

{\bf STEP 1}: the test function $\phi$.

Fix $\cR_1>0$ large enough\footnote{The curly $\cR_1$ should be distinguished from the $R_j$ appeared later in \eqref{Ij}.}. Define the radial test function $\phi(\bx)=\phi(r)$ by
\begin{equation}\label{phi}
\phi(\bx) = \chi_{7\cR_1\le |\bx|\le 8\cR_1}(\bx) * \frac{1}{2\pi}\ln |\bx| ,
\end{equation}
which satisfies 
\begin{equation}
\Delta \phi = \chi_{7\cR_1\le |\bx|\le 8\cR_1}.
\end{equation}
See Figure \ref{fig_pic10} for illustration.

\begin{figure}
\begin{center}
  \includegraphics[width=.6\linewidth]{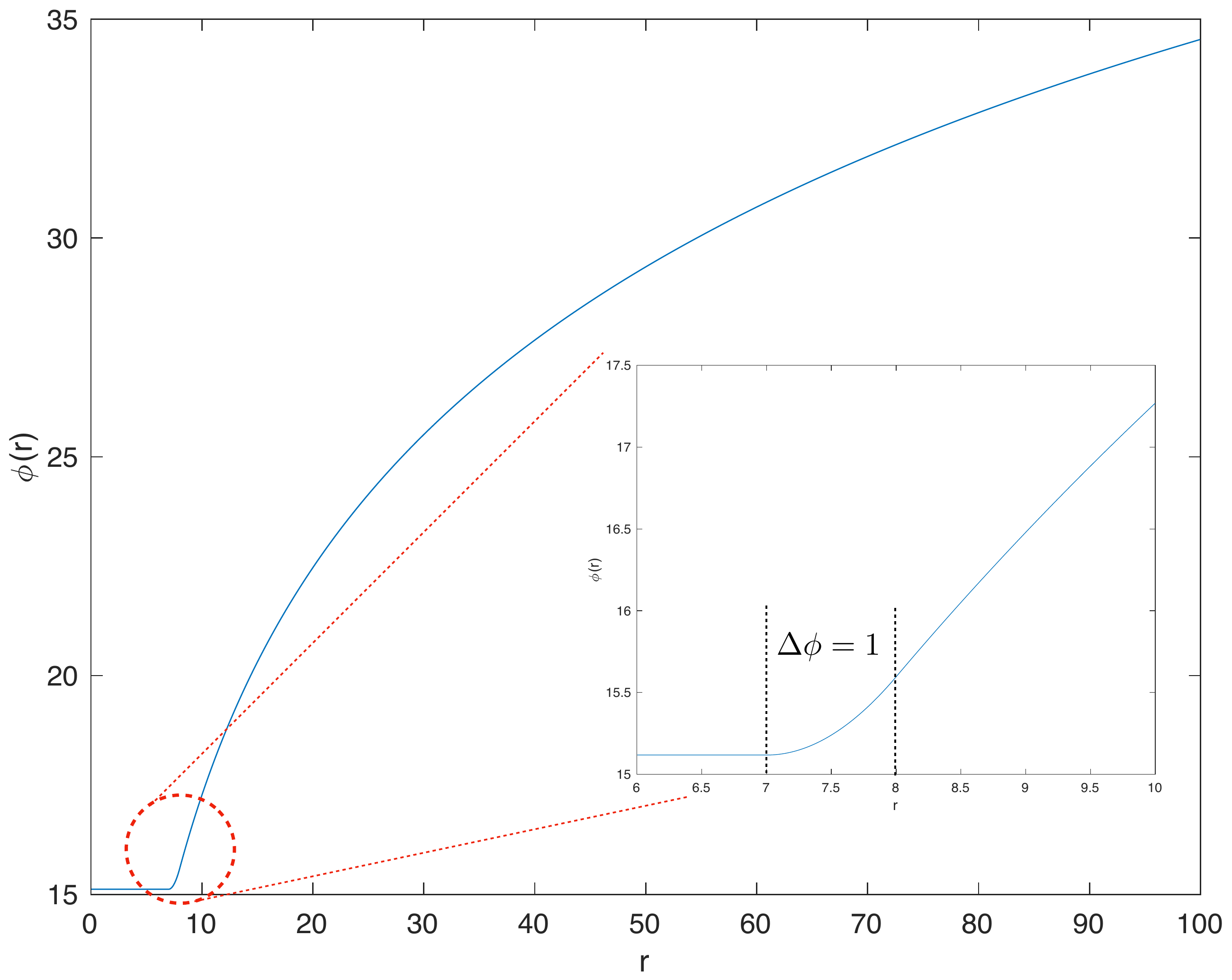}  
  \caption{The test function $\phi$ as defined in \eqref{phi}, in the radial variable. In the figure, $\cR_1=1$ is taken. Compared to \cite{S1D}, here we take the interval $[7\cR_1,8\cR_1]$ instead of $[5\cR_1,6\cR_1]$ for minor technical reasons.}
\label{fig_pic10}
\end{center}
\end{figure}

We first show that the uniform-in-time bound of 
\begin{equation}\label{mphi}
\km_\phi(t) := \int \phi(\bx)\rho(t,\bx) \rd{\bx},
\end{equation}
implies \eqref{thm1_1}.  We rewrite $\phi$ as
\begin{equation}\begin{split}
\phi(\bx) = &  \frac{1}{2\pi}\int_{7\cR_1\le |\by|\le 8\cR_1}\ln|\bx-\by|\rd{\by} 
=  \frac{1}{2\pi}\int_{7\cR_1}^{8\cR_1}\int_{-\pi}^{\pi}\ln|\bx-\by| \rd{\theta}\cdot s\rd{s}, \\
\end{split}\end{equation}
where we used the polar coordinates \eqref{polar}. Since $\phi$ is radially-symmetric, we may take $\bx = (r,0)^T,\,r>0$.

We claim that (viewing $\phi$ as a function of $r$)
\begin{equation}\label{claim0}
\phi'(r) \ge 0,\,\forall r>0;\quad \phi'(r) = 0,\,\forall 0<r\le 7\cR_1.
\end{equation}
To see this, for any $s\ge 7\cR_1$, 
\begin{equation}\begin{split}
\partial_r\int_{-\pi}^{\pi}\ln|\bx-\by| \rd{\theta} = & \frac{1}{2}\partial_r\int_{-\pi}^{\pi} \ln (r^2+s^2-2rs\cos\theta)\rd{\theta} 
=  \int_{-\pi}^{\pi} \frac{r-s\cos\theta}{r^2+s^2-2rs\cos\theta}\rd{\theta}.
\end{split}\end{equation}
For $r>s$, the RHS is clearly positive. For $r<s$, using the fact that $\int_{-\pi}^{\pi}\ln|\bx-\by| \rd{\theta}$ is harmonic and radially-symmetric in $\{\bx: |\bx|<s\}$, it is clearly constant in $\bx$. Combining the two cases and integrating in $7\cR_1\le s\le 8\cR_1$ gives \eqref{claim0}. 

Therefore we have the lower bound
\begin{equation}
\phi(\bx) \ge \phi(0),\quad \forall \bx.
\end{equation}
Also, it is clear that $\phi(\bx)$ behaves like $C\ln|\bx|$ as $|\bx|\rightarrow \infty$. Therefore it follows that the uniform-in-time bound of \eqref{mphi} implies \eqref{thm1_1}.

Next we aim to show the uniform-in-time bound of \eqref{mphi}. The time evolution of $\km_\phi$ is given by
\begin{equation}\begin{split}
\frac{\rd}{\rd{t}}\km_\phi = & -\int \phi(\bx)\nabla\cdot(\rho(\bx) \bu(\bx))\rd{\bx} + \int \phi(\bx)\Delta(\rho(\bx)^m) \rd{\bx} \\
= & \int \nabla\phi(\bx)\cdot\bu(\bx)\rho(\bx) \rd{\bx} + \int \Delta\phi(\bx)\rho(\bx)^m \rd{\bx}, \\
\end{split}\end{equation}
where the $t$-dependence is omitted. Then the two terms on the RHS are treated separately.

{\bf STEP 2}: the contribution from the interaction term. 

Take $\cR_1$ large, so that Theorem \ref{thm_erc} implies that \eqref{deferc} with $\cR_1$ replaced by $8\cR_1$ holds for any $r>s\ge \cR$ (where $F[\phi,W](r,s)$ is  defined in \eqref{ercF}, and $\cR$ is a large number given by Theorem \ref{thm_erc}, depending on $\cR_1$). Then
\begin{equation}\label{phiinter}\begin{split}
& \int \nabla\phi(\bx)\cdot\bu(\bx)\rho(\bx) \rd{\bx} \\
= & -\int\int \nabla\phi(\bx)\cdot\nabla W(\bx-\by)\rho(\by)\rd{\by}\rho(\bx) \rd{\bx} \\
= & -\frac{1}{2}\int\int (\nabla\phi(\bx)-\nabla\phi(\by))\cdot\nabla W(\bx-\by)\rho(\by)\rd{\by}\rho(\bx) \\
%= & -\frac{1}{2}\iint_{[0,\infty)\times[0,\infty)}\int_{-\pi}^\pi \int_{-\pi}^\pi  (\nabla\phi(\bx)-\nabla\phi(\by))\cdot\nabla W(\bx-\by)\rd{\theta}\rd{\varphi}\rho(s)s\rd{s}\rho(r)r\rd{r} \\
= & -\frac{1}{2}\iint_{(r,s)\in[0,7\cR_1]\times[0,7\cR_1]} (\nabla\phi(\bx)-\nabla\phi(\by))\cdot\nabla W(\bx-\by)\rho(\by)\rd{\by}\rho(\bx)\rd{\bx} \\
& -\frac{1}{2}\iint_{(r,s)\in[7\cR_1,\cR]\times[7\cR_1,\cR]} (\nabla\phi(\bx)-\nabla\phi(\by))\cdot\nabla W(\bx-\by)\rho(\by)\rd{\by}\rho(\bx)\rd{\bx} \\
& -\iint_{(r,s)\in[7\cR_1,\cR]\times[\cR,\infty)} (\nabla\phi(\bx)-\nabla\phi(\by))\cdot\nabla W(\bx-\by)\rho(\by)\rd{\by}\rho(\bx)\rd{\bx} \\
& -\iint_{(r,s)\in[7\cR_1,\infty)\times[0,7\cR_1]} (\nabla\phi(\bx)-\nabla\phi(\by))\cdot\nabla W(\bx-\by)\rho(\by)\rd{\by}\rho(\bx)\rd{\bx} \\
& -\frac{1}{2}\iint_{(r,s)\in[\cR,\infty)\times[\cR,\infty)} (\nabla\phi(\bx)-\nabla\phi(\by))\cdot\nabla W(\bx-\by)\rho(\by)\rd{\by}\rho(\bx)\rd{\bx}\\
=: & I_1+I_2+I_3+I_4+I_5,
\end{split}\end{equation}
where we used  the symmetry between $\bx$ and $\by$ in the second and third equalities.

We estimate $I_1,I_2,I_3,I_4,I_5$ as follows:
\begin{itemize} 

\item By \eqref{claim0}, $I_1=0$. Also,
\begin{equation}\begin{split}
I_4 = & -\iint_{(r,s)\in[7\cR_1,\infty)\times[0,7\cR_1]}\ \nabla\phi(\bx)\cdot\nabla W(\bx-\by)\rho(\by)\rd{\by}\rho(\bx)\rd{\bx} \le 0,\\
\end{split}\end{equation}
since
\begin{equation}
\nabla\phi(\bx)\cdot\nabla W(\bx-\by) = \frac{\phi'(r)W'(|\bx-\by|)}{r\cdot|\bx-\by|} \bx\cdot(\bx-\by)\ge \frac{\phi'(r)W'(|\bx-\by|)}{r\cdot|\bx-\by|} (r^2-rs) \ge 0,
\end{equation}
for any $r>s$, using $\phi'(r)\ge 0$ from \eqref{claim0}.

\item To estimate $I_2$ and $I_3$, we consider two radially-symmetric sets $A$ and $B$ inside $\{|\bx|\ge 7\cR_1\}$. By Lemma \ref{lem_circle} (which also gives the fact that $\|\nabla\phi\|_{L^\infty}<\infty$ when replacing $W$ by $\ln r$), we have
\begin{equation}\begin{split}
& \left|\int_B \nabla\phi(\bx)\cdot\nabla W(\bx-\by)\rho(\by)\rd{\by}\right| \le  C \left|\int_B \nabla W(\bx-\by)\rho(\by)\rd{\by}\right|\\
= &  C \left|\int \nabla W(\bx-\by)\int_{\tilde{B}} \rho(u)\delta(|\by|-u)\rd{u}\rd{\by}\right|
\le   C \int_{\tilde{B}}\rho(u)\rd{u}\\
= &  C \int_{\tilde{B}}\frac{1}{u}\rho(u)u\rd{u}
\le   C \int_B\rho(\by)\rd{\by},\\
\end{split}\end{equation}
where we used the notation $\tilde{B}$ to denote the radial range of $B$, and used the assumption $B\subset\{|\bx|\ge 7\cR_1\}$ in the last inequality. Then integrating in $\bx\in A$ and symmetrizing gives
\begin{equation}
\iint_{(r,s)\in A\times B} (\nabla\phi(\bx)-\nabla\phi(\by))\cdot\nabla W(\bx-\by)\rho(\by)\rd{\by}\rho(\bx)\rd{\bx} \le C\int_A\rho(\by)\rd{\by}\int_B\rho(\by)\rd{\by}.
\end{equation}
Applying this with $A=B=\{|\bx|\in[7\cR_1,\cR]\}$ and $A=\{|\bx|\in[7\cR_1,\cR]\},B=\{|\bx|\in[\cR,\infty)\}$ gives
\begin{equation}
|I_2|+|I_3| \le C\int_{r\in [7\cR_1,\cR]}\rho(\bx)\rd{\bx}.
\end{equation}

\item Finally, we rewrite
\begin{equation}\begin{split}
I_5 = & -\iint_{(r,s)\in[\cR,\infty)\times[\cR,\infty),\,r>s} (\nabla\phi(\bx)-\nabla\phi(\by))\cdot\nabla W(\bx-\by)\rho(\by)\rd{\by}\rho(\bx)\rd{\bx} \\
= & -2\pi \iint_{(r,s)\in[\cR,\infty)\times[\cR,\infty),\,r>s} F[\phi,W](r,s)s\rho(s)\rd{s}r\rho(r)\rd{r}, \\
\end{split}\end{equation}
by using the polar coordinates as in \eqref{polar}, and eliminating the $\varphi$ variable by radial symmetry. Here $F[\phi,W](r,s)$ is defined in \eqref{ercF}. Then we have $I_5\le 0$ by Theorem \ref{thm_erc} since $\phi(\bx) = \int_{7\cR_1}^{8\cR_1} \phi_\epsilon(\bx)\rd{\epsilon}$, using the notation in \eqref{deferc}.
\end{itemize}

Therefore 
\begin{equation}
\int \nabla\phi(\bx)\cdot\bu(\bx)\rho(\bx) \rd{\bx} \le C\int_{r\in [7\cR_1,\cR]}\rho(\bx)\rd{\bx}.
\end{equation}

{\bf STEP 3}: the contribution from the diffusion term.

By the construction of $\phi$ as a convolution of $\chi_{7\cR_1\le |\bx| \le 8\cR_1}$ with the fundamental solution of the Laplacian, we have
\begin{equation}
\int \Delta\phi(\bx)\rho(\bx)^m \rd{\bx} = \int_{7\cR_1\le |\bx| \le 8\cR_1}\rho(\bx)^m \rd{\bx} \le \|\rho\|_{L^\infty_{t,\bx}}^{m-1}\int_{7\cR_1\le |\bx| \le 8\cR_1}\rho(\bx) \rd{\bx},
\end{equation}
using $m\ge 1$ and the uniform $L^\infty$ bound in Lemma \ref{lem_reg1}.

Combining STEP 1 and STEP 2 gives
\begin{equation}
\frac{\rd}{\rd{t}}\km_\phi \le C\int_{|\bx|\in [7\cR_1,\cR]}\rho(\bx)\rd{\bx}.
\end{equation}
Integrating in $t$ gives
\begin{equation}
\km_\phi(t) \le \km_\phi(0) + C\int_0^t \int_{|\bx|\in [7\cR_1,\cR]}\rho(t_1,\bx)\rd{\bx} \rd{t_1}.
\end{equation}
Then the uniform bound of $\km_\phi$ follows from Theorem \ref{thm_finite} for $\cR_1$ large enough, applied to a finite union of annuli which covers $[7\cR_1,\cR]$.

\end{proof}

\subsection{The essentially-radially-contractive (ERC) property}

We define $F[\phi,W](r,s)$ by 
\begin{equation}\label{ercF}
F[\phi,W](r,s) = \int_{-\pi}^\pi \Big(\nabla\phi((r,0)^T)-\nabla\phi((s\cos\theta,s\sin\theta)^T)\Big)\cdot\nabla W\Big((r,0)^T-(s\cos\theta,s\sin\theta)^T\Big)\rd{\theta},
\end{equation}
which is the change in the $\phi$-moment coming from the interaction between two circles with  radius $r$ and $s$. The arguments $\phi$ and $W$ may be omitted when it is clear from the context.

\begin{theorem}\label{thm_erc}
{\bf (A1)} implies that $W$ is \emph{essentially-radially-contractive (ERC)}, defined as the following property: for any $\cR_1>0$, there exists $\cR>\cR_1$, such that 
\begin{equation}\label{deferc}
F[\phi_\epsilon,W](r,s) \ge 0,\quad  \phi_\epsilon(\bx):=\delta(|\bx|-\epsilon) * \frac{1}{2\pi}\ln|\bx|,
\end{equation}
for all $0<\epsilon\le \cR_1$ and $r>s\ge \cR$.
\end{theorem}

The issue of ERC is the main difference between the situation in 1D and 2D: for 1D one has the ERC property for any attractive potential $W$, in the sense that \eqref{deferc} holds with $F$ and $\phi_\epsilon$ replaced by\footnote{In 1D the unit sphere consists of two points, corresponding to $\theta=0,\pi$ in \eqref{ercF}.} $F_{1D}[\phi,W](r,s) = \phi'(r-s)W'(r-s)+\phi'(r+s)W'(r+s),\,\phi_\epsilon(x) = \delta(|x|-\epsilon)* |x|$, where $|x|$ is the fundamental solution of the Laplacian in 1D. In fact, this comes from the convexity of $\phi_\epsilon$. However, in 2D, the ERC property may fail for some attractive potential $W$ due to the non-convexity of $\phi_\epsilon(\bx)$, and one needs the extra assumption {\bf (A1)} to guarantee the ERC property. To illustrate how {\bf (A1)} helps with the ERC property, we prove the following
\begin{proposition}\label{prop_meow}
Let $\alpha> 1$. For any $r>s>0$,
\begin{equation}
F\Big[\ln|\bx|, \frac{|\bx|^{-\alpha+1}}{-\alpha+1}\Big](r,s) \ge 0 \quad\Leftrightarrow\quad \alpha\le 3.
\end{equation}
\end{proposition}

To prove this proposition, we introduce the following notations:
\begin{equation}\label{notation}\begin{split}
& \beta=\frac{\alpha+1}{2} >1 ,\quad z = \frac{1}{2}\Big(\frac{r}{s}+\frac{s}{r}\Big)>1, \\
\end{split}\end{equation}
which will be used in this proof, as well as Section \ref{sec_erc}.

\begin{proof}

\begin{equation}
\nabla (\ln|\bx|) = \frac{\bx}{r^2},\quad \nabla \Big(\frac{|\bx|^{-\alpha+1}}{-\alpha+1}\Big) = r^{-(\alpha+1)}\bx.
\end{equation}
Thus (with the notations \eqref{notation})
\begin{equation}\label{Fsr}\begin{split}
F(r,s) = & \int_{-\pi}^\pi ((r-s\cos\theta)^2+(s\sin\theta)^2)^{-\beta} \Big(\frac{1}{r}-\frac{1}{s}\cos\theta,-\frac{1}{s}\sin\theta\Big)^T\cdot (r-s\cos\theta,-s\sin\theta)\rd{\theta} \\
= & s^{-2\beta} \int_{-\pi}^\pi \Big((\frac{r}{s}-\cos\theta)^2+(\sin\theta)^2\Big)^{-\beta} \Big(\frac{s}{r}-\cos\theta,-\sin\theta\Big)^T\cdot \Big(\frac{r}{s}-\cos\theta,-\sin\theta\Big)^T\rd{\theta} \\
= & s^{-2\beta} \int_{-\pi}^\pi \Big(\frac{r^2}{s^2}+1-2\frac{r}{s}\cos\theta\Big)^{-\beta} \Big(2-(\frac{r}{s}+\frac{s}{r})\cos\theta\Big)\rd{\theta} \\
= & 2^{-\beta+1}(rs)^{-\beta} f(z),
\end{split}\end{equation}
where
\begin{equation}\label{fz1}
f(z):= \int_{-\pi}^\pi (z-\cos\theta)^{-\beta}(1-z\cos\theta)\rd{\theta},\quad z>1.
\end{equation}

To analyze the positivity of $f(z)$, we use integration by parts to compute (for any $\beta>1$)
\begin{equation}\begin{split}
 \int_{-\pi}^\pi (z-\cos\theta)^{-\beta}\sin^2\theta \rd{\theta}  
= & -\frac{1}{\beta-1}\int_{-\pi}^\pi \sin\theta \rd{(z-\cos\theta)^{-\beta+1}}  \\
= & \frac{1}{\beta-1}\int_{-\pi}^\pi (z-\cos\theta)^{-\beta+1}\cos\theta\rd{\theta} \\
= & \frac{1}{\beta-1}\int_{-\pi}^\pi (z-\cos\theta)^{-\beta}(z\cos\theta-1+\sin^2\theta)\rd{\theta} \\
= & \frac{1}{\beta-1}\Big(-f(z)+\int_{-\pi}^\pi (z-\cos\theta)^{-\beta}\sin^2\theta\rd{\theta}\Big). \\
\end{split}\end{equation}
Therefore
\begin{equation}\label{fz2}
f(z) = (2-\beta)\int_{-\pi}^\pi (z-\cos\theta)^{-\beta}\sin^2\theta \rd{\theta},
\end{equation}
where the last integrand is non-negative. The conclusion follows by noticing that $\beta\le 2\Leftrightarrow \alpha\le 3$.

\end{proof}

Despite its simplicity, this proposition enables us to prove Theorem \ref{thm_erc} by using comparison argument for $W$ with the power-law potential $r^{-\alpha+1}/(-\alpha+1)$, in view of \eqref{Winc} (a consequence of {\bf (A1)}). The difference between $\phi_\epsilon$ and $\ln|\bx|$ can be handled by perturbative arguments, if $\alpha$ is strictly less than 3.

\subsection{Finiteness of time integral in $\rho$}\label{sec_finite}

In this subsection we prove the following result, which was used in the last step of the proof of Theorem \ref{thm1}:
\begin{theorem}\label{thm_finite}
For $\cR_1>0$ large enough, the solution $\rho(t,\bx)$ to \eqref{eq0} satisfies
\begin{equation}
\int_0^\infty \int_{7\cR_1\le |\bx| \le 8\cR_1} \rho(t,\bx)\rd{\bx}\rd{t} < \infty.
\end{equation}
\end{theorem}

The proof of Theorem \ref{thm_finite} follows a similar approach as~\cite{S1D}: use the gradient flow structure of \eqref{eq0} and design curves of density distributions which decrease the total energy, c.f. Lemma \ref{lem_basic}. We introduce the $h$-representation (similar to~\cite{S1D}) of a radially-symmetric density distribution $\rho(\bx)$, viewed as a function of $r$:
%we first decompose any radially-symmetric density distribution $\rho$ into its radially-decreasing and non-radially-decreasing parts:
\begin{equation}\label{rhoh1}
\rho(r) = \int_0^\infty \sum_{j}\chi_{I_j(h)}(r) \rd{h},
\end{equation}
where the level sets of $\rho$ is decomposed as
\begin{equation}\label{rhoh2}
\{r: \rho(r)\ge h\} = \bigcup_j I_j(h),
\end{equation}
as a disjoint union of closed intervals\footnote{We will always assume that the union in \eqref{rhoh2} is a finite union, and any $I_j$ which is a single point is ignored. The general case can be treated by approximation arguments, which are omitted in this paper.}, with $I_0$ being the unique interval containing 0 (if there is such an interval). We write 
\begin{equation}\label{Ij}
I_j = [r_j,R_j],\quad R_j>r_j.
\end{equation}

In~\cite{S1D}, we used the decomposition 
\begin{equation}\label{rhoh3}
\rho(r) = \rho^*(r) + \mu(r),\quad \rho^*(r) = \int_0^\infty \chi_{I_0(h)}(r) \rd{h},
\end{equation}
as its radially-decreasing and non-radially-decreasing parts, and used the continuous Steiner symmetrization (CSS) curves and the local clustering curve to control $\rho^*$ and $\mu$ respectively. However, in 2D, the CSS curves are not effective in decreasing the interaction energy if $r_j<<R_j$, see Figure \ref{fig_pic47} (left). Therefore, for fixed\footnote{The $\varepsilon$ dependence will be suppressed when unnecessary. Also, $\varepsilon$ should be distinguished from any $\epsilon$ used in this paper.} $0<\varepsilon< 1/4$, we need to further decompose $\mu$ and write
\begin{equation}\label{rhodecomp}
\rho(r) = \rho^*(r) + \mu_\sharp(r) + \mu_\flat(r),
\end{equation}
where the \emph{sharp and flat parts} of $\mu$ are defined as
\begin{equation}
\mu_\sharp(r) = \int_0^\infty \sum_{j\ge 1, \text{sharp}}\chi_{I_j(h)}(r) \rd{h},\quad \mu_\flat(r) = \int_0^\infty \sum_{j\ge 1, \text{flat}}\chi_{I_j(h)}(r) \rd{h},
\end{equation}
where a radial interval $I_j(h),\,j\ge 1$ is called \emph{sharp} if $r_j\ge \varepsilon R_j$, and otherwise \emph{flat}. See Figure \ref{fig_pic47} (right) for illustration.

\begin{figure}
\begin{center}
  \includegraphics[width=.45\linewidth]{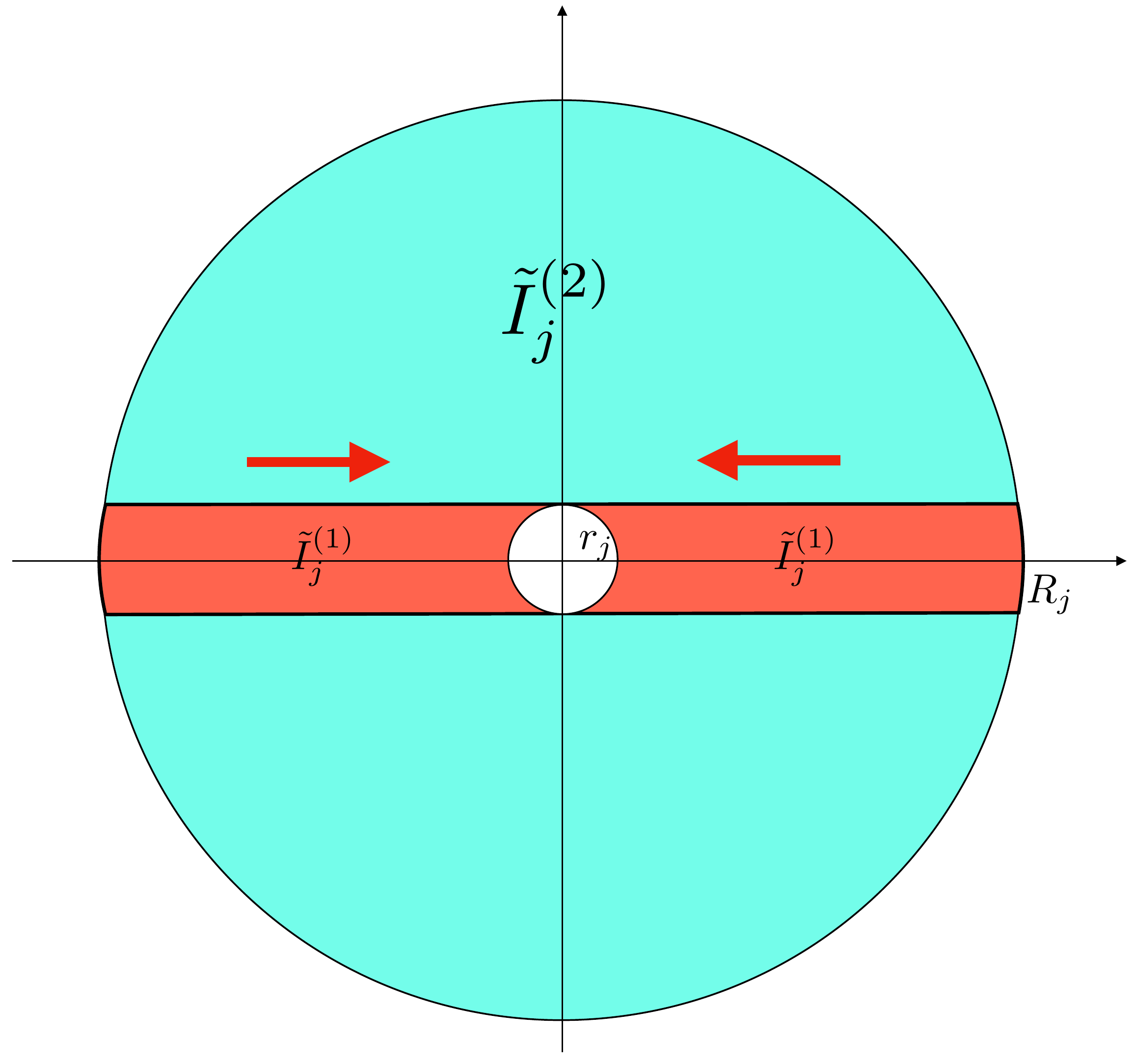}  
  \includegraphics[width=.45\linewidth]{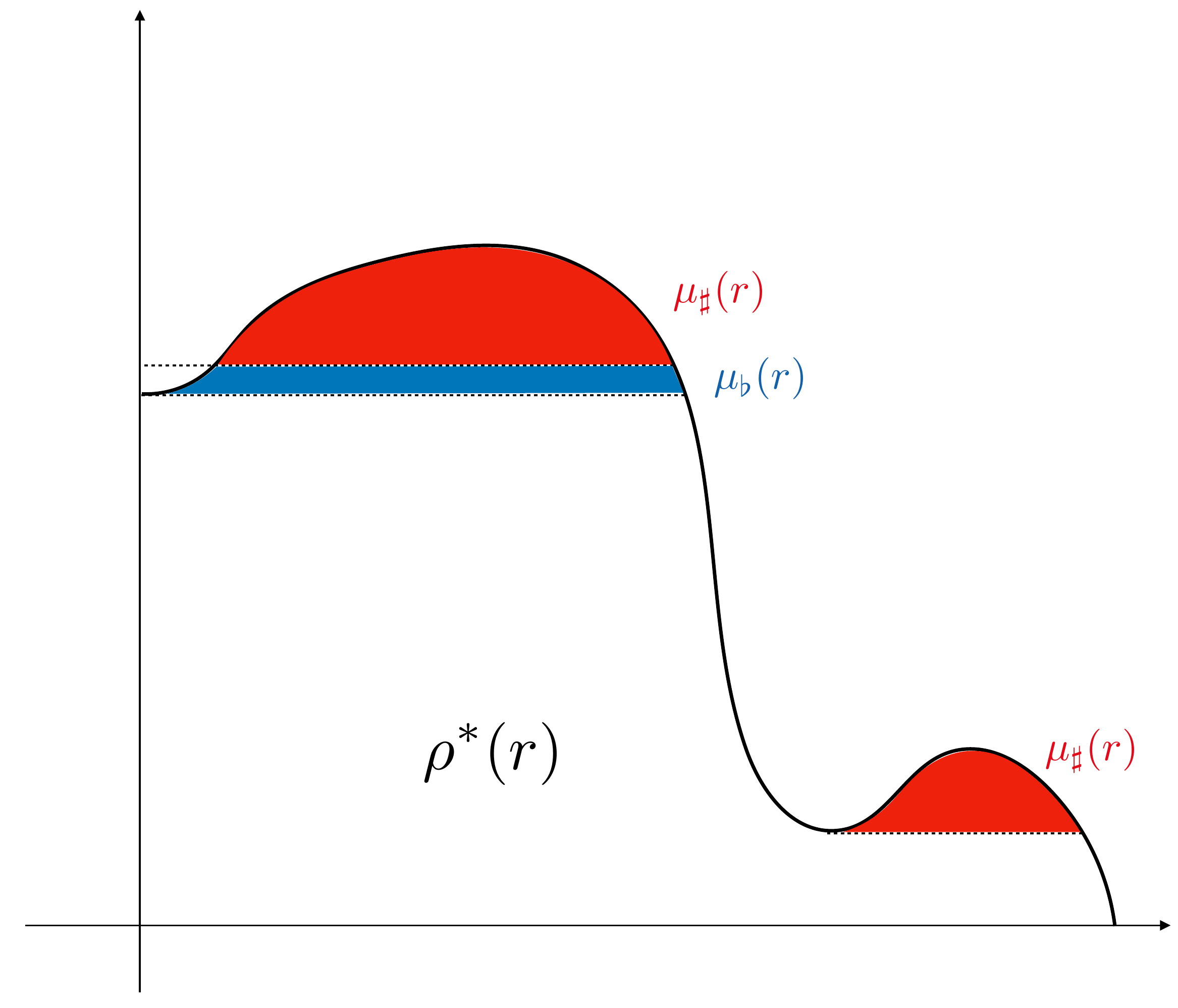}  
  \caption{Left: the CSS1 curve for a slice $\tI_j=\{\bx: r_j\le |\bx|\le R_j\}$. The red part is moving horizontally towards the center, and its area becomes small if $r_j<<R_j$. Right: the decomposition \eqref{rhodecomp}. The blue part $\mu_\flat$ contains those slices with $r_j<<R_j$.}
\label{fig_pic47}
\end{center}
\end{figure}

For the sharp part $\mu_\sharp$, the case $r_j<<R_j$ is ruled out, and we handle its time integral by using the CSS curves and obtain
\begin{proposition}\label{prop_finite1}
For $\cR_1>0$ large enough, we have
\begin{equation}
\int_0^\infty \int_{4\cR_1\le |\bx| \le \cR_2} \mu_\sharp(t,\bx)\rd{\bx}\rd{t} < \infty,
\end{equation}
for any $\cR_2 > 4\cR_1$.
\end{proposition}

For the flat part $\mu_\flat$, it shares the following property with $\rho^*$:
\begin{equation}\label{rhosmall}
\rho^*(r)+\mu_\flat(r) \le C\frac{1}{r^2},\quad \forall r>0,
\end{equation}
which is the key property in the energy analysis in the local clustering curve, providing the degeneracy of the internal energy at large $r$. This enable us to handle both $\rho^*$ and $\mu_\flat$ by the local clustering curve, and obtain
\begin{proposition}\label{prop_finite2}
For $\cR_1>0$ large enough, we have
\begin{equation}
\int_0^\infty \int_{7\cR_1\le |\bx| \le 8\cR_1} (\rho^*(t,\bx)+\mu_\sharp(t,\bx))\rd{\bx}\rd{t} < \infty.
\end{equation}
\end{proposition}

Combining Propositions \ref{prop_finite1} and \ref{prop_finite2} gives Theorem \ref{thm_finite}.

\section{Proof of Theorem \ref{thm_erc}: the ERC property}\label{sec_erc}

We start from some lemmas. Define
\begin{equation}\label{eta}
\eta(\epsilon):=\frac{1}{2\pi}\int_{-\pi}^\pi \frac{1-\epsilon\cos\theta}{(1-\epsilon\cos\theta)^2+(\epsilon\sin\theta)^2}\rd{\theta},
\end{equation}
for $|\epsilon|<1$. Then we have
\begin{lemma}
$\eta(\epsilon)$ is a smooth even function in $\epsilon$, satisfying the estimate
\begin{equation}\label{etaeps}
|\eta'(\epsilon)|\le C\epsilon,\quad |\eta(\epsilon)-1| \le C\epsilon^2,\quad \forall |\epsilon|\le \frac{1}{2}.
\end{equation}
Furthermore
\begin{equation}\label{etaeps2}
\eta(\epsilon) \le C,\quad \forall |\epsilon|<1.
\end{equation}
\end{lemma}

\begin{proof}
The smoothness of $\eta(\epsilon)$ follows from the fact that  the denominator $(1-\epsilon\cos\theta)^2+(\epsilon\sin\theta)^2$ is away from 0 near any fixed $|\epsilon|<1$. The even property of $\eta(\epsilon)$ follows from 
\begin{equation}
\eta(\epsilon)=\frac{1}{\pi}\int_{0}^\pi \frac{1-\epsilon\cos\theta}{(1-\epsilon\cos\theta)^2+(\epsilon\sin\theta)^2}\rd{\theta},
\end{equation}
and the change of variable $\theta\mapsto \pi-\theta$. Therefore $\eta'(0)=0$ and \eqref{etaeps} follows from $\eta(0)=1$.

To see \eqref{etaeps2}, for $\epsilon>\frac{1}{2}$,
\begin{equation}\begin{split}
\eta(\epsilon)= & \frac{1}{\pi}\int_{0}^\pi \frac{(1-\epsilon)+\epsilon(1-\cos\theta)}{[(1-\epsilon)+\epsilon(1-\cos\theta)]^2+(\epsilon\sin\theta)^2}\rd{\theta}
\le   C\int_{0}^\pi \frac{(1-\epsilon)+\epsilon\theta^2}{[(1-\epsilon)+\epsilon \theta^2]^2+\epsilon^2\theta^2}\rd{\theta}\\
%=  & C\int_{0}^\pi \frac{(1-\epsilon)+\epsilon\theta^2}{\epsilon^2\theta^4 + (\epsilon^2+2\epsilon(1-\epsilon))\theta^2 + (1-\epsilon)^2}\rd{\theta}\\
\le  & C\int_{0}^\pi \frac{(1-\epsilon)+\theta^2}{\theta^2 + (1-\epsilon)^2}\rd{\theta}
\le   C\int_{0}^\pi \frac{1-\epsilon}{\theta^2 + (1-\epsilon)^2}\rd{\theta} + C
=   C\tan^{-1}\frac{\theta}{1-\epsilon}\Big|_{\theta=0}^\pi + C
\le  C.
\end{split}\end{equation}
The uniform bound of $\eta(\epsilon)$ on $[0,\frac{1}{2}]$ follows from its smoothness.

\end{proof}

Then we  prove an estimate on an angular integral:
\begin{lemma}\label{lem_th}
Fix $1<\beta<2$. Then
\begin{equation}\label{lem_th}
\frac{z-1}{z}\int_{-\pi}^\pi (z-\cos\theta)^{-\beta}\rd{\theta} \le C\int_{-\pi}^\pi (z-\cos\theta)^{-\beta}(1-z\cos\theta)\rd{\theta},\quad \forall  z>1,
\end{equation}
where $C$ depends on $\beta$.
\end{lemma}
The main point of this lemma is the behavior near $z=1$: reformulating the RHS integral as $\int_{-\pi}^\pi (z-\cos\theta)^{-\beta}\sin^2\theta \rd{\theta}$ by \eqref{fz1} and \eqref{fz2}, the $\sin^2\theta$ factor in the integrand is degenerate near $\theta=0$, which is exactly the place where the factor $(z-\cos\theta)^{-\beta}$ is most singular. The $z-1$ factor on the LHS quantifies this degeneracy for $z$ near 1.

\begin{proof}

By \eqref{fz1} and \eqref{fz2}, the integral on the RHS of \eqref{lem_th} is
\begin{equation}\label{IR}\begin{split}
&  \int_{-\pi}^\pi (z-\cos\theta)^{-\beta}(1-z\cos\theta)\rd{\theta} \\
= & (2-\beta)\int_{-\pi}^\pi (z-\cos\theta)^{-\beta}\sin^2\theta \rd{\theta} \\
\ge  & (2-\beta)\int_{0}^{\pi/2} (z-\cos\theta)^{-\beta}\sin^2\theta \rd{\theta} + (2-\beta)\int_{\pi/2}^{3\pi/4} (z-\cos\theta)^{-\beta}\sin^2\theta \rd{\theta} \\
\ge  & (2-\beta)\int_{0}^{\pi/2} (z-\cos\theta)^{-\beta}\sin^2\theta \rd{\theta} + c(z+1)^{-\beta}. \\
\end{split}\end{equation}

We write  the LHS of \eqref{lem_th} as
\begin{equation}\begin{split}
 \frac{z-1}{z}\int_{-\pi}^\pi (z-\cos\theta)^{-\beta}\rd{\theta} = & \frac{2}{z}\int_{0}^\pi (z-\cos\theta)^{-\beta}(z-1)\rd{\theta} \\
= & \frac{2}{z}\int_{0}^{\pi/2} (z-\cos\theta)^{-\beta}(z-1)\rd{\theta} + \frac{2}{z}\int_{\pi/2}^\pi (z-\cos\theta)^{-\beta}(z-1)\rd{\theta} \\
\le & \frac{2}{z}\int_{0}^{\pi/2} (z-\cos\theta)^{-\beta}(z-1)\rd{\theta} + C(z-1)z^{-\beta-1} \\
\le & 2\int_{0}^{\pi/2} (z-\cos\theta)^{-\beta}(z-1)\rd{\theta} + C(z+1)^{-\beta}, \\
\end{split}\end{equation}
where  we used $z>1$. Therefore it suffices to prove
\begin{equation}
I_L:=\int_{0}^{\pi/2} (z-\cos\theta)^{-\beta}(z-1)\rd{\theta} \le C \int_{0}^{\pi/2} (z-\cos\theta)^{-\beta}\sin^2\theta \rd{\theta}=: CI_R.
\end{equation}

We cut the above integrals into $0\le \theta \le \theta_1$ and $\theta_1\le \theta \le\pi/2$, where $\theta_1\in[0,\pi/2]$ is determined by
\begin{equation}
\left\{\begin{split}
& \sin^2\theta_1 = z-1 ,\quad \text{if there exists such $\theta_1\in (\theta_0,\pi/2)$}; \\
& \pi/2 ,\quad \text{otherwise}.
\end{split}\right.
\end{equation}

For any $\theta$ with $\theta_1< \theta \le\pi/2$, it is clear that 
\begin{equation}
\sin^2\theta \ge z-1,
\end{equation}
and thus
\begin{equation}
\int_{\theta_1}^{\pi/2} (z-\cos\theta)^{-\beta}(z-1)\rd{\theta} \le \int_{\theta_1}^{\pi/2} (z-\cos\theta)^{-\beta}\sin^2\theta\rd{\theta} .
\end{equation}

By the definition of $\theta_1$ we have
\begin{equation}\label{theta1}
\min\{\sqrt{z-1},\frac{\pi}{2}\} \le \theta_1 \le \frac{\pi}{2}\sqrt{z-1}.
\end{equation}
Therefore
\begin{equation}\label{theta01_1}
\int_{0}^{\theta_1} (z-\cos\theta)^{-\beta}(z-1)\rd{\theta} \le \int_{0}^{\theta_1} (z-1)^{-\beta}(z-1)\rd{\theta} \le C(z-1)^{1-\beta}\theta_1,
\end{equation}
and
\begin{equation}\label{theta01_2}\begin{split}
& \int_{0}^{\pi/2} (z-\cos\theta)^{-\beta}\sin^2\theta \rd{\theta} 
\ge  \int_{\theta_1/2}^{\theta_1} \Big(z-1+\frac{\theta^2}{2}\Big)^{-\beta}\Big(\frac{2}{\pi}\cdot \theta\Big)^2 \rd{\theta} \\
\ge & \int_{\theta_1/2}^{\theta_1} \Big(z-1+\frac{\theta_1^2}{2}\Big)^{-\beta}\Big(\frac{2}{\pi}\cdot \frac{\theta_1}{2}\Big)^2 \rd{\theta} 
\ge  c(z-1+\theta_1^2)^{-\beta}\theta_1^3.
\end{split}\end{equation}

Now we compare the RHS of \eqref{theta01_1} and \eqref{theta01_2}. If $\theta_1=\pi/2$, then \eqref{theta1} gives $z-1 \ge 1$. Then
\begin{equation}
(z-1+\theta_1^2)^{-\beta}\theta_1^2 \ge c
,\quad (z-1)^{1-\beta} \le 1,
\end{equation}
which gives $I_L\le CI_R$. If $\theta_1<\pi/2$, then \eqref{theta1} gives $\sqrt{z-1}\le \theta_1\le \frac{\pi}{2}\sqrt{z-1}$. Therefore
\begin{equation}
(z-1+\theta_1^2)^{-\beta}\theta_1^2 \ge c(z-1)^{1-\beta},
\end{equation}
which also gives $I_L\le CI_R$.

\end{proof}

\begin{proof}[Proof of Theorem \ref{thm_erc}]

We first notice that \eqref{Winc} (a consequence of {\bf (A1)}) implies that $W'(r)r^{\alpha}$ is an increasing function in $r$. Similarly
\begin{equation}\label{Wassu1_2}
(W'(r)r^{-A})' = W''(r)r^{-A} -AW'(r)r^{-A-1} \le 0,
\end{equation}
and thus $W'(r)r^{-A}$ is a decreasing function in $r$.

{\bf STEP 1:} rescaling.

We will show that it suffices to prove the case when $\cR=1$ and $\cR_1$ sufficiently small. 

Assume this case has been proved. In the general case, we first notice the following rescaling:
\begin{equation}\begin{split}
& F[\phi_\epsilon,W](r,s) \\
= & \int_{-\pi}^\pi \Big(\nabla\phi_\epsilon((r,0)^T)-\nabla\phi_\epsilon((s\cos\theta,s\sin\theta)^T)\Big)\cdot\nabla W\Big((r,0)^T-(s\cos\theta,s\sin\theta)^T\Big)\rd{\theta} \\
 = & \int_{-\pi}^\pi \Big(\nabla\phi_\epsilon(\cR(\frac{r}{\cR},0)^T)-\nabla\phi_\epsilon(\cR(\frac{s}{\cR}\cos\theta,\frac{s}{\cR}\sin\theta)^T)\Big)\\
 & \cdot\nabla W\Big(\cR(\frac{r}{\cR},0)^T-\cR(\frac{s}{\cR}\cos\theta,\frac{s}{\cR}\sin\theta)^T\Big)\rd{\theta} \\
= & \frac{1}{\cR^2}F[\tilde{\phi},\tilde{W}](\frac{r}{\cR},\frac{s}{\cR}),
\end{split}\end{equation}
where
\begin{equation}
\tilde{W}(r) =W(\cR r),\quad \nabla \tilde{W}(\bx) = \cR \nabla W(\cR \bx),
\end{equation}
and
\begin{equation}
\tilde{\phi}(r) = \phi_\epsilon(\cR r),\quad \nabla \tilde{\phi}(\bx) = \cR \nabla \phi_\epsilon(\cR \bx).
\end{equation}
$\tilde{W}$ satisfies \eqref{Wassu1} (with the same constants $\alpha$ and $A$) by noticing that 
\begin{equation}
\frac{r\tilde{W}''(r)}{\tilde{W}'(r)} = \frac{\cR r W''(\cR r)}{W'(\cR r)}.
\end{equation}
$\tilde{\phi}$ can be written as
\begin{equation}\begin{split}
\tilde{\phi}(r) = & \frac{1}{2\pi}\int \delta(|\cR\bx-\by|- \epsilon)\ln |\by|\rd{\by}\\ 
= & \frac{1}{2\pi}\int \delta\Big(\cR(|\bx-\frac{\by}{\cR}|- \frac{\epsilon}{\cR})\Big)(\ln |\frac{\by}{\cR}|+\ln\cR)\rd{\by} \\
= & \frac{1}{2\pi \cR}\int \delta\Big(|\bx-\frac{\by}{\cR}|- \frac{\epsilon}{\cR}\Big)(\ln |\frac{\by}{\cR}|+\ln\cR)\rd{\by} \\
= & \frac{\cR}{2\pi}\int \delta\Big(|\bx-\by|- \frac{\epsilon}{\cR}\Big)(\ln |\by|+\ln\cR)\rd{\by} \\
= & \cR \phi_{\epsilon/\cR}(\bx) + \frac{\cR\ln\cR}{2\pi}\int \delta\Big(|\by|- \frac{\epsilon}{\cR}\Big)\rd{\by}, \\
\end{split}\end{equation}
where the last term is independent of $\bx$, and therefore $\nabla\tilde{\phi} = \cR\nabla \phi_{\epsilon/\cR}$. Then
\begin{equation}
F[\tilde{\phi},\tilde{W}](\frac{r}{\cR},\frac{s}{\cR}) = \cR F[\phi_{\epsilon/\cR},\tilde{W}](\frac{r}{\cR},\frac{s}{\cR}) \ge 0,\quad \forall r>s\ge \cR,
\end{equation}
by the assumed case, if $\cR$ is large enough so that $\epsilon/\cR \le \cR_1/\cR$ is small enough.

{\bf STEP 2}: expand $F(r,s)$. 

In the rest of the proof, we assume $\cR=1$ and $\epsilon \le \cR_1$ sufficiently small. We will omit the subscript of $\phi_\epsilon$ and write it as $\phi$.

Using $\nabla W(\bx) = W'(|\bx|)\frac{\bx}{|\bx|}$, we write $F(r,s)$ as
\begin{equation}\begin{split}
F(r,s) = & \int_{-\pi}^\pi \frac{W'(d_\theta)}{d_\theta} \Big(\nabla\phi((r,0)^T)-\nabla\phi((s\cos\theta,s\sin\theta)^T)\Big)\cdot\Big((r,0)^T-(s\cos\theta,s\sin\theta)^T\Big)\rd{\theta}, \\
\end{split}\end{equation}
where $d_\theta$ is defined as (recall the definition of $z$ in \eqref{notation})
\begin{equation}\label{notation1}\begin{split}
& d_\theta  = \Big|(r,0)^T-(s\cos\theta,s\sin\theta)^T\Big| = \sqrt{r^2+s^2-2rs\cos\theta},\quad d_\theta^2 = 2rs(z-\cos\theta) .\\
\end{split}\end{equation}

The radially-symmetric function $\phi$ can be written as
\begin{equation}\label{phieps}
\phi(r) = \frac{1}{2\pi}\int_{-\pi}^\pi \ln\Big|(r,0)^T-\epsilon(\cos\theta,\sin\theta)^T\Big|\rd{\theta}= \frac{1}{4\pi}\int_{-\pi}^\pi \ln\Big((r-\epsilon\cos\theta)^2+(\epsilon\sin\theta)^2\Big)\rd{\theta}.
\end{equation}
Then
\begin{equation}
\nabla \phi(\bx) = \frac{\bx}{r}\cdot \frac{1}{2\pi}\int_{-\pi}^\pi \frac{r-\epsilon\cos\theta}{(r-\epsilon\cos\theta)^2+(\epsilon\sin\theta)^2}\rd{\theta} = \frac{\bx}{r^2}\eta(\frac{\epsilon}{r}),
\end{equation}
where $\eta(\epsilon)$ is defined in \eqref{eta}. Then  we write
\begin{equation}\label{pic1eq}\begin{split}
& \Big(\nabla\phi((r,0)^T)-\nabla\phi((s\cos\theta,s\sin\theta)^T)\Big)\cdot\Big((r,0)^T-(s\cos\theta,s\sin\theta)^T\Big) \\
= & \Big(\eta(\frac{\epsilon}{r})(\frac{1}{r},0)^T-\eta(\frac{\epsilon}{s})\frac{1}{s}(\cos\theta,\sin\theta)^T\Big)\cdot (r-s\cos\theta,-s\sin\theta)^T \\
= & 2-\Big(\frac{r}{s}+\frac{s}{r}\Big)\cos\theta + \frac{1}{r}\Big(\eta(\frac{\epsilon}{r})-1\Big)(r-s\cos\theta) + \frac{1}{s}\Big(\eta(\frac{\epsilon}{s})-1\Big)(s-r\cos\theta).
\end{split}\end{equation}

Therefore  $F$ can be written  as
\begin{equation}\label{F1}\begin{split}
F(r,s) = &  2\int_{-\pi}^\pi \frac{W'(d_\theta)}{d_\theta} \cdot (1-z\cos\theta)\rd{\theta} \\
& + \int_{-\pi}^\pi \frac{W'(d_\theta)}{d_\theta}  \cdot \Big[ \frac{1}{r}\Big(\eta(\frac{\epsilon}{r})-1\Big)(r-s\cos\theta) + \frac{1}{s}\Big(\eta(\frac{\epsilon}{s})-1\Big)(s-r\cos\theta)\Big]\rd{\theta} \\
= & 2M + I,
\end{split}\end{equation}
where $z$ is defined in \eqref{notation}.

{\bf STEP 3}: the case $r\ge C_1s$, where $C_1$ is large, to be determined. In this case we have $z \ge r/(2s) \ge C_1/2$. 

By \eqref{notation1}, we further write
\begin{equation}\begin{split}
\frac{W'(d_\theta)}{d_\theta} = w(z-\cos\theta) ,\quad w(u) := \frac{W'(\sqrt{2rsu})}{\sqrt{2rsu}}.
\end{split}\end{equation}

 Notice that \eqref{Wassu1} implies
\begin{equation}\label{wzu1}\begin{split}
w'(u) = & \frac{W''(\sqrt{2rsu})}{\sqrt{2rsu}}\cdot \frac{\sqrt{2rs}}{2\sqrt{u}} - \frac{W'(\sqrt{2rsu})}{\sqrt{2rsu}}\cdot\frac{1}{2u} \\
\ge & -\alpha\frac{W'(\sqrt{2rsu})}{\sqrt{2rsu}\cdot\sqrt{2rsu}}\cdot \frac{\sqrt{2rs}}{2\sqrt{u}} - \frac{W'(\sqrt{2rsu})}{\sqrt{2rsu}}\cdot\frac{1}{2u}, \\
= & -\beta\frac{w(u)}{u}
\end{split}\end{equation}
where $\beta$ is defined in \eqref{notation}, and similarly
\begin{equation}\begin{split}
w'(u)  \le & \frac{A-1}{2}\cdot \frac{w(u)}{u}.
\end{split}\end{equation}
Integrating these differential inequalities in $u$, we see that  for any $u>0$,
\begin{equation}
w(z)\exp(-\beta\cdot\frac{1}{z}\cdot u) \le w(z+u) \le w(z)\exp(\frac{A-1}{2}\cdot\frac{1}{z}\cdot u).
\end{equation}
Therefore we proved the case $u\ge 0$ of
\begin{equation}\label{wzu}
\Big|\frac{w(z+u)}{w(z)} - 1\Big| \le \exp(\frac{A}{z}u)-1 \le \frac{2A}{z},\quad \forall |u|\le 1,
\end{equation}
by assuming $A>\alpha$ without loss of generality, and $z\ge C_1/2$ large enough. The case $u<0$ can be proved similarly.

{\bf STEP 3-1}: estimate $M$.

Using \eqref{wzu1} and \eqref{wzu}, the term $M$ in \eqref{F1} is estimated by
\begin{equation}\label{estM0}\begin{split}
M = & \int_{-\pi}^\pi w(z-\cos\theta)(1-z\cos\theta)\rd{\theta} \\
= & \int_{-\pi}^\pi w(z-\cos\theta)\rd{\theta} - z\int_{-\pi}^\pi w(z-\cos\theta)\cos\theta\rd{\theta} \\
= & \int_{-\pi}^\pi w(z-\cos\theta)\rd{\theta} - z\int_{0}^\pi (w(z-\cos\theta)-w(z+\cos\theta))\cos\theta\rd{\theta} \\
= & \int_{-\pi}^\pi w(z-\cos\theta)\rd{\theta} - 2z\int_{0}^{\pi/2} (w(z-\cos\theta)-w(z+\cos\theta))\cos\theta\rd{\theta} \\
= & \int_{-\pi}^\pi w(z-\cos\theta)\rd{\theta} + 2z\int_{0}^{\pi/2} \int_{-\cos\theta}^{\cos\theta}w'(z+u)\rd{u}\cos\theta\rd{\theta} \\
\ge & \int_{-\pi}^\pi w(z-\cos\theta)\rd{\theta} - 2\beta z\int_{0}^{\pi/2} \int_{-\cos\theta}^{\cos\theta}\frac{w(z+u)}{z+u}\rd{u}\cos\theta\rd{\theta} \\
\ge & 2\pi(1-\frac{2A}{z})w(z) - \pi \beta\cdot (1+\frac{1}{z-1})\cdot (1+\frac{2A}{z})w(z) \\
\ge & cw(z), \\
\end{split}\end{equation}
by taking $z\ge C_1/2$ large, since $\beta<2$. See Figure \ref{fig_pic1} for illustration. Now $C_1$ is chosen so that the above estimate holds.

\begin{figure}
\begin{center}
  \includegraphics[width=.5\linewidth]{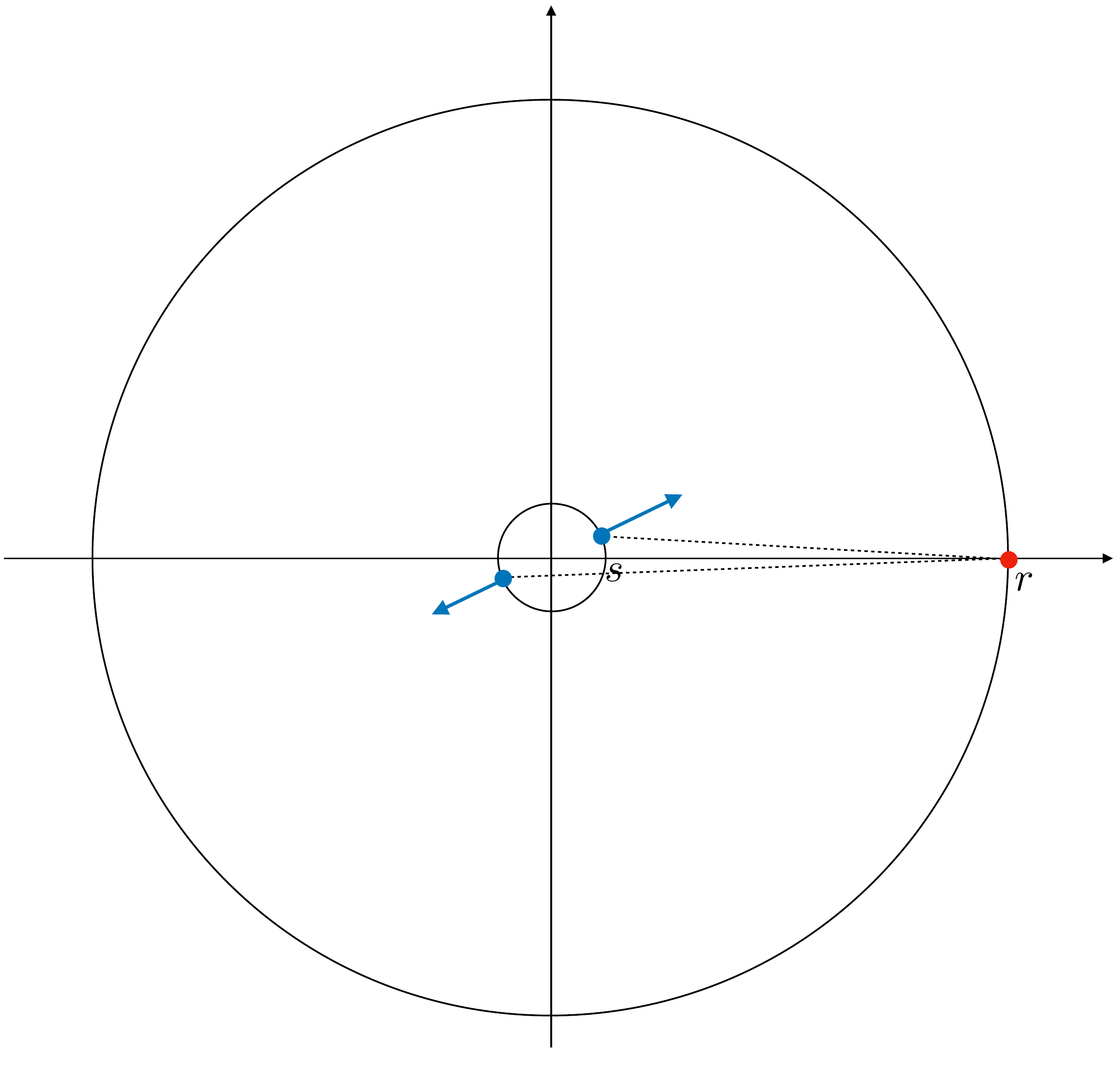}  
  \caption{The case $r\ge C_1s$. In the third equality of \eqref{estM0}, the cancellation between $\theta$ and $\theta+\pi$ is used, as shown in the two arrows in the picture. This originates from the term $\nabla\phi((s\cos\theta,s\sin\theta)^T)\cdot((r,0)^T-(s\cos\theta,s\sin\theta)^T) $ in \eqref{pic1eq}, which is the inner product between an arrow and its corresponding dashed segment.}
\label{fig_pic1}
\end{center}
\end{figure}

{\bf STEP 3-2}: estimate $I$.

To estimate the term $I$ in \eqref{F1}, we notice that
\begin{equation}\begin{split}
& \frac{1}{r}\Big(\eta(\frac{\epsilon}{r})-1\Big)(r-s\cos\theta) + \frac{1}{s}\Big(\eta(\frac{\epsilon}{s})-1\Big)(s-r\cos\theta) \\
= & -\frac{1}{s}\Big(\eta(\frac{\epsilon}{s})-1\Big)r\cos\theta +  \Big[\frac{1}{r}\Big(\eta(\frac{\epsilon}{r})-1\Big)(r-s\cos\theta) + \Big(\eta(\frac{\epsilon}{s})-1\Big) \Big].
\end{split}\end{equation}
By \eqref{etaeps}, the coefficient in front of $\cos\theta$ can be estimated by
\begin{equation}
\Big|\frac{1}{s}\Big(\eta(\frac{\epsilon}{s})-1\Big)r\Big| \le C(\frac{\epsilon}{s})^2 \frac{r}{s} \le C\epsilon^2 z,
\end{equation}
since $s\ge 1$, and the rest can be estimated by
\begin{equation}
\Big|\frac{1}{r}\Big(\eta(\frac{\epsilon}{r})-1\Big)(r-s\cos\theta) + \Big(\eta(\frac{\epsilon}{s})-1\Big)\Big| \le C\epsilon^2,
\end{equation}
 since $r>s\ge 1$.

Notice that, first, as in the estimate of $M$,
\begin{equation}
\Big| \int_{-\pi}^\pi w(z-\cos\theta)  \cdot \cos\theta\rd{\theta}\Big| \le C\frac{w(z)}{z}.
\end{equation}
Next, by \eqref{wzu},
\begin{equation}
 \int_{-\pi}^\pi w(z-\cos\theta)\rd{\theta} \le Cw(z).
\end{equation}
Therefore 
\begin{equation}
|I| \le C\epsilon^2 w(z),
\end{equation}
and it can be absorbed by $M$ (see \eqref{estM0}) if $\epsilon$ is small enough. Therefore the proof of the case $r\ge C_1s$ is finished.

{\bf STEP 4}: the case $r< C_1s$. In this case $z = \frac{1}{2}(\frac{r}{s}+\frac{s}{r}) \le C_1 = C$ (since $C_1=C_1(A,\beta)$ is already chosen).

We further write the main term $M$ into two parts, by comparing the potential $W$ with the potential $\frac{r^{-\alpha+1}}{-\alpha+1}$. Define 
\begin{equation}\label{theta0}
\theta_0 = \cos^{-1} \frac{1}{z} \in (0,\frac{\pi}{2}),\quad a_0 = W'(d_{\theta_0})d_{\theta_0}^{\alpha}.
\end{equation}
See Figure \ref{fig_pic23} (left) for illustration. Then we write
\begin{equation}\begin{split}
 M = & a_0\int_{-\pi}^\pi  d_\theta ^{-2\beta} \cdot \Big(2-(\frac{r}{s}+\frac{s}{r})\cos\theta\Big)\rd{\theta} \\
& + \int_{-\pi}^\pi \Big(\frac{W'(d_\theta )}{d_\theta } - a_0 d_\theta ^{-2\beta}\Big) \cdot \Big(2-(\frac{r}{s}+\frac{s}{r})\cos\theta\Big)\rd{\theta} \\
= & 2(2rs)^{-\beta}a_0\int_{-\pi}^\pi  (z-\cos\theta)^{-\beta} \cdot (1-z\cos\theta)\rd{\theta} \\
& + 2\int_{-\pi}^\pi d_\theta^{-1}(W'(d_\theta ) - a_0 d_\theta ^{-\alpha}) \cdot (1-z\cos\theta)\rd{\theta} \\
=: & 2M_1 + 2M_2.
\end{split}\end{equation}

\begin{figure}
\begin{center}
  \includegraphics[width=.45\linewidth]{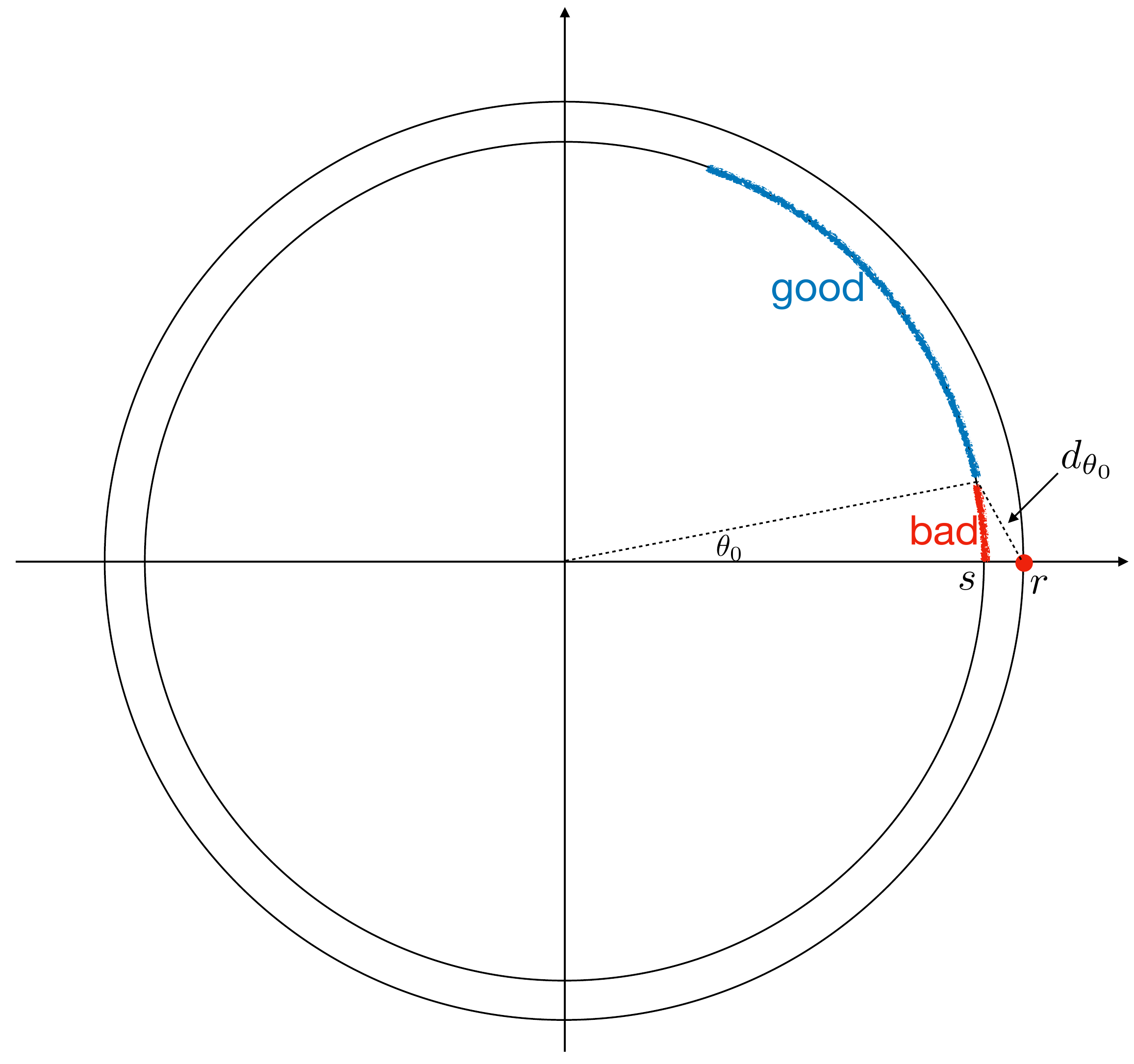}  
  \includegraphics[width=.45\linewidth]{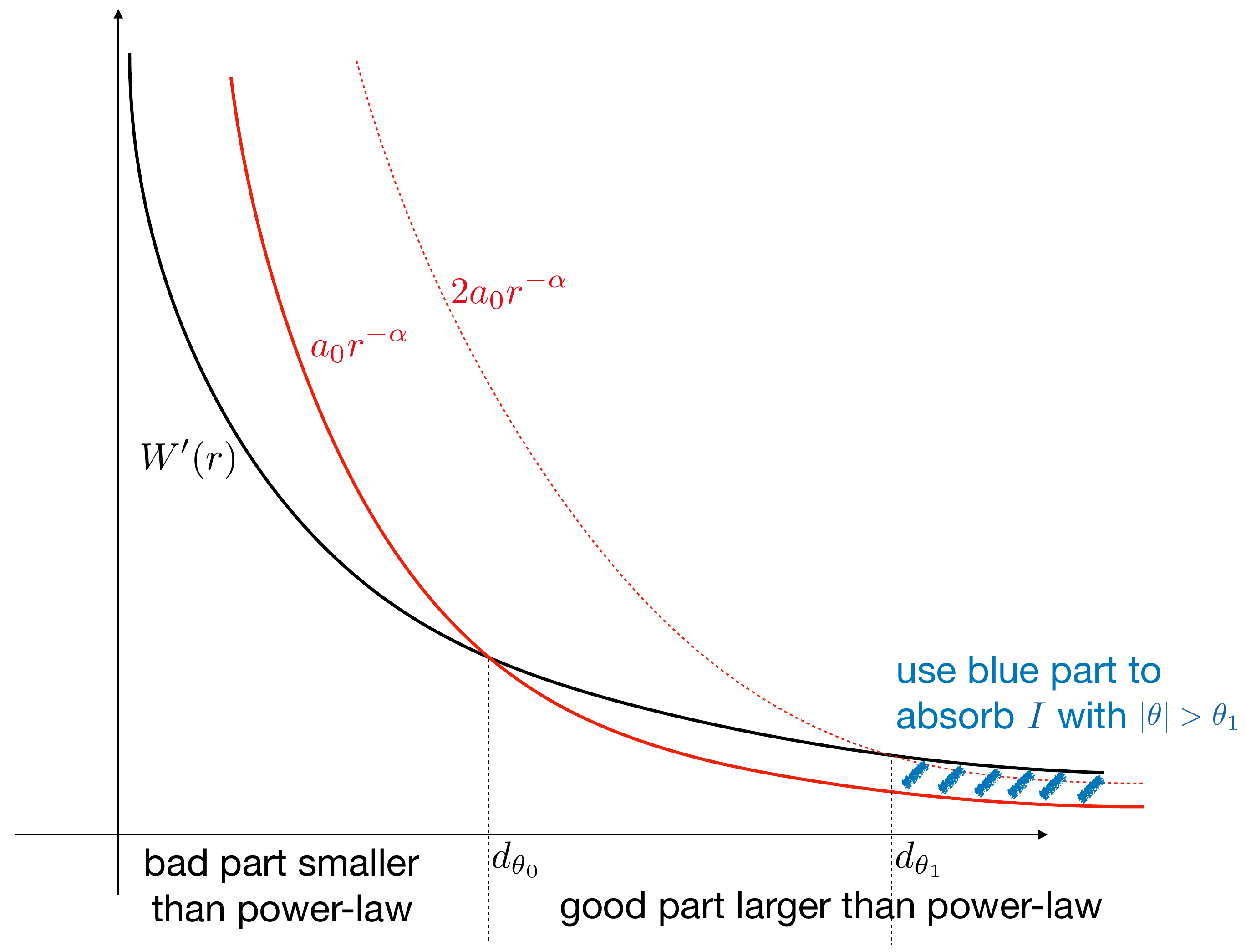}  
  \caption{Left: the definition of $\theta_0$ separates the sign of the integrand in $M$, as defined in \eqref{F1}.  Right: comparison with power-law interactions. For $r\le d_{\theta_0}$, the integrand in $M$ is negative (bad), and $W'(r)$ is smaller than the power-law $a_0r^{-\alpha}$. For $r\ge d_{\theta_0}$, the integrand in $M$ is positive (good), and $W'(r)$ is larger than the power-law $a_0r^{-\alpha}$. When $r$ is even larger ($r\ge d_{\theta_1}$), $W'(r)$ is at least twice as large as $a_0r^{-\alpha}$, and the extra part (blue region) can be used to control $I$.}
\label{fig_pic23}
\end{center}
\end{figure}

Lemma \ref{lem_th} implies $M_1\ge 0$ with a lower bound. Then we estimate the terms $M_2$ and $I$ separately.

{\bf STEP 4-1}: positivity of $M_2$.

By the definition of $a_0$, we have 
\begin{equation}
W'(d_{\theta_0})d_{\theta_0}^{\alpha} = a_0 .
\end{equation}
By \eqref{Winc} and the increasing property of the map $\theta\mapsto d_\theta $ for $\theta\in(0,\pi)$, we have
\begin{equation}
(W'(d_\theta ) - a_0 d_\theta ^{-\alpha})\cdot (\cos\theta_0-\cos\theta) \ge 0,\quad \forall \theta\in(-\pi,\pi).
\end{equation}
Notice that $\cos\theta_0-\cos\theta = \frac{1}{z}(1-z\cos\theta)$ by \eqref{theta0}, the definition of $\theta_0$. Therefore the integrand of $M_2$ is non-negative, which implies $M_2\ge 0$. See Figure \ref{fig_pic23} (right) for illustration. 

{\bf STEP 4-2}: estimate $I$.

Noticing that $s< r \le Cs$, we give the following estimate for the integrand of $I$:
\begin{equation}\label{I2term}\begin{split}
& \left|\frac{1}{r}\Big(\eta(\frac{\epsilon}{r})-1\Big)(r-s\cos\theta) + \frac{1}{s}\Big(\eta(\frac{\epsilon}{s})-1\Big)(s-r\cos\theta)\right| \\
= & \left|\frac{1}{r}\Big(\eta(\frac{\epsilon}{r})-1\Big)(r-s) + \frac{1}{s}\Big(\eta(\frac{\epsilon}{s})-1\Big)(s-r)\right| + \left|\Big[\frac{s}{r}\Big(\eta(\frac{\epsilon}{r})-1\Big) + \frac{r}{s}\Big(\eta(\frac{\epsilon}{s})-1\Big)\Big](1-\cos\theta)\right| \\
= & |r-s|\cdot\left|\Big(\frac{1}{r}-\frac{1}{s}\Big)\Big(\eta(\frac{\epsilon}{r})-1\Big) - \frac{1}{s}\Big(\eta(\frac{\epsilon}{s})-\eta(\frac{\epsilon}{r})\Big) \right| + \left|\Big[\frac{s}{r}\Big(\eta(\frac{\epsilon}{r})-1\Big) + \frac{r}{s}\Big(\eta(\frac{\epsilon}{s})-1\Big)\Big](1-\cos\theta)\right| \\
\le & |r-s| \cdot \Big|\frac{1}{r}-\frac{1}{s}\Big|\cdot \Big|\eta(\frac{\epsilon}{r})-1\Big| + |r-s| \cdot \Big| \frac{1}{s}\Big|\cdot\Big|\eta(\frac{\epsilon}{s})-\eta(\frac{\epsilon}{r})\Big| + C\epsilon^2(1-\cos\theta)\frac{1}{r^2} \\
\le & C\epsilon^2\Big( |r-s|^2\frac{1}{r^4}+(1-\cos\theta)\frac{1}{r^2}\Big),
\end{split}\end{equation}
where we used \eqref{etaeps} for small $\epsilon$.

Then notice that
\begin{equation}
z^2-1 = \frac{(r+s)^2}{4r^2s^2}(r-s)^2 \ge \frac{1}{4r^2}(r-s)^2,
\end{equation}
which implies
\begin{equation}\label{I2term2}
\frac{z-1}{z} \ge c\frac{1}{r^2}(r-s)^2,
\end{equation}
since $z\le C$.

We cut the $I$ integral into $|\theta|\le \theta_1$ and $\theta>\theta_1$, where $\theta_1>\theta_0$ is determined by
\begin{equation}\label{theta11}
\left\{\begin{split}
& W'(d_{\theta_1})d_{\theta_1}^{\alpha} = 2a_0 ,\quad \text{if there exists such $\theta_1\in (\theta_0,\pi)$}; \\
& \pi ,\quad \text{otherwise}.
\end{split}\right.
\end{equation}
In the first case in \eqref{theta11}, by \eqref{Wassu1_2}, we have
\begin{equation}
\Big(\frac{d_{\theta_1}}{d_{\theta_0}}\Big)^{\alpha} = 2\frac{W'(d_{\theta_0})}{W'(d_{\theta_1})} \ge 2\Big(\frac{d_{\theta_0}}{d_{\theta_1}}\Big)^{A},
\end{equation}
which implies
\begin{equation}
\frac{1-\frac{1}{z}\cos\theta_1}{1-\frac{1}{z}\cos\theta_0} = \frac{d_{\theta_1}}{d_{\theta_0}} \ge 1+c,\quad c = 2^{1/(A+\alpha)}-1>0.
\end{equation}
In the second case in \eqref{theta11}, we have
\begin{equation}
\frac{1-\frac{1}{z}\cos\theta_1}{1-\frac{1}{z}\cos\theta_0} =\frac{1+\frac{1}{z}}{1-\frac{1}{z^2}}   \ge 1+c,
\end{equation}
since $z\le C$. Therefore
\begin{equation}\label{theta1_2}\begin{split}
1-z\cos\theta_1 = & 1-z^2\Big(1-(1-\frac{1}{z}\cos\theta_1)\Big) \ge 1-z^2\Big(1-(1+c)(1-\frac{1}{z}\cos\theta_0)\Big) \\
= & 1-z^2\Big(1-(1+c)(1-\frac{1}{z^2})\Big) 
=  1-z^2(-c+\frac{1+c}{z^2}) = c(z^2-1) \ge c(z-1),
\end{split}\end{equation}
and the same holds when $\theta_1$ is replaced by any $\theta \in[\theta_1,\pi]$.

{\bf Case 1}: $|\theta|\le \theta_1$. In this case $W'(d_{\theta})d_{\theta}^{\alpha} \le 2a_0$.

\begin{equation}\label{I2}\begin{split}
& \int_{-\theta_1}^{\theta_1} \frac{W'(d_\theta )}{d_\theta }  \cdot \Big| \frac{1}{r}\Big(\eta(\frac{\epsilon}{r})-1\Big)(r-s\cos\theta) + \frac{1}{s}\Big(\eta(\frac{\epsilon}{s})-1\Big)(s-r\cos\theta)\Big|\rd{\theta} \\
\le  & Ca_0\epsilon^2\int_{-\pi}^{\pi} d_\theta ^{-2\beta}  \cdot  \Big(|r-s|^2\frac{1}{r^4} + (1-\cos\theta)\frac{1}{r^2}\Big)  \rd{\theta}  \\
= & Ca_0\epsilon^2 (rs)^{-\beta}\Big(\frac{1}{r^2}\cdot\frac{z-1}{z}\int_{-\pi}^\pi (z-\cos\theta)^{-\beta} \rd{\theta} + \frac{1}{r^2}\int_{-\pi}^\pi (z-\cos\theta)^{-\beta} (1-\cos\theta)\rd{\theta}\Big),
\end{split}\end{equation}
by \eqref{I2term} and \eqref{I2term2}. To handle the second integral above,
\begin{equation}\begin{split}
& \int_{-\pi}^\pi (z-\cos\theta)^{-\beta} (1-\cos\theta)\rd{\theta} \\
= & 2\int_{0}^{\pi/2} (z-\cos\theta)^{-\beta} 2\sin^2(\theta/2) \rd{\theta} + 2\int_{\pi/2}^\pi (z-\cos\theta)^{-\beta} 2\sin^2(\theta/2) \rd{\theta} \\
\le & 4\int_{0}^{\pi/2} (z-\cos\theta)^{-\beta} \sin^2\theta \rd{\theta} + C \\
\le & C\int_{-\pi}^\pi (z-\cos\theta)^{-\beta}(1-z\cos\theta)\rd{\theta} ,
\end{split}\end{equation}
where we used \eqref{IR} with $z\le C$. Then combined with Lemma \ref{lem_th} to handle the first integral on the RHS of \eqref{I2}, we obtain
\begin{equation}\begin{split}
& \int_{-\theta_1}^{\theta_1} \frac{W'(d_\theta )}{d_\theta }  \cdot \Big| \frac{1}{r}\Big(\eta(\frac{\epsilon}{r})-1\Big)(r-s\cos\theta) + \frac{1}{s}\Big(\eta(\frac{\epsilon}{s})-1\Big)(s-r\cos\theta)\Big|\rd{\theta} \\
\le & Ca_0\epsilon^2 (rs)^{-\beta}\frac{1}{r^2}\cdot\int_{-\pi}^\pi (z-\cos\theta)^{-\beta}(1-z\cos\theta)\rd{\theta} \le C\epsilon^2 M_1,  \\
\end{split}\end{equation}
using $r\ge 1$.

{\bf Case 2}: $|\theta|> \theta_1$. The integral $I$ with $|\theta|> \theta_1$ is empty if $\theta_1=\pi$. Therefore we only need to consider the first case in \eqref{theta11}.

We notice that for $|\theta|> \theta_1$, \eqref{theta11} implies that
\begin{equation}
\frac{W'(d_\theta )}{d_\theta } \ge \frac{1}{2}\Big(\frac{W'(d_\theta )}{d_\theta }- a_0 d_\theta ^{-2\beta}\Big),
\end{equation}
and \eqref{theta1_2} with \eqref{I2term2} implies that
\begin{equation}
|r-s|^2 \le Cr^2(1-z\cos\theta),\quad z-1 \le C(1-z\cos\theta).
\end{equation}
Therefore
\begin{equation}\begin{split}
& \int_{\theta_1\le |\theta|\le \pi} \frac{W'(d_\theta )}{d_\theta }  \cdot \Big| \frac{1}{r}\Big(\eta(\frac{\epsilon}{r})-1\Big)(r-s\cos\theta) + \frac{1}{s}\Big(\eta(\frac{\epsilon}{s})-1\Big)(s-r\cos\theta)\Big|\rd{\theta} \\
\le & C\epsilon^2 \int_{\theta_1\le |\theta|\le \pi} \Big(\frac{W'(d_\theta )}{d_\theta }- a_0 d_\theta ^{-2\beta}\Big) \cdot \Big(|r-s|^2\frac{1}{r^4} + (1-\cos\theta)\frac{1}{r^2}\Big)\rd{\theta} \\
\le & C\epsilon^2\frac{1}{r^2} \int_{\theta_1\le |\theta|\le \pi} \Big(\frac{W'(d_\theta )}{d_\theta }- a_0 d_\theta ^{-2\beta}\Big) \cdot (1-z\cos\theta)\rd{\theta} \\
\le & C\epsilon^2 M_2,
\end{split}\end{equation}
using $r\ge 1$ and the fact that the integrand in $M_2$, being the same as the last integrand, is nonnegative.

Finally, by choosing $\epsilon$ small enough, we can absorb the integral $I$ with $|\theta|\le \theta_1$ and $|\theta|> \theta_1$ by $M_1$ and $M_2$ respectively, and  therefore the proof of the case $r< C_1s$ is finished.

\end{proof}

\section{Some lemmas}

In this section we give some lemmas which will be used in the proof of Theorem \ref{thm_finite}.

\subsection{Basic lemmas}

From now on,  $\rho(t,\cdot)$ always denotes the solution to \eqref{eq0}. $\rho_t(\cdot)$ denotes a \emph{curve} of density distribution, i.e., a family of density distributions parametrized by $t\ge 0$, starting from some given $\rho=\rho_0$. $\rho_t$ may refer to different curves in different contexts.

We first state the following lemma which is a consequence of the 2-Wasserstein gradient flow structure of \eqref{eq0}. 
\begin{lemma}\label{lem_basic}
Let $\rho_t(\bx),\,0\le t \le t_1,\,t_1>0$ satisfy $\rho_0 = \rho_{\ini}$ and
\begin{equation}\label{lem_basic_1}
\partial_t \rho_t + \nabla\cdot (\rho_t \bv_t) = 0,
\end{equation}
for some velocity field $\bv_t(\bx)$ with $\int |\bv_t|^2\rho_t\rd{\bx} < \infty$. Assume 
\begin{equation}
\frac{\rd}{\rd{t}}\Big|_{t=0} E[\rho_t] < 0.
\end{equation}
Then the solution $\rho(t,\bx)$ to \eqref{eq0} satisfies
\begin{equation}
\frac{\rd}{\rd{t}}\Big|_{t=0} E[\rho(t,\cdot)] \le -\frac{(\frac{\rd}{\rd{t}} E[\rho_t(\cdot)])^2}{\int |\bv_t|^2\rho_t\rd{\bx}}\Big|_{t=0}.
\end{equation}
\end{lemma}
This lemma is the multi-dimensional version of Lemma 3.1 of~\cite{S1D}, and can be proved in a similar way. Therefore we omit its proof. This lemma says that, as long as we can find a curve $\rho_t$ which decreases the total energy, with its 2-Wasserstein cost $\int |\bv_t|^2\rho_t\rd{\bx}$ being finite, then we obtain a lower bound for the energy dissipation rate of the solution $\rho(t,\cdot)$. 

The following lemma describes the cost of a CSS-type curve in the 2-Wasserstein sense, when the curve is described by the horizontal movement of the level sets at each level $h$. It is the multi-dimensional analogue of Lemma 4.6 of~\cite{S1D}, and can be proved in a similar way, and we omit its proof.
\begin{lemma}\label{lem_cost}
Let $\rho_t$ be defined by
\begin{equation}
\rho_t(\bx) = \int_0^\infty \Big(\chi_{S_0(h)}(\bx)+ \sum_{j\ge 1} \chi_{S_{j,t}(h)}(\bx)\Big)\rd{h},
\end{equation}
where the sets $S_{j,t}(h),\,j\ge 1$ are translations of $S_{j,0}(h)$ with speed 1:
\begin{equation}
S_{j,t}(h) = S_{j,0}(h)+t\omega_j(h),\quad \omega_j(h) \in \mathbb{S}^1.
\end{equation}
Then at $t=0$, $\rho_t$ satisfies \eqref{lem_basic_1} with 
\begin{equation}
\bv_0(\bx) = \frac{1}{\rho_0(\bx)}\int_0^{\rho_0(\bx)} \sum_{j\ge 1,\,\bx\in S_{j,0}} \omega_j(h) \rd{h}
\end{equation}
Furthermore, we have the estimate
\begin{equation}
\int |\bv_0|^2\rho_0\rd{\bx} \le \int_0^\infty \sum_{j\ge 1} |S_{j,0}(h)|\rd{h}.
\end{equation}
\end{lemma}
The last estimate shows that, if some slices $S_{j,t}(h),\,j\ge 1$ are moving with speed 1 while other slices $S_0(h)$ are not moving, then the cost of the curve $\rho_t$ is bounded by the total mass of all moving mass.

The following lemma is a direct consequence of Theorem 1.1 of \cite{KZ}:
\begin{lemma}\label{lem_reg1}
The solution $\rho(t,\cdot)$ to \eqref{eq0} satisfies
\begin{equation}
\|\rho\|_{L^\infty([0,\infty)\times \mathbb{R}^2)} \le C.
\end{equation}
\end{lemma}

Finally we show that the total energy is bounded from below:
\begin{lemma}\label{lem_Em}
There exists $E_-\in \mathbb{R}$ such that 
\begin{equation}
E[\rho(t,\cdot)]\ge E_-,\quad \forall t\ge 0.
\end{equation}
\end{lemma}

\begin{proof}

{\bf (A2)} implies that 
\begin{equation}
W(r) = W(1) - \int_r^1 W'(s)\rd{s} \ge W(1) - C\int_r^1 \frac{1}{s}\rd{s} = W(1) + C\ln r,\quad \forall 0<r\le 1.
\end{equation}
Therefore, for any $t,\bx$,
\begin{equation}\begin{split}
 \int_{|\bx-\by|\le 1} W(\bx-\by)\rho(t,\by)\rd{\by} 
\ge & \int_{|\bx-\by|\le 1} (W(1) + C\ln |\bx-\by|)\rho(t,\by)\rd{\by} \\
\ge & \min\{W(1),0\} + C\|\rho\|_{L^\infty}\int_{|\bx-\by|\le 1}\ln |\bx-\by|\rd{\by} 
\ge  a,
\end{split}\end{equation}
for some $a\in\mathbb{R}$, independent of $t,\bx$. Also, 
\begin{equation}\begin{split}
& \int_{|\bx-\by|> 1} W(\bx-\by)\rho(t,\by)\rd{\by} \ge \int_{|\bx-\by|> 1} W(1)\rho(t,\by)\rd{\by} \ge  \min\{W(1),0\}.\\
\end{split}\end{equation}
Therefore,
\begin{equation}
\cI[\rho(t,\cdot)] \ge a +  \min\{W(1),0\} ,
\end{equation}
and then the conclusion follows from the positivity of $\cS[\rho(t,\cdot)]$.

\end{proof}

\subsection{Lemmas on the force field}

In this subsection we give several lemmas, which will be used to quantify the energy change in the CSS curves.

We define the interaction energy between two horizontal `slices' as
\begin{equation}\label{slices}
\cI_{x_2,y_2}[f,g] := \int\int W(\bx-\by)f(x_1)g(y_1)\rd{y_1}\rd{x_1} ,
\end{equation}
as a refined version of the bilinear interaction energy \eqref{bilinear}, where $\bx=(x_1,x_2)^T$ and $\by=(y_1,y_2)^T$, and $f$ and $g$ are functions of one variable.

In this subsection, all $\lambda$ appeared in an inequality denotes a positive lower bound of all possibly used values of $W'$. Therefore, $\lambda$ depends on $W$ and the parameters appeared in the statement. In the applications in later sections, the proper value of $\lambda$ will be made clear in the context. 

In fact, all results in this subsection works for all radial attractive potentials and do not require the assumptions {\bf (A1)}-{\bf (A5)}, except Lemma \ref{lem_circle} which requires {\bf (A2)}.

\begin{lemma}\label{lem_CSSb1}
Let $\kc >0, z>0$. Then 
\begin{equation}\label{lem_CSSb1_1}
\frac{\rd}{\rd{t}}\Big|_{t=0}\cI_{x_2,y_2}[\chi_{[-z,z]},\delta(\cdot - (\kc -t))] \le 0,
\end{equation}
where  $\cI_{x_2,y_2}$ is defined in \eqref{slices}. Furthermore, if $|x_2-y_2|\le \kc$, then
\begin{equation}\label{lem_CSSb1_2}
\frac{\rd}{\rd{t}}\Big|_{t=0}\cI_{x_2,y_2}[\chi_{[-z,z]},\delta(\cdot - (\kc -t))] \le -c\lambda \min\{\kc,z\}.
\end{equation}
\end{lemma}
See Figure \ref{fig_pic6} (left) for illustration.

\begin{figure}
\begin{center}
  \includegraphics[width=.49\linewidth]{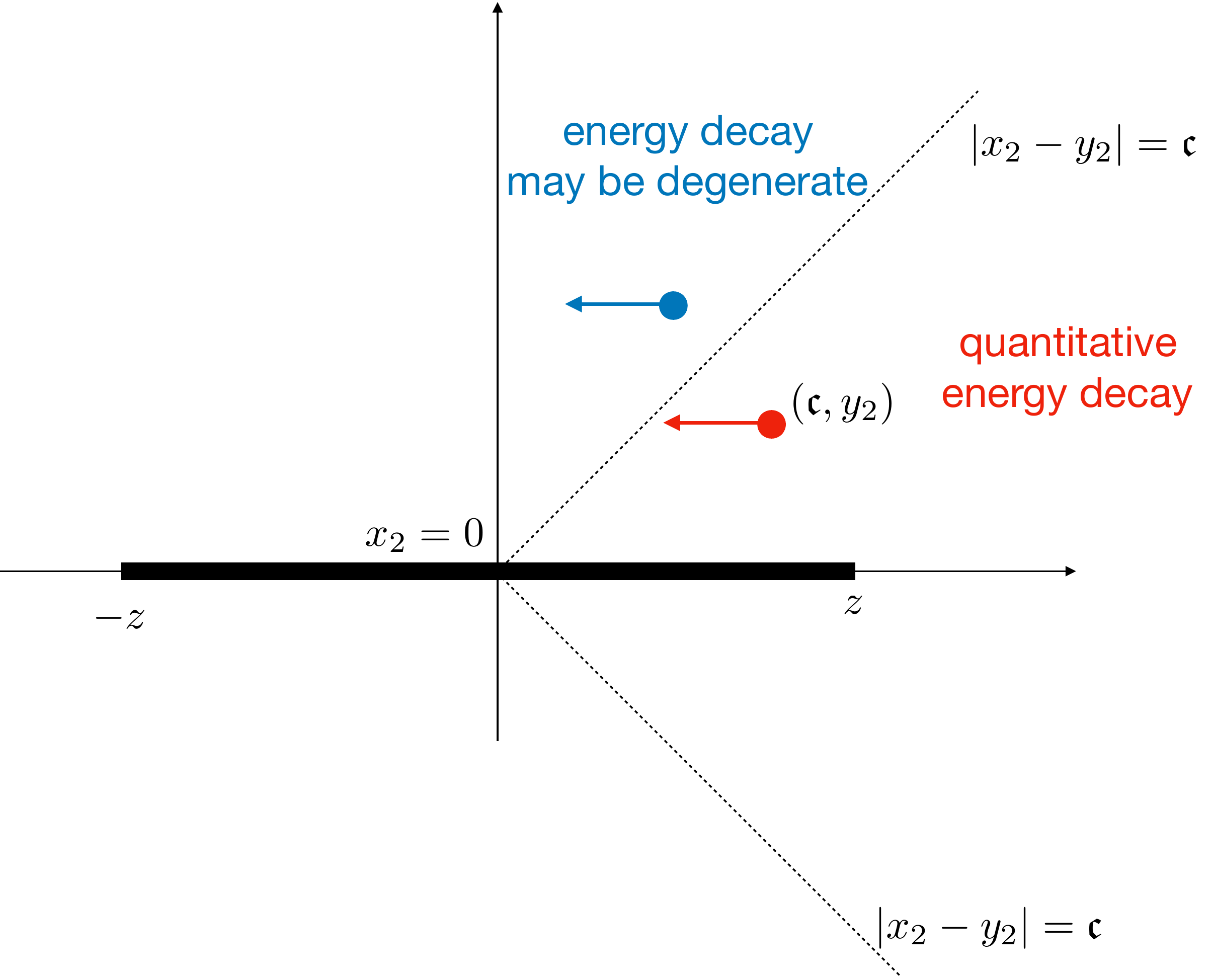}  
  \includegraphics[width=.49\linewidth]{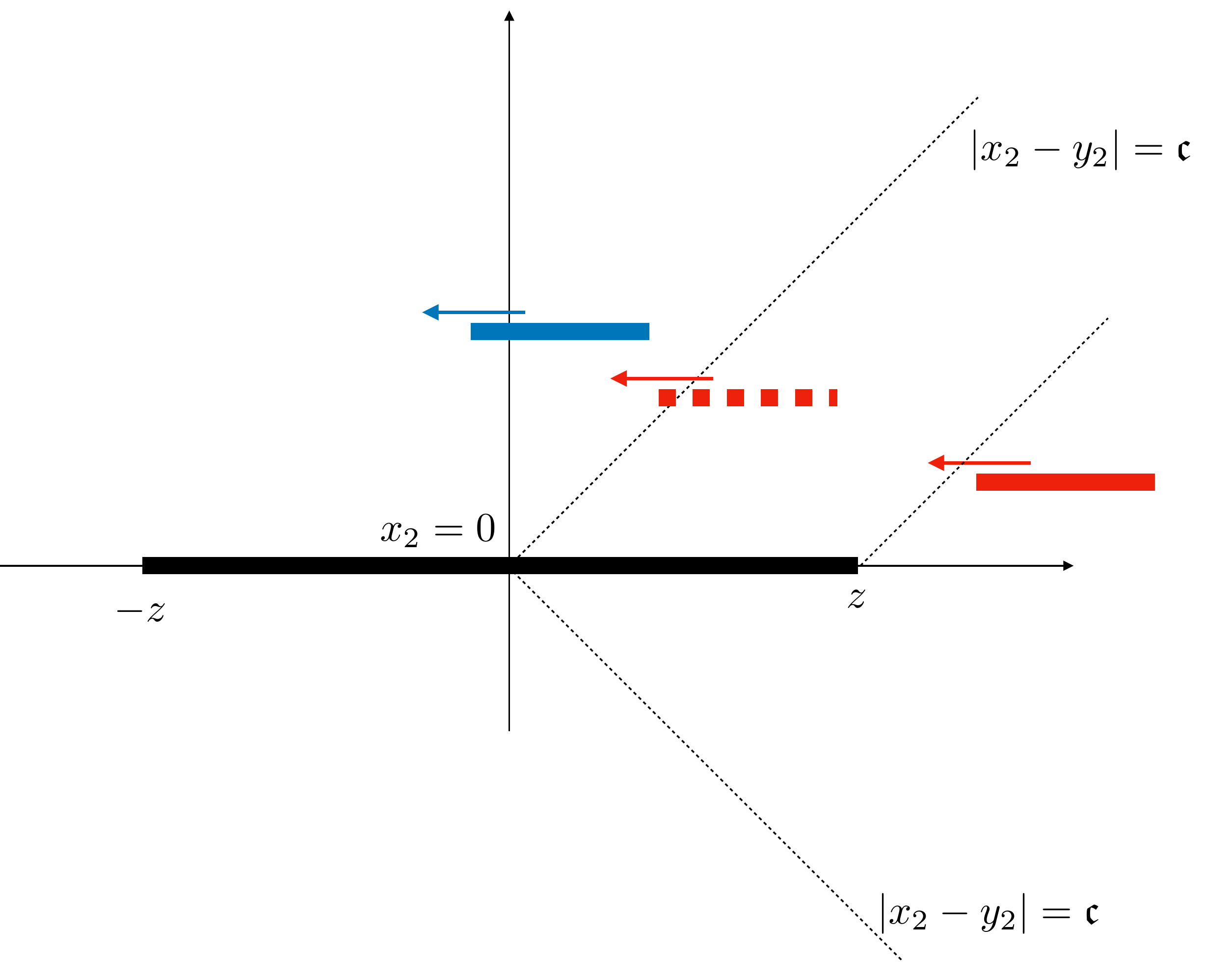}  
  \caption{Left: Lemma \ref{lem_CSSb1}, decay of interaction energy when moving a point mass against an interval. This decay may become degenerate when $\kc<<|x_2-y_2|$ (blue point), since the movement is horizontal, but the direction of the force may be almost vertical. This is avoided by the condition $|x_2-y_2|\le \kc$. Right: Lemma \ref{lem_CSSb2}, decay of interaction energy when moving an interval against an interval. The dashed red segment and the colored red segment denote the case of \eqref{lem_CSSb2_1} and \eqref{lem_CSSb2_2} respectively.}
\label{fig_pic6}
\end{center}
\end{figure}

\begin{proof}
If $\kc >z$, then
\begin{equation}\label{lem_CSS_01}\begin{split}
& \frac{\rd}{\rd{t}}\Big|_{t=0}\cI_{x_2,y_2}[\chi_{[-z,z]},\delta(\cdot - (\kc -t))] \\
= & \frac{\rd}{\rd{t}}\Big|_{t=0}\int\int W(\bx-\by)\chi_{[-z,z]}(x_1)\delta(y_1- (\kc -t))\rd{y_1}\rd{x_1} \\
= & \frac{\rd}{\rd{t}}\Big|_{t=0}\int W\Big((x_1-(\kc -t),x_2-y_2)^T\Big)\chi_{[-z,z]}(x_1)\rd{x_1} \\
= & \int \nabla W\Big((x_1-\kc ,x_2-y_2)^T\Big)\cdot\vec{e}_1\chi_{[-z,z]}(x_1)\rd{x_1}, \\
\end{split}\end{equation}
where $\vec{e}_1$ denotes the unit vector $(1,0)^T$. Notice that
\begin{equation}
\nabla W\Big((x_1-\kc ,x_2-y_2)^T\Big)\cdot\vec{e}_1 = \frac{W'\Big(\sqrt{(x_1-\kc)^2+(x_2-y_2)^2}\Big)}{\sqrt{(x_1-\kc )^2+(x_2-y_2)^2}}(x_1-\kc ).
\end{equation}
Therefore the last integrand in \eqref{lem_CSS_01} is always negative, since $x_1-\kc \le z-\kc <0$ by $\kc >z$, and $W'(r)>0$. This proves \eqref{lem_CSSb1_1} (for the case $\kc>z$).

Furthermore, if $|x_2-y_2|\le \kc$, then
\begin{equation}
\sqrt{(x_1-\kc)^2+(x_2-y_2)^2} \le \sqrt{(x_1-\kc)^2+(x_1-\kc)^2} = \sqrt{2}|x_1-\kc|,\quad \forall x_1\le 0.
\end{equation}
Therefore
\begin{equation}\begin{split}
 & \int \nabla W\Big((x_1-\kc ,x_2-y_2)^T\Big)\cdot\vec{e}_1\chi_{[-z,z]}(x_1)\rd{x_1} \\
\le & - \frac{1}{\sqrt{2}}\int_{-z}^0 W'\Big(\sqrt{(x_1-\kc)^2+(x_2-y_2)^2}\Big) \rd{x_1}\\
\le & - \frac{1}{\sqrt{2}}\lambda z,
\end{split}\end{equation}
which proves \eqref{lem_CSSb1_2} (for the case $\kc>z$).

If $\kc\le z$, then by symmetry we have
\begin{equation}
\int \nabla W\Big((x_1-\kc ,x_2-y_2)^T\Big)\cdot\vec{e}_1\chi_{[-z,z]}(x_1)\rd{x_1} = \int \nabla W\Big((x_1-\kc ,x_2-y_2)^T\Big)\cdot\vec{e}_1\chi_{[-z,2\kc-z]}(x_1)\rd{x_1},
\end{equation}
and the conclusion is obtained by using a translated version of the case $\kc>z$.

\end{proof}

\begin{lemma}\label{lem_CSSb2}
Let $\kc>0, z>0,z'>0$. Then 
\begin{equation}\label{lem_CSSb2_0}
\frac{\rd}{\rd{t}}\Big|_{t=0}\cI_{x_2,y_2}[\chi_{[-z,z]},\chi_{[(\kc-t)-z',(\kc-t)+z']}] \le 0.
\end{equation}
Furthermore, if $\kc>z'$ and $|x_2-y_2|\le \kc$, then 
\begin{equation}\label{lem_CSSb2_1}
\frac{\rd}{\rd{t}}\Big|_{t=0}\cI_{x_2,y_2}[\chi_{[-z,z]},\chi_{[(\kc-t)-z',(\kc-t)+z']}] \le -c\lambda \min\{z,\kc\}z'.
\end{equation}
If $|x_2-y_2|\le \kc-z-z'$, then 
\begin{equation}\label{lem_CSSb2_2}
\frac{\rd}{\rd{t}}\Big|_{t=0}\cI_{x_2,y_2}[\chi_{[-z,z]},\chi_{[(\kc-t)-z',(\kc-t)+z']}] \le -c\lambda(S,S')\cdot |S|\cdot|S'|,
\end{equation}
for any $S\subset[-z,z]$ and $S'\subset[\kc-z',\kc+z']$, and $\lambda(S,S')$ denotes a lower bound for all possibly used $W'$ in the interaction between $(x_1,x_2)$ and $(y_1,y_2)$ with $x_1\in S$ and $y_1\in S'$.
\end{lemma}

Notice that the last condition $|x_2-y_2|\le \kc-z-z'$ says that the horizontal distance between the closest pair of endpoints of the two intervals $\{x_2\}\times [-z,z]$ and $\{y_2\}\times [\kc-z',\kc+z']$ is at least their vertical distance $|x_2-y_2|$. See Figure \ref{fig_pic6} (right) for illustration.

\begin{proof}

We first show \eqref{lem_CSSb2_0}. If $\kc > z'$, then 
\begin{equation}\label{lem_CSSb2_01}
\frac{\rd}{\rd{t}}\Big|_{t=0}\cI_{x_2,y_2}[\chi_{[-z,z]},\chi_{[(\kc-t)-z',(\kc-t)+z']}] = \int_{\kc-z'}^{\kc+z'}\frac{\rd}{\rd{t}}\Big|_{t=0}\cI_{x_2,y_2}[\chi_{[-z,z]},\delta(\cdot - (\kc_1-t))] \rd{\kc_1},
\end{equation}
with $\kc_1>0$ always holds. Then \eqref{lem_CSSb2_0} follows from \eqref{lem_CSSb1_1}. If $\kc\le z'$, then by symmetry we have
\begin{equation}
\frac{\rd}{\rd{t}}\Big|_{t=0}\cI_{x_2,y_2}[\chi_{[-z,z]},\chi_{[(\kc-t)-z',(\kc-t)+z']}] = \frac{\rd}{\rd{t}}\Big|_{t=0}\cI_{x_2,y_2}[\chi_{[-z,z]},\chi_{[z'-\kc-t,(\kc-t)+z']}],
\end{equation}
and \eqref{lem_CSSb2_0} follows from the previous case.

To see \eqref{lem_CSSb2_1}, notice that in the case $\kc>z'$ and $|x_2-y_2|\le \kc$, in \eqref{lem_CSSb2_01} there holds
\begin{equation}
\kc_1\ge |x_2-y_2|,\quad \forall \kc \le \kc_1 \le \kc+z'.
\end{equation}
Then applying \eqref{lem_CSSb1_2} to these $\kc_1$ gives \eqref{lem_CSSb2_1}.

To see \eqref{lem_CSSb2_2}, we write
\begin{equation}\begin{split}
& \frac{\rd}{\rd{t}}\Big|_{t=0}\cI_{x_2,y_2}[\chi_{[-z,z]},\chi_{[(\kc-t)-z',(\kc-t)+z']}] \\
= & \int_{\kc-z'}^{\kc+z'}\frac{\rd}{\rd{t}}\Big|_{t=0}\cI_{x_2,y_2}[\chi_{[-z,z]},\delta(\cdot - (\kc_1-t))] \rd{\kc_1} \\
= & \int_{\kc-z'}^{\kc+z'}\int_{-z}^z \nabla W\Big((x_1-\kc_1 ,x_2-y_2)^T\Big)\cdot\vec{e}_1\rd{x_1} \rd{\kc_1}, \\
\end{split}\end{equation}
and $|x_2-y_2| \le \kc_1-x_1$ always holds, due to the assumption $|x_2-y_2|\le \kc-z-z'$. Therefore the integrand above is always negative, and those $(x_1,x_2)$ and $(y_1,y_2)$ with $x_1\in S$ and $y_1\in S'$ can be estimated in the same way as the proof of Lemma \ref{lem_CSSb1} to give \eqref{lem_CSSb2_2}.

\end{proof}

%\begin{remark}
%
%The quantitative estimates \eqref{lem_CSSb2_1} can be refined into
%\begin{equation}\label{lem_CSSb2_1}
%\frac{\rd}{\rd{t}}\Big|_{t=0}\cI_{x_2,y_2}[\chi_{[-z,z]},\chi_{[(\kc-t)-z',(\kc-t)+z']}] \le -c\lambda \min\{z,\kc\}z'
%\end{equation}
%
%
%\end{remark}

\begin{lemma}\label{lem_disk}
Potential generated by a disk is attractive:
\begin{equation}\label{lem_disk_1}
\int \nabla W(\bx-\by)\cdot \frac{\bx}{|\bx|}\chi_{|\by|\le R}(\by)\rd{\by}  \ge c \lambda \min\{R,|\bx|\}^2,
\end{equation}
for any $\bx\ne 0$, $R>0$.
\end{lemma}

\begin{proof}

By radial symmetry, we may assume $\bx=(r,0)^T,\,r>0$ without loss of generality. Then $\frac{\bx}{|\bx|}=\vec{e}_1$.

Notice that
\begin{equation}\begin{split}
& - \int \nabla W(\bx-\by)\cdot \frac{\bx}{|\bx|}\chi_{|\by|\le R}(\by)\rd{\by} \\
= & \int_{-R}^R \frac{\rd}{\rd{t}}\Big|_{t=0}\cI_{x_2,0}\Big[\chi_{\left[-\sqrt{R^2-x_2^2},\sqrt{R^2-x_2^2}\right]}, \delta(\cdot - (r-t))\Big] \rd{x_2} \\
\le & - c\lambda \int_{x_2\in[-R,R], |x_2|\le r}  \min\Big\{r,\sqrt{R^2-x_2^2}\Big\} \rd{x_2}, \\
\end{split}\end{equation}
by Lemma \ref{lem_CSSb1}.

For the case $|\bx| = r > R$, we have
\begin{equation}\begin{split}
& \int_{x_2\in[-R,R], |x_2|\le r}  \min\Big\{r,\sqrt{R^2-x_2^2}\Big\} \rd{x_2} 
=  \int_{x_2\in[-R,R]}  \sqrt{R^2-x_2^2} \rd{x_2} 
=  cR^2.
\end{split}\end{equation}

For the case $r \le R$, we have
\begin{equation}\begin{split}
& \int_{x_2\in[-R,R], |x_2|\le r}  \min\Big\{r,\sqrt{R^2-x_2^2}\Big\} \rd{x_2} 
=  \int_{x_2\in[-r,r]}  \min\Big\{r,\sqrt{R^2-x_2^2}\Big\} \rd{x_2} \\
\ge & \int_{x_2\in[-r/2,r/2]}  \min\Big\{r,\sqrt{R^2-x_2^2}\Big\} \rd{x_2} 
\ge  \int_{x_2\in[-r/2,r/2]}  \frac{1}{2}r \rd{x_2} 
= cr^2.
\end{split}\end{equation}

\end{proof}

\begin{lemma}\label{lem_circle}
Assume {\bf (A2)}. Then the potential generated by a circle is Lipschitz:
\begin{equation}\label{lem_circle_1}
\left|\int \nabla W(\bx-\by)\delta(|\by|-R)\rd{\by}\right| \le C  ,
\end{equation}
for any $\bx\in\mathbb{R}^2$, $R>0$. Furthermore, if $|\bx|>R$, then
\begin{equation}\label{lem_circle_2}
\int \nabla W(\bx-\by)\delta(|\by|-R)\rd{\by} \ge c\lambda R .
\end{equation}
\end{lemma}

Notice that the density distribution $\delta(|\by|-R)$ has total mass $2\pi R$.

\begin{proof}

\eqref{lem_circle_2} can be proved similarly as the case $r>R$ of Lemma \ref{lem_disk} and we omit its proof.

By radial symmetry, we may assume $\bx=(r,0)^T,\,r>0$ without loss of generality. Then the vector $\int \nabla W(\bx-\by)\delta(|\by|-R)\rd{\by}$ is parallel to $\vec{e}_1$.

Then we compute
\begin{equation}\begin{split}
& \int \nabla W(\bx-\by)\cdot \vec{e}_1 \delta(|\by|-R)\rd{\by} \\
= & R\int_{-\pi}^\pi \nabla W\Big((r-R\cos\theta, -R\sin\theta)^T\Big)\cdot \vec{e}_1  \rd{\theta} \\
= & R\int_{-\pi}^\pi W'\Big(\sqrt{r^2-2rR\cos\theta+R^2}\Big) \frac{r-R\cos\theta}{\sqrt{r^2-2rR\cos\theta+R^2}}  \rd{\theta}. \\
\end{split}\end{equation}

By the assumption ({\bf A2}),
\begin{equation}
W'(r) \le C\frac{1}{r}.
\end{equation}
If $r>R$, then $r-R\cos\theta>0$ for any $\theta$. Therefore
\begin{equation}\begin{split}
& \left|R\int_{-\pi}^\pi W'\Big(\sqrt{r^2-2rR\cos\theta+R^2}\Big) \frac{r-R\cos\theta}{\sqrt{r^2-2rR\cos\theta+R^2}}  \rd{\theta}\right| \\
= & R\int_{-\pi}^\pi W'\Big(\sqrt{r^2-2rR\cos\theta+R^2}\Big) \frac{r-R\cos\theta}{\sqrt{r^2-2rR\cos\theta+R^2}}  \rd{\theta} \\
\le & CR\int_{-\pi}^\pi  \frac{r-R\cos\theta}{r^2-2rR\cos\theta+R^2}  \rd{\theta} \\
= & \frac{CR}{r}\eta(\frac{R}{r}), \\
\end{split}\end{equation}
where $\eta$ is defined in \eqref{eta}. Then the conclusion follows from the uniform bound \eqref{etaeps2}.

If $r<R$, then define the reflected point $\tilde{r} = 2R-r$ which satisfies $\tilde{r}>R$. We claim
\begin{equation}\label{claimrR}
|r-R\cos\theta| \le \tilde{r}-R\cos\theta,\quad \sqrt{r^2-2rR\cos\theta+R^2} \ge c\sqrt{\tilde{r}^2-2\tilde{r}R\cos\theta+R^2}.
\end{equation}
The first inequality is clear. To see the second inequality, we separate into the following cases:
\begin{itemize}
\item If $|\theta| < 0.1(1-\frac{r}{R})$, then
\begin{equation}\begin{split}
\sqrt{r^2-2rR\cos\theta+R^2} = & \sqrt{(r-R\cos\theta)^2+(R\sin\theta)^2} \ge R\cos\theta -r\\
 \ge & R\Big(1-\frac{0.1^2}{2}(1-\frac{r}{R})^2\Big)-r,
\end{split}\end{equation}
and
\begin{equation}
\sqrt{\tilde{r}^2-2\tilde{r}R\cos\theta+R^2} \le \tilde{r}-R\cos\theta +R|\sin\theta| \le \tilde{r}-R\Big(1-\frac{0.1^2}{2}(1-\frac{r}{R})^2\Big) + 0.1R(1-\frac{r}{R}).
\end{equation}
Taking quotient gives (with $y:=r/R$)
\begin{equation}
\frac{\sqrt{r^2-2rR\cos\theta+R^2}}{\sqrt{\tilde{r}^2-2\tilde{r}R\cos\theta+R^2}} \ge \frac{1-0.005(1-y)^2 - y}{2-y-1+0.005(1-y)^2 + 1-y} = \frac{(1-y)(0.995+0.005y)}{(1-y)(2.005-0.005y)} \ge \frac{1}{3}.
\end{equation}
for any $0\le y < 1$. 
\item If $ 0.1(1-\frac{r}{R}) \le |\theta| \le \frac{\pi}{2}$, then
\begin{equation}
\sqrt{r^2-2rR\cos\theta+R^2} \ge R|\sin\theta| \ge \frac{2}{\pi}R|\theta|,
\end{equation}
and
\begin{equation}\begin{split}
\sqrt{\tilde{r}^2-2\tilde{r}R\cos\theta+R^2} \le & \tilde{r}-R\cos\theta +R|\sin\theta| \le \tilde{r}-R(1-\frac{\theta^2}{2}) + R|\theta| \\
\le & R-r + 2R|\theta| \le 10R|\theta|+ 2R|\theta| = 12R|\theta|.
\end{split}\end{equation}
\item If $  |\theta| \ge \frac{\pi}{2}$, then
\begin{equation}
\sqrt{r^2-2rR\cos\theta+R^2} \ge R,\quad \sqrt{\tilde{r}^2-2\tilde{r}R\cos\theta+R^2} \le 3R.
\end{equation}

\end{itemize}

This proves \eqref{claimrR}, and the conclusion follows from the previous case $r>R$.

\end{proof}

\subsection{Consequence of the subcritical condition}

We first show that, the subcritical condition {\bf (A5)} is the key to guarantee that a positive amount of mass will stay near the center:
\begin{lemma}\label{lem_crho}
There exists $\cR_1>0$ large, and $c_\rho>0$, such that
\begin{equation}
\int_{|\bx|\le \cR_1} \rho(t,\bx)\rd{\bx} \ge c_\rho,\quad \forall t\ge 0.
\end{equation}
\end{lemma}

\begin{proof}
{\bf (A4)}  implies that for any $t\ge 0$, $\rho(t,\cdot)$ is non-negative, radially-symmetric, and has mass 1. {\bf (A5)} implies that there exists $\cR_0>0$ large and $\epsilon>0$ small,  such that
\begin{equation}
E[\rho_{\ini}] \le \frac{1}{2}W(\cR_0) - \epsilon.
\end{equation}
Therefore,
\begin{equation}\label{Eupper}
E[\rho(t,\cdot)] \le E[\rho_{\ini}] \le \frac{1}{2}W(\cR_0) - \epsilon,
\end{equation}
for all $t\ge 0$.

Suppose on the contrary that for some $t\ge 0$, 
\begin{equation}
\int_{|\bx|\le \cR_1} \rho(t,\bx)\rd{\bx} < c_\rho,
\end{equation}
for some large $\cR_1>0$ and small $c_\rho>0$ to be determined. Then (omitting the dependence on $t$ from now on)
\begin{equation}
\int_{|\bx|\ge \cR_1} \rho(\bx)\rd{\bx} > 1-c_\rho.
\end{equation}
Therefore, at fixed $\bx = (r,0)^T$ with $r>\cR_1/2$, the potential generated by $\rho$ can be written as (with $d_\theta$ as defined in \eqref{notation1})
\begin{equation}\label{Wdtheta}\begin{split}
& \int W(\bx-\by)\rho(\by)\rd{\by} 
=  \int_0^\infty \int_{-\pi}^\pi W( d_\theta) \rd{\theta}s\rho(s)\rd{s} \\
= & \iint_{\Big([0,\infty)\times[-\pi,\pi]\Big) \backslash \Big([r-\cR_0,r+\cR_0]\times [-\theta_0,\theta_0]\Big)}  W( d_\theta) \rd{\theta}s\rho(s)\rd{s} \\
& + \iint_{[r-\cR_0,r+\cR_0]\times [-\theta_0,\theta_0]}  W( d_\theta) \rd{\theta}s\rho(s)\rd{s}, \\
\end{split}\end{equation}
where 
\begin{equation}
\theta_0 = \theta_0(r) = 2\cR_0/r,
\end{equation}
is small by taking $r>\cR_1/2$ large. See Figure \ref{fig_pic11} for illustration.

\begin{figure}
\begin{center}
  \includegraphics[width=.5\linewidth]{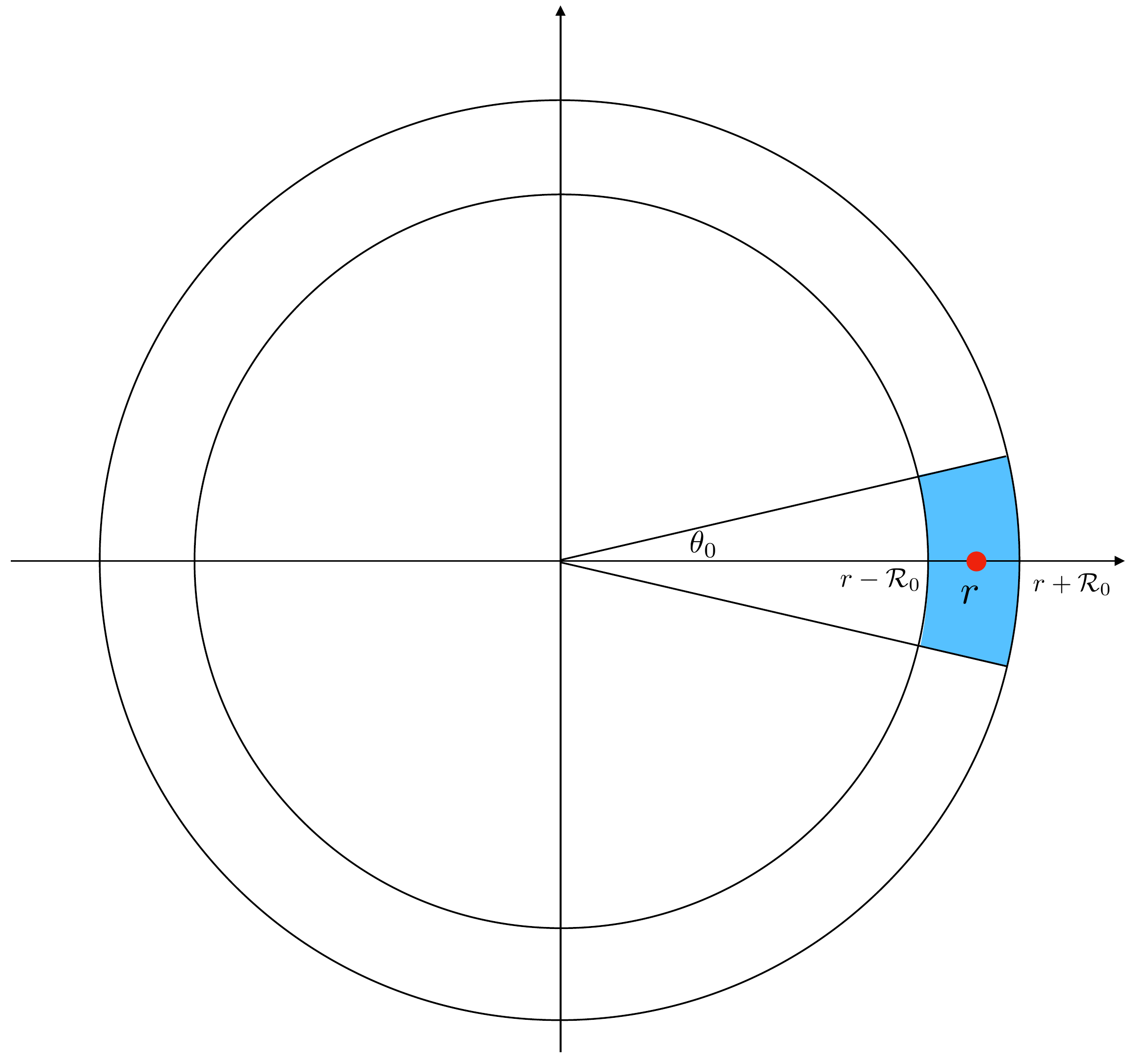}  
  \caption{The decomposition \eqref{Wdtheta}. The blue region corresponds to the last term, which contains every point with $d_\theta<\cR_0$.}
\label{fig_pic11}
\end{center}
\end{figure}

{\bf STEP 1}: estimate the first term on the RHS of \eqref{Wdtheta}.

We first claim that for any $(s,\theta)\in \Big([0,\infty)\times[-\pi,\pi]\Big) \backslash \Big([r-\cR_0,r+\cR_0]\times [-\theta_0,\theta_0]\Big)$, we have
\begin{equation}\label{dlarge}
d_\theta   \ge \cR_0.
\end{equation}
In fact, if $|s-r|>\cR_0$, then either $|r-s\cos\theta|> \cR_0$ or $|s-r\cos\theta|> \cR_0$ (depending on whether $r>s$ or $s>r$). Thus \eqref{dlarge} follows by noticing
\begin{equation}
d_\theta = \sqrt{(r-s\cos\theta)^2+(s\sin\theta)^2} = \sqrt{(s-r\cos\theta)^2+(r\sin\theta)^2}.
\end{equation}
If $|\theta| > \theta_0$, then \eqref{dlarge} follows from
\begin{equation}
|r\sin\theta| \ge r \cdot \frac{2}{\pi}\theta \ge r \cdot \frac{2}{\pi}\cdot \frac{2\cR_0}{r} \ge \cR_0.
\end{equation}
Therefore
\begin{equation}\label{Wdtheta11}\begin{split}
& \iint_{ \Big([0,\infty)\times[-\pi,\pi]\Big) \backslash \Big([r-\cR_0,r+\cR_0]\times [-\theta_0,\theta_0]\Big)}  W( d_\theta) \rd{\theta}s\rho(s)\rd{s} \\
\ge & W(\cR_0) \iint_{  \Big([0,\infty)\times[-\pi,\pi]\Big) \backslash \Big([r-\cR_0,r+\cR_0]\times [-\theta_0,\theta_0]\Big)}  \rd{\theta}s\rho(s)\rd{s} \\
= & W(\cR_0)\left( 1-\iint_{ [r-\cR_0,r+\cR_0]\times [-\theta_0,\theta_0]}  \rd{\theta}s\rho(s)\rd{s} \right) \\
= & W(\cR_0)\left( 1-2\theta_0\int_{ [r-\cR_0,r+\cR_0]} s\rho(s)\rd{s} \right) \\
\ge & W(\cR_0)\left( 1-\frac{\theta_0}{\pi}\iint_{ [0,\infty)\times [-\pi,\pi]}  \rd{\theta}s\rho(s)\rd{s} \right) \\
= & W(\cR_0)\Big( 1-\frac{\theta_0}{\pi} \Big). \\
\end{split}\end{equation}

{\bf STEP 2}: estimate the second term on the RHS of \eqref{Wdtheta}.

First notice that
\begin{equation}
|s-r\cos\theta| \le |s-r| + r(1-\cos\theta) \le \cR_0 + r\cdot\frac{\theta_0^2}{2} = \cR_0 + r\cdot\frac{2\cR_0^2}{r^2} \le 2\cR_0,
\end{equation}
if $r>\cR_1/2$ is large enough. It follows that
\begin{equation}
d_\theta = \sqrt{(s-r\cos\theta)^2+(r\sin\theta)^2} \le |s-r\cos\theta| + r|\theta_0| \le 4\cR_0.
\end{equation}

Then, notice that {\bf (A2)} implies that, for any $0<u\le 4\cR_0$,
\begin{equation}\label{Wlower}\begin{split}
W(u) = & W(4\cR_0) - \int_u^{4\cR_0}W'(u_1)\rd{u_1} \ge  W(4\cR_0) - C\int_u^{4\cR_0}\frac{1}{u_1}\rd{u_1} \\
= & \Big(W(4\cR_0) - C\ln(4\cR_0)\Big) + C\ln u =: A + C\ln u,
\end{split}\end{equation}
(where $A$ may be positive or negative). Therefore
\begin{equation}\begin{split}
& \iint_{[r-\cR_0,r+\cR_0]\times [-\theta_0,\theta_0]}  W( d_\theta) \rd{\theta}s\rho(s)\rd{s}  \\
\ge & A \iint_{[r-\cR_0,r+\cR_0]\times [-\theta_0,\theta_0]} \rd{\theta}s\rho(s)\rd{s} + C\iint_{[r-\cR_0,r+\cR_0]\times [-\theta_0,\theta_0]}  \ln d_\theta \rd{\theta}s\rho(s)\rd{s}.
\end{split}\end{equation}
As in \eqref{Wdtheta11}, we have $\iint_{[r-\cR_0,r+\cR_0]\times [-\theta_0,\theta_0]} \rd{\theta}s\rho(s)\rd{s} \le \frac{\theta_0}{\pi}$. To estimate the other term,
\begin{equation}\begin{split}
\int_{-\theta_0}^{\theta_0}  \ln d_\theta \rd{\theta} = & \frac{1}{2}\int_{-\theta_0}^{\theta_0}  \ln  \Big((s-r\cos\theta)^2+(r\sin\theta)^2\Big)\rd{\theta} \\
\ge & \int_{-\theta_0}^{\theta_0}  \ln |r\sin\theta|\rd{\theta} 
\ge  2\theta_0\ln \frac{2r}{\pi} + 2\int_0^{\theta_0}\ln \theta \rd{\theta} 
=  2\theta_0\ln \frac{2r}{\pi} + 2(\theta_0\ln\theta_0 - \theta_0). \\
\end{split}\end{equation}
Therefore we obtain
\begin{equation}\begin{split}
& \iint_{[r-\cR_0,r+\cR_0]\times [-\theta_0,\theta_0]}  W( d_\theta) \rd{\theta}s\rho(s)\rd{s}  \\
\ge & -|A|\cdot\frac{\theta_0}{\pi} + C\Big(2\theta_0\ln \frac{2r}{\pi} + 2(\theta_0\ln\theta_0 - \theta_0)\Big)\int_{[r-\cR_0,r+\cR_0]}s\rho(s)\rd{s} \\
\ge & -|A|\cdot\frac{\theta_0}{\pi} + \min\Big\{C\Big(2\theta_0\ln \frac{2r}{\pi} + 2(\theta_0\ln\theta_0 - \theta_0)\Big),0\Big\}, \\
\end{split}\end{equation}
where the last inequality follows from the fact that the total mass is 1.

{\bf STEP 3}: estimate the potential for small $r$.

For $r = |\bx| \le \cR_1/2$, we have 
\begin{equation}\label{Wrsmall}
\int W(\bx-\by)\rho(\by)\rd{\by} =  \int_{|\by| > \cR_1} W(\bx-\by)\rho(\by)\rd{\by} + \int_{|\by| \le \cR_1} W(\bx-\by)\rho(\by)\rd{\by} .
\end{equation}
The first term is estimated by
\begin{equation}
\int_{|\by| > \cR_1} W(\bx-\by)\rho(\by)\rd{\by} \ge W(\cR_1/2)\int_{|\by| > \cR_1} \rho(\by)\rd{\by} \ge \min\{W(\cR_1/2),0\},
\end{equation}
since $|\bx-\by| \ge \cR_1/2$ always holds in this case. 

To estimate the second term on the RHS of \eqref{Wrsmall}, using the $L^\infty$ bound on $\rho$ (from Lemma \ref{lem_reg1}) and the increasing property of $W(r)$,
\begin{equation}\begin{split}
\int_{|\by| \le \cR_1} W(\bx-\by)\rho(\by)\rd{\by}  \ge & \int_{|\by-\bx| \le \delta} W(\bx-\by)\|\rho\|_{L^\infty}\rd{\by} \ge \|\rho\|_{L^\infty}\int_{|\by| \le \delta} (A+C\ln|\by|)\rd{\by},
\end{split}\end{equation}
using the notation \eqref{Wlower}, where the first inequality is obtained by minimizing among all possible $\rho$ with the same total mass and $L^\infty$ bound, and
\begin{equation}\label{defdelta}
\delta  = \sqrt{\frac{\int_{|\by| \le \cR_1}\rho(\by)\rd{\by}}{\pi\|\rho\|_{L^\infty}}} \le \sqrt{\frac{c_\rho}{\pi\|\rho\|_{L^\infty}}},
\end{equation}
is determined by $\int_{|\by| \le \cR_1}\rho(\by)\rd{\by} = \int_{|\by-\bx| \le \delta} \|\rho\|_{L^\infty} \rd{\by}$.

Then notice that
\begin{equation}\begin{split}
 \int_{|\by| \le \delta} \ln|\by|\rd{\by} = 2\pi\int_0^\delta s\ln s \rd{s} = 2\pi(\frac{\delta^2}{2}\ln\delta - \frac{\delta^2}{4}),
\end{split}\end{equation}
and we obtain
\begin{equation}\label{Wrsmall1}
\int W(\bx-\by)\rho(\by)\rd{\by} \ge  \min\{W(\cR_1/2),0\} -  C|A| \cdot\pi\delta^2 + C\cdot2\pi(\frac{\delta^2}{2}\ln\delta - \frac{\delta^2}{4}).
\end{equation}

{\bf STEP 4}: choose $\cR_1$ and $c_\rho$ to finalize.

STEPs 1 and 2 imply that for any $|\bx|>\cR_1/2$,
\begin{equation}
\int W(\bx-\by)\rho(\by)\rd{\by} \ge W(\cR_0)\Big( 1-\frac{\theta_0}{\pi} \Big) + \Big(-|A|\cdot\frac{\theta_0}{\pi} + \min\Big\{C\Big(2\theta_0\ln \frac{2r}{\pi} + 2(\theta_0\ln\theta_0 - \theta_0)\Big),0\Big\}\Big),
\end{equation}
By choosing $\cR_1$ large enough so that $r>\cR_1/2$ is large and $\theta_0=2\cR_0/r$ is small enough, we can obtain
\begin{equation}
\int W(\bx-\by)\rho(\by)\rd{\by} \ge W(\cR_0) - \epsilon/2,\quad \forall |\bx|>\cR_1/2.
\end{equation}

For $|\bx|\le\cR_1/2$, from STEP 3, by taking $c_\rho$ small,
\begin{equation}
\int W(\bx-\by)\rho(\by)\rd{\by} \ge \min\{W(\cR_1/2),0\} - \epsilon,\quad \forall |\bx|\le\cR_1/2,
\end{equation}
from \eqref{Wrsmall1} (recall \eqref{defdelta}, the definition of $\delta$, implying that $\delta$ can be made small if $c_\rho$ is small).

Therefore we obtain
\begin{equation}\begin{split}
2\cI[\rho] = & \int\int W(\bx-\by)\rho(\by)\rd{\by}\rho(\bx)\rd{\bx} \\
= & \int_{|\bx|>\cR_1/2}\int W(\bx-\by)\rho(\by)\rd{\by}\rho(\bx)\rd{\bx} + \int_{|\bx|\le \cR_1/2}\int W(\bx-\by)\rho(\by)\rd{\by}\rho(\bx)\rd{\bx} \\
\ge & \int_{|\bx|>\cR_1/2}(W(\cR_0) - \epsilon/2)\rho(\bx)\rd{\bx} + \int_{|\bx|\le \cR_1/2}(\min\{W(\cR_1/2),0\} - \epsilon)\rho(\bx)\rd{\bx} \\
\ge & (1-c_\rho)(W(\cR_0) - \epsilon/2) + (\min\{W(\cR_1/2),0\} - \epsilon) c_\rho \\
\ge & W(\cR_0) - \epsilon/2 - \Big( W(\cR_0) +( \epsilon-\min\{W(\cR_1/2),0\})\Big) c_\rho .
\end{split}\end{equation}
By further taking $c_\rho$ small, we have
\begin{equation}
\Big| W(\cR_0) +( \epsilon-\min\{W(\cR_1/2),0\})\Big| c_\rho  \le \frac{\epsilon}{4},
\end{equation}
which contradicts \eqref{Eupper}.

\end{proof}

Next we show that the mass near the center guaranteed by Lemma \ref{lem_crho} provides enough attraction force towards the center, by the following lemma and its corollary:

\begin{lemma}\label{lem_ctr}
For any $0<\delta<1$, there holds
\begin{equation}\begin{split}
& \int \nabla W(\bx-\by)\cdot \frac{\bx}{|\bx|}\Big(\rho^*(\by) + \rho_\flat(\by)\Big) \rd{\by}\\
& \ge  c\lambda(r + \delta^{-1/2})\int_{|\by|\le r} \Big(\max\{\rho^*(\by)-\delta,0\} + \max\{ \rho_\flat(\by)-\delta,0\}\Big)\rd{\by} - \frac{C\varepsilon^2}{r}.
\end{split}\end{equation}
\end{lemma}

%Here we denote $f_{\ge 0} := \max\{f,0\}$ for any quantity $f$.

\begin{proof}

Using the $h$-representation \eqref{rhoh1},
\begin{equation}\label{rhostfl}\begin{split}
& \int \nabla W(\bx-\by)\cdot \frac{\bx}{|\bx|}(\rho^*(\by) + \rho_\flat(\by)) \rd{\by} \\
= & \int_0^\infty\Big[ \int \nabla W(\bx-\by)\cdot \frac{\bx}{|\bx|}\chi_{|\by|\le R_0(h)}(\by)\rd{\by} + \sum_{j \text{ flat}}\int \nabla W(\bx-\by)\cdot \frac{\bx}{|\bx|}\chi_{r_j(h)\le |\by|\le R_j(h)}(\by)\rd{\by} \Big]\rd{h} \\
= & \int_0^\infty\Big[ \int \nabla W(\bx-\by)\cdot \frac{\bx}{|\bx|}\chi_{|\by|\le R_0(h)}(\by)\rd{\by} + \sum_{j \text{ flat}}\int \nabla W(\bx-\by)\cdot \frac{\bx}{|\bx|}\chi_{ |\by|\le R_j(h)}(\by)\rd{\by} \Big]\rd{h} \\
& -  \int_0^\infty \sum_{j \text{ flat}}\int \nabla W(\bx-\by)\cdot \frac{\bx}{|\bx|}\chi_{ |\by|\le r_j(h)}(\by)\rd{\by}\rd{h}. \\
\end{split}\end{equation}

In the last expression, the first term is the positive contribution from the potential generated by disks. Notice that by \eqref{rhosmall}, 
\begin{equation}
R_0(h) \le Ch^{-1/2},\quad R_j(h) \le Ch^{-1/2},\quad \forall j \text{ flat}.
\end{equation}
Then Lemma \ref{lem_disk}, applied to those $h\ge \delta$, implies that 
\begin{equation}\begin{split}
& \int_0^\infty\Big[ \int \nabla W(\bx-\by)\cdot \frac{\bx}{|\bx|}\chi_{|\by|\le R_0(h)}(\by)\rd{\by} + \sum_{j \text{ flat}}\int \nabla W(\bx-\by)\cdot \frac{\bx}{|\bx|}\chi_{ |\by|\le R_j(h)}(\by)\rd{\by} \Big]\rd{h} \\
\ge & c\lambda(r + \delta^{-1/2})\int_\delta^\infty \Big[\min\{R_0(h),r\}^2 + \sum_{j \text{ flat}}\min\{R_j(h),r\}^2\Big]\rd{h} \\
\ge  & c\lambda(r + \delta^{-1/2})\int_{|\by|\le r} \Big(\max\{\rho^*(\by)-\delta,0\} + \max\{ \rho_\flat(\by)-\delta,0\}\Big)\rd{\by}.
\end{split}\end{equation}

For the negative contributions in \eqref{rhostfl} from $\{ |\by|\le r_j(h)\}$, by \eqref{lem_circle_1} for $s>r/2$, together with a trivial estimate (which follows from {\bf (A2)})
\begin{equation}
\left|\int \nabla W(\bx-\by)\cdot \frac{\bx}{|\bx|}\delta( |\by|- s)\rd{\by}\right| \le C\frac{s}{r},\quad \forall s\le \frac{r}{2},
\end{equation}
we obtain
\begin{equation}\begin{split}
& \left|\int_0^\infty \sum_{j \text{ flat}}\int \nabla W(\bx-\by)\cdot \frac{\bx}{|\bx|}\chi_{ |\by|\le r_j(h)}(\by)\rd{\by}\rd{h}\right| \\
\le & \int_0^\infty \sum_{j \text{ flat}}\int_0^{r_j(h)}\left|\int \nabla W(\bx-\by)\cdot \frac{\bx}{|\bx|}\delta( |\by|- s)\rd{\by}\right|\rd{s}\rd{h} \\
\le & C\int_0^\infty \sum_{j \text{ flat}}\Big(\int_0^{\min\{r_j(h),r/2\}}\frac{s}{r}\rd{s} + (r_j(h)-\min\{r_j(h),r/2\})\Big)\rd{h} \\
= & C\int_0^\infty \sum_{j \text{ flat}}\Big(\frac{\min\{r_j(h),r/2\}^2}{2r} + (r_j(h)-\min\{r_j(h),r/2\})\Big)\rd{h}. \\
\end{split}\end{equation}
If $r_j(h)\le r/2$, then the flatness of the interval $I_j$ gives
\begin{equation}
\frac{\min\{r_j(h),r/2\}^2}{2r} + (r_j(h)-\min\{r_j(h),r/2\}) = \frac{r_j(h)^2}{2r} \le \frac{C\varepsilon^2}{r} (R_j(h)^2-r_j(h)^2).
\end{equation}
If $r_j(h)> r/2$, then we also have
\begin{equation}\begin{split}
& \frac{\min\{r_j(h),r/2\}^2}{2r} + (r_j(h)-\min\{r_j(h),r/2\}) = \frac{r}{8} + r_j(h)-\frac{r}{2} \le r_j(h) \\
\le &  \frac{2r_j(h)^2}{r} \le \frac{C\varepsilon^2}{r} (R_j(h)^2-r_j(h)^2). 
\end{split}\end{equation}
Therefore
\begin{equation}\begin{split}
& \left|\int_0^\infty \sum_{j \text{ flat}}\int \nabla W(\bx-\by)\cdot \frac{\bx}{|\bx|}\chi_{ |\by|\le r_j(h)}(\by)\rd{\by}\rd{h}\right| \\
\le &  \frac{C\varepsilon^2}{r}\int_0^\infty \sum_{j \text{ flat}}(R_j(h)^2-r_j(h)^2)\rd{h} 
=   \frac{C\varepsilon^2}{r}\int \mu_\flat(\by)\rd{\by} 
\le   \frac{C\varepsilon^2}{r}. \\
\end{split}\end{equation}

Combining everything together, we have
\begin{equation}\begin{split}
& \int \nabla W(\bx-\by)\cdot \frac{\bx}{|\bx|}(\rho^*(\by) + \rho_\flat(\by) ) \rd{\by} \\
\ge & c\lambda(r + \delta^{-1/2}) \int_{|\by|\le r} \Big(\max\{\rho^*(\by)-\delta,0\} + \max\{ \rho_\flat(\by)-\delta,0\}\Big)\rd{\by} - \frac{C\varepsilon^2}{r}. \\
%\ge & c\lambda(r + \delta^{-1/2})\left[ \int_{|\by|\le r} \Big(\rho^*(\by)+\rho_\flat(\by)\Big)\rd{\by} - 2\pi \delta r^2\right]_{\ge0} - \frac{C\varepsilon^2}{r}. \\
\end{split}\end{equation}

\end{proof}

\begin{corollary}\label{cor_ctr}
Let $\cR_1$ be large enough such that Lemma \ref{lem_crho} holds with $c_\rho>0$, and $\cR_2>\cR_1$. Let $\varepsilon=\varepsilon(\cR_1,\cR_2,c_\rho)$ be small enough. Then
\begin{equation}
\int \nabla W(\bx-\by)\cdot \frac{\bx}{|\bx|}\Big(\rho^*(\by) + \rho_\flat(\by) + \rho_\sharp(\by)\chi_{|\by|\le \cR_1}(\by)\Big) \rd{\by} \ge c \lambda(\cR_2+C\cR_1),
\end{equation}
for any $\cR_1\le |\bx|\le \cR_2$.
\end{corollary}

\begin{proof}
First notice that by \eqref{lem_circle_2},
\begin{equation}\begin{split}
& \int \nabla W(\bx-\by)\cdot \frac{\bx}{|\bx|} \rho_\sharp(\by)\chi_{|\by|\le \cR_1}(\by)) \rd{\by} =  \int_0^{\cR_1} \int \nabla W(\bx-\by)\cdot \frac{\bx}{|\bx|}\delta(|\by|-s)\rd{\by}\rho_\sharp(s)\rd{s} \\
\ge & c\lambda(\cR_1+\cR_2)\int_0^{\cR_1}s\rho_\sharp(s)\rd{s} 
=  c\lambda(\cR_1+\cR_2) \int_{|\by|\le \cR_1} \rho_\sharp(\by) \rd{\by}.
\end{split}\end{equation}

Therefore, by Lemma \ref{lem_ctr}, for small $\delta>0$,
\begin{equation}\begin{split}
& \int \nabla W(\bx-\by)\cdot \frac{\bx}{|\bx|}\Big(\rho^*(\by) + \rho_\flat(\by) + \rho_\sharp(\by)\chi_{|\by|\le \cR_1}(\by)\Big) \rd{\by} \\
\ge & c\lambda(\cR_2 + \delta^{-1/2})\left[\int_{|\by|\le \cR_1} \rho_\sharp(\by) \rd{\by}
 +\int_{|\by|\le r} \Big(\max\{\rho^*(\by)-\delta,0\} + \max\{ \rho_\flat(\by)-\delta,0\}\Big)\rd{\by}\right] - \frac{C\varepsilon^2}{\cR_1} \\
\ge & c\lambda(\cR_2 + \delta^{-1/2})\left[\int_{|\by|\le \cR_1} \rho_\sharp(\by) \rd{\by}
 +\int_{|\by|\le \cR_1} \Big(\max\{\rho^*(\by)-\delta,0\} + \max\{ \rho_\flat(\by)-\delta,0\}\Big)\rd{\by}\right] - \frac{C\varepsilon^2}{\cR_1} \\
\ge & c\lambda(\cR_2 + \delta^{-1/2})\left[\int_{|\by|\le \cR_1} \rho_\sharp(\by) \rd{\by}
 + \int_{|\by|\le \cR_1} \Big(\rho^*(\by)+\rho_\flat(\by)\Big)\rd{\by} - 2\pi \delta \cR_1^2\right] - \frac{C\varepsilon^2}{\cR_1} \\
\ge & c\lambda(\cR_2 + \delta^{-1/2})(c_\rho - 2\pi \delta \cR_1^2) - \frac{C\varepsilon^2}{\cR_1}. \\
\end{split}\end{equation}
We choose $\delta = \frac{c_\rho}{4\pi\cR_1^2}$ and then $\varepsilon \le c\sqrt{c_\rho \lambda(\cR_2 + \delta^{-1/2})\cR_1}$ and obtain the conclusion.

\end{proof}

\section{Proof of Theorem \ref{thm_finite}: CSS and local clustering curves}

In this section we prove Theorem \ref{thm_finite}. We will prove Proposition \ref{prop_finite1} and Proposition \ref{prop_finite2}, using the CSS curves and the local clustering curve, respectively.

Following \eqref{rhoh1}, $\rho(\bx)$ can be written as
\begin{equation}
\rho(\bx) = \int_0^\infty \sum_{j}\chi_{\tI_j(h)}(\bx) \rd{h},\quad \tI_j(h):= \{\bx: |\bx|\in I_j(h)\}.
\end{equation}
Notice that $\tI_j(h)$ can be decomposed as
\begin{equation}\begin{split}
 \tI_j =  &\tI_j^{(1)} \cup \tI_j^{(2)} \\
:= & \Big\{\bx: |x_2|<r_j,  \sqrt{r_j^2-x_2^2} \le |x_1| \le \sqrt{R_j^2-x_2^2} \Big\} \cup \Big\{\bx: r_j\le |x_2| \le R_j,   |x_1| \le \sqrt{R_j^2-x_2^2} \Big\},
\end{split}\end{equation}
and 
\begin{equation}
\chi_{\tI_j}(\bx) = \left\{\begin{split}
& \chi_{\sqrt{r_j^2-x_2^2} \le x_1 \le \sqrt{R_j^2-x_2^2}}(x_1) + \chi_{-\sqrt{R_j^2-x_2^2} \le x_1 \le -\sqrt{r_j^2-x_2^2}}(x_1),\quad |x_2|<r_j \\
& \chi_{|x_1| \le \sqrt{R_j^2-x_2^2}}(x_1),\quad r_j\le |x_2| \le R_j
\end{split}.\right.
\end{equation}

We will define curves of density distributions $\rho_t$, starting from $\rho=\rho(t_0,\cdot)$ as the solution to \eqref{eq0} at some time. In particular, the starting distribution $\rho$ is radially-symmetric, non-negative, having total mass 1, and satisfies the $L^\infty$ bound (Lemma \ref{lem_reg1}).

\subsection{CSS1 curve}

Define the CSS1 curve $\rho_t$ (as in~\cite{CHVY}) for small $t>0$ by\footnote{This definition is problematic at those $x_2$ with $r_j^2-x_2^2$ being too small, but this does not affect later estimates since we only care about arbitrarily small $t>0$. This issue will be ignored in the rest of this paper.}
\begin{equation}
\tI_{j,t}^{(1)} = \Big\{\bx: |x_2|<r_j,  \sqrt{r_j^2-x_2^2}-t \le |x_1| \le \sqrt{R_j^2-x_2^2}-t \Big\},\quad \forall h,\quad \forall j\ge 1,
\end{equation}
and $\tI_j^{(2)}$ stays unmoved. See Figure \ref{fig_pic47} (left) for illustration.

\begin{proposition}\label{prop_CSS1}
For any $\cR_3>0$, the CSS1 curve $\rho_t$ satisfies
\begin{equation}
\frac{\rd}{\rd{t}}\Big|_{t=0}E[\rho_t] \le - c\varepsilon^2\lambda(2\cR_3)\Big(\int_{|\bx|\le \cR_3} \mu_\sharp(\bx)\rd{\bx}\Big)^2.
\end{equation}
\end{proposition}

It is important to notice that this estimate only gives energy dissipation from the sharp part $\mu_\sharp$, and becomes degenerate when $\varepsilon$ gets small.

\begin{proof}

The internal energy $\cS[\rho_t]$ is constant in $t$ because each $|\tI_{j,t}|$ is so, and (with the `$r_j^2-x_2^2$ being too small' issue ignored)
\begin{equation}
(m-1)\cS[\rho_t] = \int \rho_t(\bx)^m \rd{\bx} = \int_0^\infty \sum_{j} |\tI_{j,t}(h)| mh^{m-1} \rd{h}.
\end{equation}

We compute the interaction energy change along $\rho_t$:
\begin{equation}\begin{split}
\frac{\rd}{\rd{t}}\Big|_{t=0}\cI[\rho_t] = & \frac{1}{2}\int_0^\infty\int_0^\infty\sum_{j,j'} \frac{\rd}{\rd{t}}\Big|_{t=0}\cI[\chi_{\tI_{j,t}(h)},\chi_{\tI_{j',t}(h')}]\rd{h'}\rd{h}. \\
\end{split}\end{equation}
For fixed $h,h',j,j'$, 
\begin{equation}\label{CSS1est1}\begin{split}
& \frac{\rd}{\rd{t}}\Big|_{t=0}\cI[\chi_{\tI_{j,t}(h)},\chi_{\tI_{j',t}(h')}] \\
= & \frac{\rd}{\rd{t}}\Big|_{t=0}\int\int W(\bx-\by)\chi_{\tI_{j,t}(h)}(\bx)\chi_{\tI_{j',t}(h')}(\by)\rd{\by}\rd{\bx} \\
= & \frac{\rd}{\rd{t}}\Big|_{t=0}\int_{r_j\le |x_2| \le R_j}\int_{r_{j'}\le |y_2| \le R_{j'}} \cI_{x_2,y_2}[\chi_{|x_1| \le \sqrt{R_{j}^2-x_2^2}},\chi_{|y_1| \le \sqrt{R_{j'}^2-y_2^2}}]\rd{y_2}\rd{x_2} \\
& + \frac{\rd}{\rd{t}}\Big|_{t=0}\int_{|x_2|<r_j}\int_{r_{j'}\le |y_2| \le R_{j'}} \cI_{x_2,y_2}[\chi_{\sqrt{r_{j}^2-x_2^2}-t \le x_1 \le \sqrt{R_{j}^2-x_2^2}-t},\chi_{|y_1| \le \sqrt{R_{j'}^2-y_2^2}}]\rd{y_2}\rd{x_2} \\
& + \frac{\rd}{\rd{t}}\Big|_{t=0}\int_{|x_2|<r_j}\int_{|y_2|<r_{j'}} \cI_{x_2,y_2}[\chi_{\sqrt{r_{j}^2-x_2^2}-t \le x_1 \le \sqrt{R_{j}^2-x_2^2}-t},\chi_{\sqrt{r_{j'}^2-y_2^2}-t \le y_1 \le \sqrt{R_{j'}^2-y_2^2}-t}]\rd{y_2}\rd{x_2} \\
& + \frac{\rd}{\rd{t}}\Big|_{t=0}\int_{|x_2|<r_j}\int_{|y_2|<r_{j'}} \cI_{x_2,y_2}[\chi_{\sqrt{r_{j}^2-x_2^2}-t \le x_1 \le \sqrt{R_{j}^2-x_2^2}-t},\chi_{-\sqrt{R_{j'}^2-y_2^2}+t \le y_1 \le -\sqrt{r_{j'}^2-y_2^2}+t}]\rd{y_2}\rd{x_2} \\
& + \text{similar terms},
\end{split}\end{equation}
and 'similar terms' include those with $j,j'$ exchanged, or the signs of $x_1,y_1$ flipped. Notice that in the last expression of \eqref{CSS1est1}, the first term ($\tilde{I}_{j}^{(2)}$ with $\tilde{I}_{j'}^{(2)}$)  is zero because both intervals in $x_1$ and $y_1$ are not moving. The third term ($\tilde{I}_{j}^{(1)}$-right  with $\tilde{I}_{j'}^{(1)}$-right)  is zero because both intervals are moving with the same speed 1 towards left. The second ($\tilde{I}_{j}^{(1)}$-right with $\tilde{I}_{j'}^{(2)}$) and fourth ($\tilde{I}_{j}^{(1)}$-right with $\tilde{I}_{j'}^{(1)}$-left) terms are negative, by Lemma \ref{lem_CSSb2}.

To give quantitative estimate, we consider the fourth term in the last expression of \eqref{CSS1est1}, with $I_j(h)$ and $I_{j'}(h')$ being sharp. For fixed $x_2,y_2$ with $|x_2|\le r_j/2$ and $|y_2|\le r_{j'}/2$, the horizontal distance between closest endpoints of the two intervals appeared in $\cI_{x_2,y_2}$ is
\begin{equation}\begin{split}
 \sqrt{r_{j}^2-x_2^2} - \Big(- \sqrt{r_{j'}^2-y_2^2} \Big).
\ge  \frac{\sqrt{3}}{2} (r_j+r_{j'})
\end{split}\end{equation}
Combined with the vertical distance bound $|x_2-y_2| \le \frac{1}{2}(r_j+r_{j'})$, we can apply \eqref{lem_CSSb2_2} with 
\begin{equation}\begin{split}
& S=\Big[\sqrt{r_{j}^2-x_2^2}, \sqrt{R_{j}^2-x_2^2}\,\Big]\cap\Big\{x_1: \sqrt{x_1^2+x_2^2}\le \cR_3\Big\}, \\
& S'=\Big[-\sqrt{R_{j'}^2-y_2^2}, -\sqrt{r_{j'}^2-y_2^2}\,\Big]\cap\Big\{y_1: \sqrt{y_1^2+y_2^2}\le \cR_3\Big\},
\end{split}\end{equation}
and obtain
\begin{equation}\label{SSbound}\begin{split}
 \frac{\rd}{\rd{t}}\Big|_{t=0}\cI_{x_2,y_2}[\chi_{\sqrt{r_{j}^2-x_2^2}-t \le x_1 \le \sqrt{R_{j}^2-x_2^2}-t},\chi_{-\sqrt{R_{j'}^2-y_2^2}+t \le y_1 \le -\sqrt{r_{j'}^2-y_2^2}+t}] 
\le  - c\lambda(2\cR_3) \cdot |S|\cdot |S'|.
%\Big(\sqrt{R_{j}^2-x_2^2}-\sqrt{r_{j}^2-x_2^2}\Big)\cdot \Big(\sqrt{R_{j'}^2-y_2^2}-\sqrt{r_{j'}^2-y_2^2}\Big) \\
%\le & - c\lambda \cdot (R_j-r_j)(R_{j'}-r_{j'}) \\
\end{split}\end{equation}

Now we give a lower bound of $|S|$. If $R_j< \cR_3$, then $S=\Big[\sqrt{r_{j}^2-x_2^2}, \sqrt{R_{j}^2-x_2^2}\,\Big]$. Then $|S|\ge R_j-r_j$ for any $|x_2|\le r_j/2$.  Integrating in $|x_2|\le r_j/2$ gives
\begin{equation}
\int_{|x_2|\le r_j/2} |S|\rd{x_2} \ge r_j(R_j-r_j) \ge c\varepsilon |\tilde{I}_j|,
\end{equation}
where we used $R_j\le r_j/\varepsilon$ (from the sharpness of $I_j$) and $|\tI_j|=|\{\bx:|\bx|\in I_j\}| = \pi(R_j^2-r_j^2) \le (1+\frac{1}{\varepsilon})\pi (R_j-r_j)r_j$. 

If $R_j \ge \cR_3$, then one can replace $R_j$ by $\cR_3$ and conduct the same estimate, and obtain
\begin{equation}
\int_{|x_2|\le r_j/2} |S|\rd{x_2} \ge r_j(\cR_3-r_j)_{\ge0} \ge c\epsilon |\tilde{I}_j\cap\{|\bx|\le \cR_3\}|,
\end{equation}
which works for both cases.

Therefore we integrate \eqref{SSbound} in $|x_2|\le r_j/2$ and $|y_2|\le r_{j'}/2$, and obtain
\begin{equation}\begin{split}
& \frac{\rd}{\rd{t}}\Big|_{t=0}\int_{|x_2|<r_j}\int_{|y_2|<r_{j'}} \cI_{x_2,y_2}[\chi_{\sqrt{r_{j}^2-x_2^2}-t \le x_1 \le \sqrt{R_{j}^2-x_2^2}-t},\chi_{-\sqrt{R_{j'}^2-y_2^2}+t \le y_1 \le -\sqrt{r_{j'}^2-y_2^2}+t}]\rd{y_2}\rd{x_2} \\
\le & - c\lambda(2\cR_3) \cdot \varepsilon^2 \cdot |\tI_j\cap\{|\bx|\le \cR_3\}| \cdot |\tI_j'\cap\{|\bx|\le \cR_3\}|. \\
\end{split}\end{equation}
Then summing in $j,j'$ and integrating in $h,h'$ gives the conclusion.

\end{proof}

\begin{proposition}\label{prop_CSS1c}
For any $\cR_3>0$,
\begin{equation}
\int_0^\infty \Big(\int_{|\bx|\le \cR_3} \mu_\sharp(t,\bx)\rd{\bx}\Big)^4  \rd{t} \le C\frac{1}{\varepsilon^4\lambda(2\cR_3)^2}.
\end{equation}
\end{proposition}

\begin{proof}
Lemma \ref{lem_cost} implies that the cost of the CSS1 curve is bound by 1. Applying Lemma \ref{lem_basic} to the CSS1 curve with Proposition \ref{prop_CSS1}, we get
\begin{equation}
\frac{\rd}{\rd{t}}E[\rho(t,\cdot)] \le - c\varepsilon^4\lambda(2\cR_3)^2\Big(\int_{|\bx|\le \cR_3} \mu_\sharp(t,\bx)\rd{\bx}\Big)^4.
\end{equation}
Integrating in $t\in[0,\infty)$, the conclusion follows from the lower bound of $E$ (Lemma \ref{lem_Em}). 
\end{proof}

\subsection{CSS2 curve}

Fix $\cR_1\ge 1$ large enough, such that Lemma \ref{lem_crho} applies. 

Fix  $\cR_2>4\cR_1$. If $I_j(h)=[r_j,R_j],\,j\ge 1$ is sharp and contains $\cR_2$, then we \emph{split} $I_j$ into
\begin{equation}
I_j = I_{(j,1)}\cup I_{(j,2)} := [r_j,\cR_2]\cup [\cR_2,R_j].
\end{equation}
Notice that $I_{(j,1)}$ and $I_{(j,2)}$ are also sharp. In this subsection, all the indices $j$ will refer to this `after-split' version: $(j,1)$ and $(j,2)$ are different indices. Under this notation, all sharp intervals $I_j$ does not contain $\cR_2$ as an interior point.

\begin{lemma}\label{lem_rjst}
For fixed $I_j = [r_j,R_j]\subset (0,\cR_2]$ with $\frac{1}{2}(r_j+R_j) > \cR_1$, there exists a unique $0< r_j^*\le r_j$, such that for $|x_2|\le r_j$,
\begin{equation}\label{lem_rjst_1}
\frac{1}{2}\Big(\sqrt{r_j^2-x_2^2}+\sqrt{R_j^2-x_2^2}\Big)\in [\cR_1,\cR_2] \quad\Leftrightarrow\quad |x_2|\le r_j^*.
\end{equation}
\end{lemma}
For notational convenience, $r_j^*$ will be understood as $-1$ for those $I_j$ with $\frac{1}{2}(r_j+R_j) < \cR_1$, which agrees with \eqref{lem_rjst_1}.

\begin{proof}
It is clear that $\frac{1}{2}\Big(\sqrt{r_j^2-x_2^2}+\sqrt{R_j^2-x_2^2}\Big)\le \cR_2$ always holds. Notice that $x_2\mapsto \frac{1}{2}\Big(\sqrt{r_j^2-x_2^2}+\sqrt{R_j^2-x_2^2}\Big)$ is a decreasing function on $[0,r_2]$, whose value is $\frac{1}{2}(r_j+R_j)  > \cR_1$ at $x_2=0$. Then the conclusion follows.
\end{proof}

Define the CSS2 curve $\rho_t$ by 
\begin{equation}\label{CSS2}\begin{split}
\tI_{j,t}^{(1)}\cap & (\mathbb{R}\times \{x_2\}) = \left\{\begin{split}
& \sqrt{r_j^2-x_2^2}-t \le |x_1| \le \sqrt{R_j^2-x_2^2}-t,\quad |x_2|\le r_j^*\\
& \tI_{j}^{(1)}\cap  (\mathbb{R}\times \{x_2\}),\quad \text{otherwise}
\end{split}\right. \\
&  \forall h,\quad\forall j\ge 1,\text{ sharp},\, R_j \le \cR_2,
\end{split}\end{equation}
where $r_j^*$ is given as in Lemma \ref{lem_rjst}, and everything else ($j=0$, $j\ge1$ flat, or $R_j> \cR_2$) stays unmoved. This definition does not cause collision of slices at the split point $|\bx|=\cR_2$ (i.e., between $I_{(j,1)}$ and $I_{(j,2)}$ in the original notation), because the condition $R_j \le \cR_2$ guarantees that any radial interval like $I_{(j,2)}$ does not move. See Figure \ref{fig_pic5} for illustration.

\begin{figure}
\begin{center}
  \includegraphics[width=.6\linewidth]{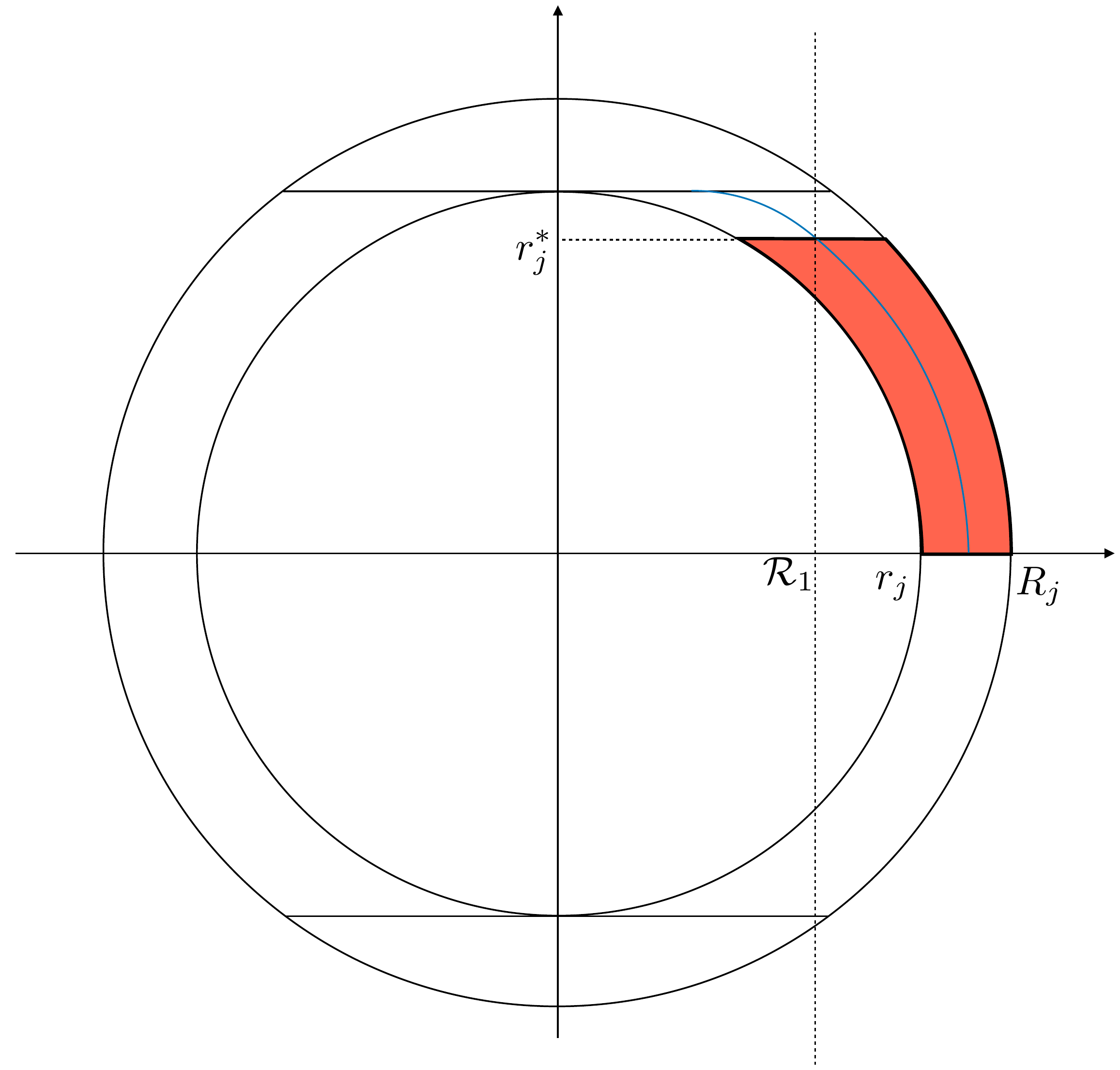}  
  \caption{The CSS2 curve, applied on a sharp $\tI_j$, with $R_j\le \cR_2$. $r_j^*$ is determined by the condition that the midpoint of a slice $\Big[\sqrt{r_j^2-x_2^2},\sqrt{R_j^2-x_2^2}\,\Big]$ (the blue  curve) is equal to $\cR_1$. Only the red part (and its symmetric counterparts) are moving (towards the center) in the CSS2 curve.}
\label{fig_pic5}
\end{center}
\end{figure}

\begin{proposition}\label{prop_CSS2}
Assume $\varepsilon=\varepsilon(\cR_1,\cR_2)$ is small enough. The CSS2 curve $\rho_t$ satisfies
\begin{equation}\label{prop_CSS2_1}\begin{split}
& \frac{\rd}{\rd{t}}\Big|_{t=0}E[\rho_t] \\
\le & - \Big(c-C\int_{\cR_2\le |\by| \le \cR_3}\mu_\sharp(\by)\rd{\by}\Big)\cdot\int \sum_{j\ge 1,\,\textnormal{sharp}, R_j\le \cR_2}\int_{|x_2|\le r_j^*}\Big(\sqrt{R_{j}^2-x_2^2}-\sqrt{r_{j}^2-x_2^2}\Big)\rd{x_2} \rd{h},
\end{split}\end{equation}
for some large $\cR_3=\cR_3(\cR_1,\cR_2)$, where $c$ and $C$ may depend on $\cR_1$ and $\cR_2$.
\end{proposition}
Notice that the last integral in $h$ is exactly the total moving mass (up to a constant multiple).

We start by proving the following lemma concerning the area of the moving mass:
\begin{lemma}\label{lem_x2}
Fix $0<\kappa<1$. For any $0<r_1^*\le r_1<R_1$,
\begin{equation}\label{lem_x2_1}
\int_{|x_2|\le \kappa r_1^*} \Big(\sqrt{R_1^2-x_2^2}-\sqrt{r_1^2-x_2^2}\Big)\rd{x_2} \ge c_\kappa\int_{|x_2|\le r_1^*} \Big(\sqrt{R_1^2-x_2^2}-\sqrt{r_1^2-x_2^2}\Big)\rd{x_2}.
\end{equation}
\end{lemma}
Notice that the constant $c_\kappa$ does not depend on the sharpness $r_1/R_1$.

\begin{proof}

First notice that
\begin{equation}\begin{split}
\sqrt{R_1^2-x_2^2}-\sqrt{r_1^2-x_2^2} =  \frac{R_1^2-r_1^2}{\sqrt{R_1^2-x_2^2}+\sqrt{r_1^2-x_2^2}},
\end{split}\end{equation}
is an increasing function in $x_2$, for $0\le x_2\le r_1$. Thus
\begin{equation}\begin{split}
\sqrt{R_1^2-x_2^2}-\sqrt{r_1^2-x_2^2} \ge R_1-r_1.
\end{split}\end{equation}

If $r_1^* \le R_1/2$, then for any $0\le x_2\le r_1^*$,
\begin{equation}\begin{split}
\sqrt{R_1^2-x_2^2}-\sqrt{r_1^2-x_2^2}  \le & \frac{R_1^2-r_1^2}{\sqrt{R_1^2-\min\{R_1/2,r_1\}^2}+\sqrt{r_1^2-\min\{R_1/2,r_1\}^2}} \\
\le & C\frac{R_1^2-r_1^2}{R_1} \le C(R_1-r_1).
\end{split}\end{equation}
Then the conclusion follows since the  integrand in \eqref{lem_x2_1} is comparable for all possibly used $x_2$.

If $r_1^* > R_1/2$, then
\begin{equation}
\int_{|x_2|\le \kappa r_1^*} \Big(\sqrt{R_1^2-x_2^2}-\sqrt{r_1^2-x_2^2}\Big)\rd{x_2} \ge 2\kappa r_1^* (R_1-r_1) \ge \kappa R_1(R_1-r_1),
\end{equation}
and
\begin{equation}
\int_{|x_2|\le r_1^*} \Big(\sqrt{R_1^2-x_2^2}-\sqrt{r_1^2-x_2^2}\Big)\rd{x_2} \le C|\{\bx: r_1\le |\bx|\le R_1\}| \le CR_1(R_1-r_1).
\end{equation}
Then the conclusion follows.

\end{proof}

\begin{proof}[Proof of Proposition \ref{prop_CSS2}]

Similar to the CSS1 curve, the internal energy $\cS[\rho_t]$ is constant in $t$.

To estimate the interaction energy change along $\rho_t$,
\begin{equation}\label{CSS2est}\begin{split}
\frac{\rd}{\rd{t}}\Big|_{t=0}\cI[\rho_t] = & \frac{\rd}{\rd{t}}\Big|_{t=0}\cI[\rho, \rho_t] \\
= & \frac{\rd}{\rd{t}}\Big|_{t=0}\int \sum_{j\ge 1,\,\text{sharp}, R_j\le \cR_2} \cI[\rho^*(\by) + \rho_\flat(\by),\chi_{\tI_{j,t}(h)}] \rd{h} \\
& + \frac{1}{2}\frac{\rd}{\rd{t}}\Big|_{t=0}\int\int \sum_{j'\ge 1,\,\text{sharp}, R_{j'}\le \cR_2} \sum_{j\ge 1,\,\text{sharp}, R_j\le \cR_2} \cI[\chi_{\tI_{j',t}(h')},\chi_{\tI_{j,t}(h)}] \rd{h}\rd{h'} \\
& + \frac{\rd}{\rd{t}}\Big|_{t=0}\int \sum_{j\ge 1,\,\text{sharp}, R_j\le \cR_2} \cI[\mu_\sharp(\by)\chi_{|\by|\ge \cR_2}(\by),\chi_{\tI_{j,t}(h)}] \rd{h},
\end{split}\end{equation}
where the symmetry between $(h,j)$ and $(h',j')$ is used in the second term on the RHS. Now we  discuss the contribution from the three terms separately.

{\bf STEP 1}: radially-decreasing and flat parts: the first term on the RHS of \eqref{CSS2est}, for fixed $(h,j)$. 

First notice that this term is possibly nonzero only when 
\begin{equation}\label{CSS2step1}
(R_j+r_j)/2 \ge \cR_1,
\end{equation}
 since otherwise $\tI_{j,t}$ is constant in $t$.

\begin{equation}\begin{split}
& \frac{\rd}{\rd{t}}\Big|_{t=0}\cI[\rho^*(\by) + \rho_\flat(\by),\chi_{\tI_{j,t}}] \\
= & \frac{\rd}{\rd{t}}\Big|_{t=0}\int \int W(\bx-\by)(\rho^*(\by) + \rho_\flat(\by))\rd{\by} \cdot \chi_{\tI_{j,t}}(\bx)\rd{\bx} \\
= & 2\frac{\rd}{\rd{t}}\Big|_{t=0}\int_{|x_2|\le r_j^*}\int \int W(\bx-\by)(\rho^*(\by) + \rho_\flat(\by))\rd{\by} \cdot \chi_{\sqrt{r_j^2-x_2^2}-t \le x_1 \le \sqrt{R_j^2-x_2^2}-t}(x_1) \rd{x_1}\rd{x_2}  \\
= & -2\int_{|x_2|\le r_j^*}\int \int \nabla W(\bx-\by)\cdot \vec{e}_1(\rho^*(\by) + \rho_\flat(\by))\rd{\by} \cdot \chi_{\sqrt{r_j^2-x_2^2} \le x_1 \le \sqrt{R_j^2-x_2^2}}(x_1) \rd{x_1}\rd{x_2} , \\
\end{split}\end{equation}
where the second inequality follows from left-right symmetry. Here the integral domain $|x_2|\le r_j^*$   are exactly those moving slices, by \eqref{CSS2}.

Notice that the vector $\int \nabla W(\bx-\by)(\rho^*(\by) + \rho_\flat(\by))\rd{\by}$ is parallel to $\frac{\bx}{|\bx|}$ by radial symmetry. Combined with Lemma \ref{lem_ctr}, we have
\begin{equation}\begin{split}
& \int \nabla W(\bx-\by)\cdot \vec{e}_1(\rho^*(\by) + \rho_\flat(\by))\rd{\by} \\ \ge & \left[c\lambda(r + \delta^{-1/2})\Big( \int_{|\by|\le r} (\rho^*(\by)+\rho_\flat(\by))\rd{\by} - 2\pi \delta r^2\Big)_{\ge 0} - \frac{C\varepsilon^2}{r}\right]\cdot \frac{x_1}{r},
\end{split}\end{equation}
for any $\delta>0$ and $x_1>0$, where we denote $f_{\ge 0} := \max\{f,0\}$ for any quantity $f$. Therefore
\begin{equation}\begin{split}\label{CSS2est0}
& \frac{\rd}{\rd{t}}\Big|_{t=0}\cI[\rho^*(\by) + \rho_\flat(\by),\chi_{\tI_{j,t}}] \\
\le & -c\lambda(\cR_2 + \delta^{-1/2})\int_{|x_2|\le r_j^*} \int_{\sqrt{r_j^2-x_2^2}}^{\sqrt{R_j^2-x_2^2}}\Big( \int_{|\by|\le r} (\rho^*(\by)+\rho_\flat(\by))\rd{\by} - 2\pi \delta r^2\Big)_{\ge 0} \cdot \frac{x_1}{r}\rd{x_1}\rd{x_2}  \\
& + C\varepsilon^2\int_{|x_2|\le r_j^*} \int_{\sqrt{r_j^2-x_2^2}}^{\sqrt{R_j^2-x_2^2}} \frac{x_1}{r^2}\rd{x_1}\rd{x_2}.
\end{split}\end{equation}
Notice that for $|x_2|\le r_j/2$,
\begin{equation}
x_1\ge \sqrt{r_j^2-(r_j/2)^2} = \frac{\sqrt{3}}{2}r_j\quad \Rightarrow \quad \frac{x_1}{r} = \frac{x_1}{\sqrt{x_1^2+x_2^2}} \ge c.
\end{equation}
 Therefore, the first term on the RHS of \eqref{CSS2est0} can be estimated by
\begin{equation}\begin{split}
& \int_{|x_2|\le r_j^*} \int_{\sqrt{r_j^2-x_2^2}}^{\sqrt{R_j^2-x_2^2}}\Big( \int_{|\by|\le r} (\rho^*(\by)+\rho_\flat(\by))\rd{\by} - 2\pi \delta r^2\Big)_{\ge 0} \cdot \frac{x_1}{r}\rd{x_1}\rd{x_2}  \\
\ge & c\int_{|x_2|\le r_j^*/2} \int_{(\sqrt{r_j^2-x_2^2}+\sqrt{R_j^2-x_2^2})/2}^{\sqrt{R_j^2-x_2^2}}\Big( \int_{|\by|\le r} (\rho^*(\by)+\rho_\flat(\by))\rd{\by} - 2\pi \delta r^2\Big)_{\ge 0} \rd{x_1}\rd{x_2}  \\
\ge & c\Big( \int_{|\by|\le \cR_1} (\rho^*(\by)+\rho_\flat(\by))\rd{\by} - 2\pi \delta \cR_2^2\Big)\cdot \int_{|x_2|\le r_j^*/2}\Big(\sqrt{R_{j}^2-x_2^2}-\sqrt{r_{j}^2-x_2^2}\Big)\rd{x_2}  \\
\ge & c\Big( \int_{|\by|\le \cR_1} (\rho^*(\by)+\rho_\flat(\by))\rd{\by} - 2\pi \delta \cR_2^2\Big)\cdot \int_{|x_2|\le r_j^*} \Big(\sqrt{R_{j}^2-x_2^2}-\sqrt{r_{j}^2-x_2^2}\Big) \rd{x_2},  \\
\end{split}\end{equation}
using $r_j^*\le r_j$ in the first inequality, $r\ge \Big(\sqrt{r_j^2-x_2^2}+\sqrt{R_j^2-x_2^2}\Big)/2 \ge \cR_1$ and $r\le \cR_2$ in the second inequality, and Lemma \ref{lem_x2} (with $\kappa=1/2$) in the last inequality. The second term on the RHS of \eqref{CSS2est0} can be estimated by
\begin{equation}\begin{split}
& \varepsilon^2\int_{|x_2|\le r_j^*} \int_{\sqrt{r_j^2-x_2^2}}^{\sqrt{R_j^2-x_2^2}} \frac{x_1}{r^2}\rd{x_1}\rd{x_2} \\
\le &\varepsilon^2 \int_{|x_2|\le r_j^*} \int_{\sqrt{r_j^2-x_2^2}}^{\sqrt{R_j^2-x_2^2}} \frac{1}{r}\rd{x_1}\rd{x_2} \\
\le &\varepsilon^2 \frac{1}{r_j}\int_{|x_2|\le r_j^*} \Big(\sqrt{R_{j}^2-x_2^2}-\sqrt{r_{j}^2-x_2^2}\Big)\rd{x_2} \\
\le & \varepsilon\cR_1\int_{|x_2|\le r_j^*}\Big(\sqrt{R_{j}^2-x_2^2}-\sqrt{r_{j}^2-x_2^2}\Big)\rd{x_2}, \\
\end{split}\end{equation}
where the last inequality uses the flatness of $I_j$, and $R_j\ge \cR_1$ (from \eqref{CSS2step1}).

{\bf STEP 2}: sharp parts inside $\cR_2$: the second term on the RHS of \eqref{CSS2est}, for fixed $(h,j)$ and $(h',j')$.

%For those $I_{(j',1)}$ with $R_{j'} > \cR_2$ and those $I_{j'}$ with $R_{j'} \le \cR_2$, they are exactly the indices appearing in the summation-integration over $(h,j)$. Therefore we are able to symmetrize $(h',j')$ and $(h,j)$ and consider instead
%\begin{equation}
%\frac{\rd}{\rd{t}}\Big|_{t=0}\cI[\chi_{\tI_{j',t}},\chi_{\tI_{j,t}}] 
%\end{equation} 
%where both $j$ and $j'$ satisfy $j\ge 1, \text{ after-split}, R_j\le \cR_2$.

Similar to the CSS1 analysis, 
%we symmetrize $(h,j)$ and $(h',j')$ and expand as in \eqref{CSS1est1} to obtain
%\begin{equation}\begin{split}
%&  \frac{\rd}{\rd{t}}\Big|_{t=0}\int\int \sum_{j'\ge 1,\,\text{sharp}, R_{j'}\le \cR_2} \sum_{j\ge 1,\,\text{sharp}, R_j\le \cR_2} \cI[\chi_{\tI_{j'}},\chi_{\tI_{j,t}}] \rd{h}\rd{h'} \\
%= &  \frac{1}{2}\frac{\rd}{\rd{t}}\Big|_{t=0}\int\int \sum_{j'\ge 1,\,\text{sharp}, R_{j'}\le \cR_2} \sum_{j\ge 1,\,\text{sharp}, R_j\le \cR_2} \cI[\chi_{\tI_{j',t}},\chi_{\tI_{j,t}}] \rd{h}\rd{h'} \\
%\end{split}\end{equation}
%with
\begin{equation}\label{CSS2est1}\begin{split}
& \frac{\rd}{\rd{t}}\Big|_{t=0}\cI[\chi_{\tI_{j,t}(h)},\chi_{\tI_{j',t}(h')}] \\
&=  \frac{\rd}{\rd{t}}\Big|_{t=0}\int_{|x_2|<r_j^*}\int_{r_{j'}\le |y_2| \le R_{j'}} \cI_{x_2,y_2}[\chi_{\sqrt{r_{j}^2-x_2^2}-t \le x_1 \le \sqrt{R_{j}^2-x_2^2}-t},\chi_{|y_1| \le \sqrt{R_{j'}^2-y_2^2}}]\rd{y_2}\rd{x_2} \\
& + \frac{\rd}{\rd{t}}\Big|_{t=0}\int_{|x_2|<r_j^*}\int_{r_{j'}^*\le |y_2| \le r_{j'}} \cI_{x_2,y_2}[\chi_{\sqrt{r_{j}^2-x_2^2}-t \le x_1 \le \sqrt{R_{j}^2-x_2^2}-t},\chi_{\sqrt{r_{j'}^2-y_2^2} \le y_1 \le \sqrt{R_{j'}^2-y_2^2}}]\rd{y_2}\rd{x_2} \\
& + \frac{\rd}{\rd{t}}\Big|_{t=0}\int_{|x_2|<r_j^*}\int_{r_{j'}^*\le |y_2| \le r_{j'}} \cI_{x_2,y_2}[\chi_{\sqrt{r_{j}^2-x_2^2}-t \le x_1 \le \sqrt{R_{j}^2-x_2^2}-t},\chi_{-\sqrt{R_{j'}^2-y_2^2} \le y_1 \le -\sqrt{r_{j'}^2-y_2^2}}]\rd{y_2}\rd{x_2} \\
& + \frac{\rd}{\rd{t}}\Big|_{t=0}\int_{|x_2|<r_j^*}\int_{|y_2|<r_{j'}^*} \cI_{x_2,y_2}[\chi_{\sqrt{r_{j}^2-x_2^2}-t \le x_1 \le \sqrt{R_{j}^2-x_2^2}-t},\chi_{\sqrt{r_{j'}^2-y_2^2}-t \le y_1 \le \sqrt{R_{j'}^2-y_2^2}-t}]\rd{y_2}\rd{x_2} \\
& + \frac{\rd}{\rd{t}}\Big|_{t=0}\int_{|x_2|<r_j^*}\int_{|y_2|<r_{j'}^*} \cI_{x_2,y_2}[\chi_{\sqrt{r_{j}^2-x_2^2}-t \le x_1 \le \sqrt{R_{j}^2-x_2^2}-t},\chi_{-\sqrt{R_{j'}^2-y_2^2}+t \le y_1 \le -\sqrt{r_{j'}^2-y_2^2}+t}]\rd{y_2}\rd{x_2} \\
& + \text{similar terms}.
\end{split}\end{equation}
See Figure \ref{fig_pic8} for illustration.

\begin{figure}
\begin{center}
  \includegraphics[width=.5\linewidth]{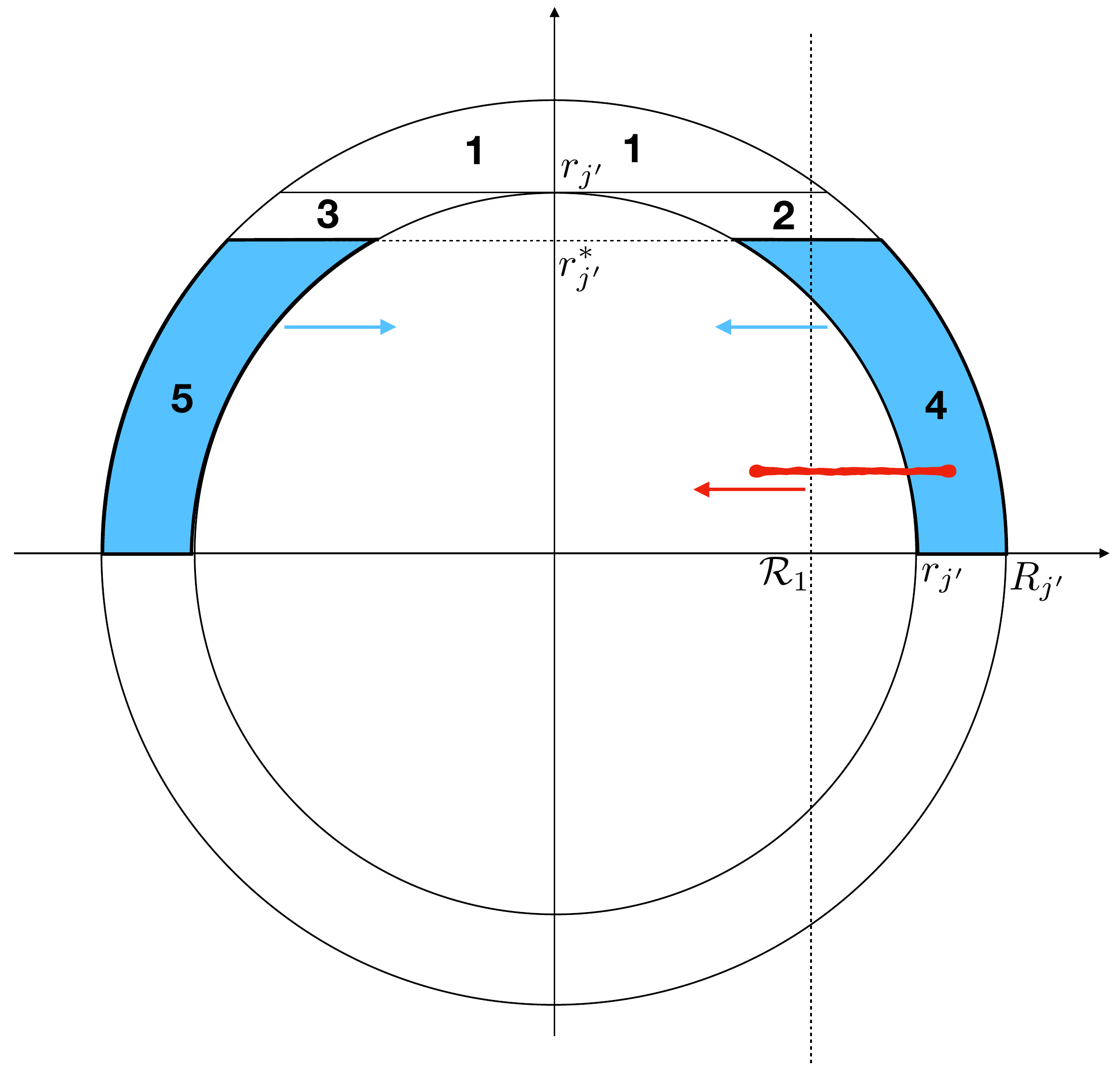}  
  \caption{The decomposition \eqref{CSS2est1}: when a (red) slice $x_1\in \Big[\sqrt{r_{j}^2-x_2^2}, \sqrt{R_{j}^2-x_2^2}\,\Big]$ is moving towards the center, the decay of the symmetric interaction energy against a sharp $\tI_{j'}$ with $R_{j'}\le \cR_2$. The five terms in \eqref{CSS2est1} correspond to the interaction energy between the red slice and the region 1 to 5 respectively. }
\label{fig_pic8}
\end{center}
\end{figure}

Every $\cI_{x_2,y_2}$ appeared above is non-positive, since the moving (towards left) interval $x_1\in \Big[\sqrt{r_{j}^2-x_2^2} , \sqrt{R_{j}^2-x_2^2}\,\Big]$ has center in $[\cR_1,\cR_2]$, and the $y_1$ interval either is  moving with the same speed (the fourth term, being zero) or has center in $(-\infty,\cR_1)$, staying or moving towards right.

In the first, third, fifth terms, the $y_1$ intervals has center $\kc \in (-\infty,0]$. If $|y_2| \le \cR_1/4$ and $|x_2| \le r_j/4$, then the horizontal difference of interval centers 
\begin{equation}\begin{split}
 \frac{1}{2}\Big(\sqrt{R_{j}^2-x_2^2}+\sqrt{r_{j}^2-x_2^2}\Big) - \kc 
\ge & \frac{1}{2}\Big(\sqrt{R_{j}^2-x_2^2}+\sqrt{r_{j}^2-x_2^2}\Big) 
\ge \max\{\sqrt{r_{j}^2-x_2^2},\cR_1\} \\
\ge & \max\{\frac{\sqrt{3}}{2} r_j,\cR_1\} 
\ge  \frac{\sqrt{3}}{2}\max\{ r_j,\cR_1\},
\end{split}\end{equation}
and the vertical difference
\begin{equation}
|x_2-y_2| \le \frac{1}{4}(\cR_1 + r_j) \le \frac{1}{2}\max\{ r_j,\cR_1\}.
\end{equation}
Therefore we can apply \eqref{lem_CSSb2_1} to obtain
\begin{equation}\begin{split}
& \frac{\rd}{\rd{t}}\Big|_{t=0}\cI_{x_2,y_2}[\chi_{\sqrt{r_{j}^2-x_2^2}-t \le x_1 
\le  \sqrt{R_{j}^2-x_2^2}-t},\chi_{|y_1| \le \sqrt{R_{j'}^2-y_2^2}}] \\ \le & -c \lambda(2\cR_2) \cdot  \min\Big\{\sqrt{R_{j'}^2-y_2^2}, \cR_1\Big\} \cdot \Big(\sqrt{R_{j}^2-x_2^2}-\sqrt{r_{j}^2-x_2^2}\Big),
\end{split}\end{equation}
for the first term in \eqref{CSS2est1}, and similar holds for the third and fifth terms. Combining them together and integrating in $x_2,y_2$, we obtain
\begin{equation}\begin{split}
& \frac{\rd}{\rd{t}}\Big|_{t=0}\cI[\chi_{\tI_{j,t}(h)},\chi_{\tI_{j',t}(h')}] \\
\le &  -c \lambda(2\cR_2) \cdot \int_{|y_2|\le \cR_1/4}\min\Big\{\sqrt{R_{j'}^2-y_2^2}-\sqrt{\max\{R_{j'}^2-y_2^2,0\}}, \cR_1\Big\}\rd{y_2} \\
& \cdot \int_{|x_2|\le \min\{r_j^*,r_j/4\}}\Big(\sqrt{R_{j}^2-x_2^2}-\sqrt{r_{j}^2-x_2^2}\Big) \rd{x_2} \\
\le &  -c \lambda(2\cR_2) \cdot |\tI_{j'}\cap \{|y_2|\le \cR_1/4\}| \cdot \int_{|x_2|\le r_j^*}\Big(\sqrt{R_{j}^2-x_2^2}-\sqrt{r_{j}^2-x_2^2}\Big) \rd{x_2},
\end{split}\end{equation}
by Lemma \ref{lem_x2} (with $\kappa=1/4$). Also notice the lower bound
\begin{equation}
 \Big|\tI_{j'}\cap \{|y_2|\le \cR_1/4\}\Big| \ge  \Big|\tI_{j'}\cap \{|\sin\theta| \le \cR_1/(4\cR_2)\}\Big| \ge c\frac{\cR_1}{\cR_2}|\tI_{j'}|,
\end{equation}
since the radial set $\tI_{j'}\subset\{\by: |\by|\le \cR_2\}$.

{\bf STEP 3}: sharp parts outside $\cR_2$: the third term on the RHS of \eqref{CSS2est}, for fixed $(h,j)$.

We use Lemma \ref{lem_circle} to obtain (noticing that $\cR_2\ge 1$ by assumption)
\begin{equation}\label{WR231}
\left| \int \nabla W(\bx-\by)\mu_\sharp(\by)\chi_{\cR_2\le |\by|\le\cR_3}(\by)\rd{\by} \right| \le C\int_{\cR_2}^{\cR_3}\mu_\sharp(r)\rd{r},
\end{equation}
where $\cR_3>\cR_2$ is to be determined. Also, 
\begin{equation}\label{WR232}
\left| \int \nabla W(\bx-\by)\mu_\sharp(\by)\chi_{|\by|\ge\cR_3}(\by)\rd{\by} \right| \le C\frac{1}{\cR_3-\cR_2}\int\mu_\sharp(\by)\chi_{|\by|\ge\cR_3}(\by)\rd{\by} \le C\frac{1}{\cR_3-\cR_2},
\end{equation}
for any $|\bx|\le \cR_2$, by {\bf (A2)}, since $|\bx-\by|\ge \cR_3-\cR_2$ for the LHS integral.

Therefore, since all the slices in $\tI_{j,t}$ are moving at speed at most 1,
\begin{equation}\begin{split}
& \left|\frac{\rd}{\rd{t}}\Big|_{t=0}\cI[\mu_\sharp(\by)\chi_{|\by|\ge\cR_2}(\by),\chi_{\tI_{j,t}}]\right| \\
\le & C\Big(\int_{\cR_2}^{\cR_3}\mu_\sharp(r)\rd{r} + \frac{1}{\cR_3-\cR_2} \Big) \int_{|x_2|\le r_j^*}\Big(\sqrt{R_{j}^2-x_2^2}-\sqrt{r_{j}^2-x_2^2}\Big)\rd{x_2} \\
\le & C\Big(\frac{1}{\cR_2}\int_{\cR_2\le |\by| \le \cR_3}\mu_\sharp(\by)\rd{\by} + \frac{1}{\cR_3-\cR_2} \Big) \int_{|x_2|\le r_j^*}\Big(\sqrt{R_{j}^2-x_2^2}-\sqrt{r_{j}^2-x_2^2}\Big)\rd{x_2}. \\
\end{split}\end{equation}

Combining the three steps together, summing in $j$ and integrating in $h$, we obtain
\begin{equation}\begin{split}
& \frac{\rd}{\rd{t}}\Big|_{t=0}\cI[\rho_t] \\
\le & \Big[-c\lambda(\cR_2 + \delta^{-1/2})\Big( \int_{|\by|\le \cR_1} (\rho^*(\by)+\mu_\flat(\by))\rd{\by} - 2\pi \delta \cR_2^2\Big) + C\varepsilon \cR_1 \\
& -c \lambda(2\cR_2)\frac{\cR_1}{\cR_2}\int_{|\by|\le \cR_2} \mu_\sharp(\by) \rd{\by} + C\Big(\frac{1}{\cR_2}\int_{\cR_2\le |\by| \le \cR_3}\mu_\sharp(\by)\rd{\by} + \frac{1}{\cR_3-\cR_2} \Big) \Big] \\
& \cdot \int \sum_{j\ge 1,\,\text{sharp}, R_j\le \cR_2}\int_{|x_2|\le r_j^*}\Big(\sqrt{R_{j}^2-x_2^2}-\sqrt{r_{j}^2-x_2^2}\Big)\rd{x_2} \rd{h}.
\end{split}\end{equation}
Notice that the last integral is exactly the total moving mass (up to a constant multiple). Finally we choose the parameters $\delta,\varepsilon,\cR_3$.  We first notice that 
\begin{equation}
\int_{|\by|\le \cR_1} (\rho^*(\by)+\mu_\flat(\by))\rd{\by} + \int_{|\by|\le \cR_2} \mu_\sharp(\by) \rd{\by} \ge \int_{|\by|\le \cR_1} \rho(\by)\rd{\by} \ge c_\rho,
\end{equation}
by Lemma \ref{lem_crho}. Therefore, by choosing $\delta=c\cR_1\cR_2^{-3}$ small enough, we may absorb the bad term $2\pi \delta \cR_2^2$. Then choosing $\varepsilon$ small enough and $\cR_3$ large enough, we may absorb the bad terms $C\varepsilon \cR_1$ and $C\frac{1}{\cR_3-\cR_2}$ respectively. Then the conclusion follows.

\end{proof}

%\begin{lemma}\label{lem_CSS2cost}
%The cost of the CSS2 curve $\rho_t$ is bounded by 
%\begin{equation}
%\int |\bv_t|^2\rho_t \rd{\bx}|_{t=0} \le \int \sum_{j\ge 1,\,\textnormal{sharp}, R_j\le \cR_2}\int_{|x_2|\le r_j^*}\Big(\sqrt{R_{j}^2-x_2^2}-\sqrt{r_{j}^2-x_2^2}\Big)\rd{x_2} \rd{h} 
%\end{equation}
%where $\bv_t$ is the velocity field satisfying \eqref{lem_basic_1}.
%\end{lemma}

%\begin{proposition}
%\begin{equation}
%\int_0^\infty \int_{|\bx|\in [4\cR_1,\cR_2]} \mu_\sharp(t,\bx) \rd{\bx} \rd{t} < \infty
%\end{equation}
%\end{proposition}

\begin{proof}[Proof of Proposition \ref{prop_finite1}]

Define $\cR_3$ as in Proposition \ref{prop_CSS2}. In this proof we omit the dependence of constants on $\cR_1,\cR_2,\cR_3,\varepsilon$ for simplicity. Proposition \ref{prop_CSS1c} implies that for any $\epsilon>0$,
\begin{equation}
\Big|\Big\{t:\int_{\cR_2\le |\by| \le \cR_3}\mu_\sharp(t,\by)\rd{\by} \ge \epsilon\Big\}\Big| \le \frac{C}{\epsilon^4}.
\end{equation}
Take $\epsilon$ such that $C\epsilon \le c/2$ where $c,C$ are the constants appeared in \eqref{prop_CSS2_1}. Then, except a set $\cT\subset[0,\infty)$ with $|\cT|<\infty$, one has
\begin{equation}\label{CSS2E}
\frac{\rd}{\rd{t}}E[\rho(t,\cdot)] \le - c\cdot\int \sum_{j\ge 1,\,\textnormal{sharp}, R_j\le \cR_2}\int_{|x_2|\le r_j^*}\Big(\sqrt{R_{j}^2-x_2^2}-\sqrt{r_{j}^2-x_2^2}\Big)\rd{x_2} \rd{h},\quad \forall t\notin \cT,
\end{equation}
by applying Lemma \ref{lem_basic} with Proposition \ref{prop_CSS2}, in view of the fact that the cost of the CSS2 curve is controlled by the total moving mass
$\int \sum_{j\ge 1,\,\textnormal{sharp}, R_j\le \cR_2}\int_{|x_2|\le r_j^*}\Big(\sqrt{R_{j}^2-x_2^2}-\sqrt{r_{j}^2-x_2^2}\Big)\rd{x_2} \rd{h}$, by Lemma \ref{lem_cost}.

Notice that for fixed $j\ge 1,\,\textnormal{sharp}, R_j\le \cR_2$, if $R_j \ge 4\cR_1$, then
\begin{equation}
\frac{1}{2}\Big(\sqrt{R_{j}^2-x_2^2}+\sqrt{r_{j}^2-x_2^2}\Big) \ge \frac{1}{2}\sqrt{R_{j}^2-x_2^2} \ge \frac{1}{4}R_j \ge \cR_1,
\end{equation}
for any $|x_2|\le \min\{R_j/2,r_j\}$. Therefore by definition $r_j^*\ge \min\{R_j/2,r_j\} \ge r_j/2$. Therefore
\begin{equation}\begin{split}
& \int_{|x_2|\le r_j^*}\Big(\sqrt{R_{j}^2-x_2^2}-\sqrt{r_{j}^2-x_2^2}\Big)\rd{x_2} \rd{h} 
\ge  \int_{|x_2|\le r_j/2}\Big(\sqrt{R_{j}^2-x_2^2}-\sqrt{r_{j}^2-x_2^2}\Big)\rd{x_2} \rd{h} 
\ge  c|\tI_j|,
\end{split}\end{equation}
using the sharpness of $I_j$ (the last $c$ depends on $\varepsilon$). Combined with \eqref{CSS2E} and the lower bound of $E$ (Lemma \ref{lem_Em}), we obtain
\begin{equation}
\int_{t\notin\cT} \int_{4\cR_1\le |\bx|\le \cR_2}\mu_\sharp(t,\bx)\rd{\bx} \rd{t} \le C.
\end{equation}
The finiteness of the integral in $t\in \cT$ follows from the finiteness of $\cT$ and the uniform $L^\infty$ bound of $\rho$ (Lemma \ref{lem_reg1}).

\end{proof}

\subsection{The local clustering curve}

Fix $\cR_1>0$ large. Fix $2\cR_1\le \crr \le 4\cR_1$ to be determined. Define the \emph{local clustering curve} $\rho_t$ by the velocity field
\begin{equation}
\bv(\bx) = - \frac{|\bx|-\crr}{\cR_1} \cdot \chi_{\crr \le |\bx| \le 8\cR_1}(\bx)\cdot \frac{\bx}{|\bx|},
\end{equation}
i.e., $\rho_t$ satisfies
\begin{equation}
\partial_t \rho_t + \nabla\cdot (\rho_t \bv) = 0,
\end{equation}
with initial data $\rho_0=\rho$. $\rho_t$ is compressing the mass in $\{\crr \le |\bx| \le 8\cR_1\}$, and we have the lower bound
\begin{equation}\label{divv}
\nabla\cdot \bv = \frac{1}{\cR_1}(-2 + \frac{\crr}{|\bx|})\cdot \chi_{\crr \le |\bx| \le 8\cR_1}(\bx) + \frac{8\cR_1-\crr}{\cR_1}\delta(|\bx|-8\cR_1) \ge -\frac{2}{\cR_1}\cdot \chi_{\crr \le |\bx| \le 8\cR_1}(\bx).
\end{equation}
which means the compression effect is not too strong.

We recall the lemma of local clustering\footnote{The original lemma in \cite{S1D} is stated with $3R$ in the place of $4R$, but this current version can be proved in exactly the same way.} from~\cite{S1D}:
\begin{lemma}[Lemma of local clustering, Lemma 5.1 of~\cite{S1D}]\label{lem_r}
Let $R>0$, $\rho(x)$ be a decreasing non-negative function defined on $[R,4R]$, and $m>1$. Then there exists $r\in [R,2R]$ such that
\begin{equation}\label{lem_r_1}
\int_r^{4R} \rho(x)^m\rd{x} \le a\int_r^{4R}(x-r)\rho(x)\rd{x},\quad a = \frac{C_m}{R}\rho(R)^{m-1}.
\end{equation}
\end{lemma}

\begin{proposition}\label{prop_lcc}
Let $\cR_1$ be large enough, and $\varepsilon=\varepsilon(\cR_1,c_\rho)$ be small enough. Then there exists $\crr\in[2\cR_1,4\cR_1]$, such that the local clustering curve $\rho_t$ satisfies
\begin{equation}\label{prop_lcc_1}
\frac{\rd}{\rd{t}}\Big|_{t=0}E[\rho_t] \le - c \lambda(C\cR_1)\cdot \int_{\crr \le r \le 8\cR_1} (r-\crr)  \rho(r) \rd{r} + C\int_{2\cR_1\le |\by| \le 8\cR_1}\mu_\sharp(\by)\rd{\by},
\end{equation}
if  
\begin{equation}\label{prop_lcc_2}
\int_{2\cR_1\le |\by| \le \cR_3}\mu_\sharp(\by)\rd{\by}\le c \lambda(C\cR_1),
\end{equation}
for some large $\cR_3=\cR_3(\cR_1)$.
\end{proposition}

\begin{proof}

{\bf STEP 1}: estimate the decay of the interaction energy.

\begin{equation}\begin{split}
\frac{\rd}{\rd{t}}\Big|_{t=0} \cI[\rho_t] = & \cI[\rho, \partial_t |_{t=0}\rho_t] \\
= & - \int\int W(\bx-\by)\rho(\by)\rd{\by} \nabla\cdot(\rho(\bx) \bv(\bx))\rd{\bx} \\
= & \int_{\crr \le |\bx| \le 8\cR_1}\int \nabla W(\bx-\by)\rho(\by)\rd{\by}\rho(\bx) \bv(\bx)\rd{\bx}, \\
\end{split}\end{equation}
by noticing that $\bv(\bx)$ is zero except the range $\crr \le |\bx| \le 8\cR_1$. We decompose
\begin{equation}\begin{split}
\int \nabla W(\bx-\by)\rho(\by)\rd{\by} = & \int \nabla W(\bx-\by)\Big(\rho^*(\by) + \mu_\flat(\by) + \mu_\sharp(\by)\chi_{|\by|\le \crr}(\by) \Big) \rd{\by} \\
& + \int \nabla W(\bx-\by)\mu_\sharp(\by)\chi_{|\by|> \crr}(\by)\rd{\by}.
\end{split}\end{equation}

Corollary \ref{cor_ctr} (with $\cR_2=8\cR_1$) shows that, for $\cR_1$ large enough, for any $\crr \le |\bx| \le 8\cR_1$,
\begin{equation}
\int \nabla W(\bx-\by)\cdot \frac{\bx}{|\bx|}\Big(\rho^*(\by) + \rho_\flat(\by) + \rho_\sharp(\by)\chi_{|\by|\le \cR_1}(\by)\Big) \rd{\by} \ge c \lambda(C\cR_1),
\end{equation}
if $\varepsilon$ is small enough. Next, \eqref{WR231} and \eqref{WR232}, with $\cR_2$ replaced by $\crr\ge 2\cR_1$, implies that, for any $\crr \le |\bx| \le 8\cR_1$,
\begin{equation}
\left| \int \nabla W(\bx-\by)\mu_\sharp(\by)\chi_{[\crr,\infty)}(\by)\rd{\by} \right| \le C\int_{2\cR_1}^{\cR_3}\mu_\sharp(r)\rd{r} + C\frac{1}{\cR_3-8\cR_1}.
\end{equation}
We take $\cR_3$ large enough so that $C\frac{1}{\cR_3-8\cR_1} \le c \lambda(C\cR_1)/2$, and obtain
\begin{equation}
\int \nabla W(\bx-\by)\cdot \frac{\bx}{|\bx|}\rho(\by) \rd{\by} \ge c \lambda(C\cR_1) - C\int_{2\cR_1}^{\cR_3}\mu_\sharp(r)\rd{r}.
\end{equation}

Therefore, under the assumption \eqref{prop_lcc_2}, we have
\begin{equation}\label{lc_est1}\begin{split}
\frac{\rd}{\rd{t}}\Big|_{t=0} \cI[\rho_t] \le -c \lambda(C\cR_1) \int_{\crr \le |\bx| \le 8\cR_1} \frac{|\bx|-\crr}{\cR_1}  \rho(\bx) \rd{\bx}\le -c \lambda(C\cR_1) \int_{\crr \le r \le 8\cR_1} (r-\crr)  \rho(r) \rd{r}. \\
\end{split}\end{equation}

{\bf STEP 2}: estimate the increment of the interaction energy.

\begin{equation}\begin{split}
\frac{\rd}{\rd{t}}\Big|_{t=0} \cS[\rho_t] 
= & -m\int\rho(\bx)^{m-1} \nabla\cdot(\rho(\bx) \bv(\bx))\rd{\bx} 
=  m(m-1)\int\rho(\bx)^{m-1}\nabla\rho(\bx) \cdot \bv(\bx)\rd{\bx} \\
= & (m-1)\int\nabla(\rho(\bx)^{m}) \cdot \bv(\bx)\rd{\bx} 
=  -(m-1)\int\rho(\bx)^{m} \nabla \cdot \bv(\bx)\rd{\bx} \\
\le & \frac{2(m-1)}{\cR_1}\int_{\crr \le |\bx| \le 8\cR_1}\rho(\bx)^{m} \rd{\bx} 
\le  C\int_{\crr \le r \le 8\cR_1}\rho(r)^m \rd{r} \\
\le & C\int_{\crr \le r \le 8\cR_1}(\rho^*(r)+\mu_\flat(r))^m \rd{r} +  C\int_{2\cR_1 \le r \le 8\cR_1}\mu_\sharp(r) \rd{r}, \\
\end{split}\end{equation}
using \eqref{divv} in the first inequality, and the uniform bound of $\|\rho\|_{L^\infty}$ (from Lemma \ref{lem_reg1}) in the last inequality.

To estimate the first term in the last expression above, we define $\rho^{**}(r)$ as the non-increasing rearrangement of $\rho^*(r)+\mu_\flat(r)$ on $[2\cR_1,8\cR_1]$. See Figure \ref{fig_pic9} for illustration.

\begin{figure}
\begin{center}
  \includegraphics[width=.6\linewidth]{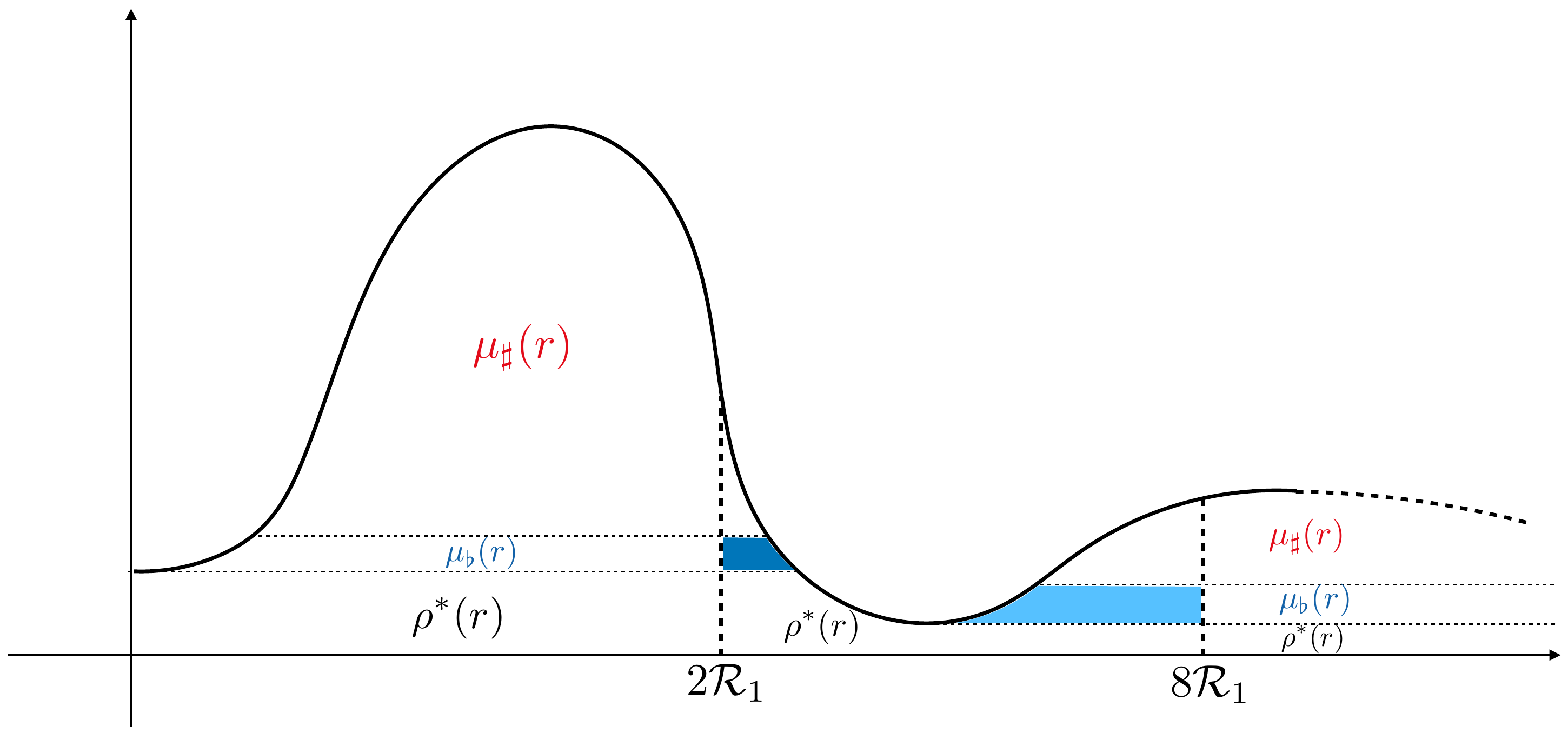}  
  \caption{The definition of $\rho^{**}$. In $\rho^*(r)+\mu_\flat(r)$, the right blue region is shifted towards left to obtain $\rho^{**}$. This change may make $\int_{\crr \le r \le 8\cR_1}\rho^{**}(r)^m \rd{r}$ smaller than $\int_{\crr \le r \le 8\cR_1}(\rho^*(r)+\mu_\flat(r))^m \rd{r}$ since some mass is shifted out of the interval $[\crr,8\cR_1]$, and this is compensated by the third term in \eqref{claim1}.}
\label{fig_pic9}
\end{center}
\end{figure}
 Then the RHS integral in \eqref{lc_est1} can be estimated as
\begin{equation}\begin{split}
\int_{\crr \le r \le 8\cR_1} (r-\crr)  \rho(r) \rd{r}  \ge & \int_{\crr \le r \le 8\cR_1} (r-\crr)  (\rho^*(r)+\mu_\flat(r)) \rd{r} \\
= & \int_{2\cR_1 \le r \le 8\cR_1} (r-\crr)\chi_{\crr \le r \le 8\cR_1}(r)  (\rho^*(r)+\mu_\flat(r)) \rd{r} \\
\ge & \int_{2\cR_1 \le r \le 8\cR_1} (r-\crr)\chi_{\crr \le r \le 8\cR_1}(r)  \rho^{**}(r) \rd{r} \\
= & \int_{\crr \le r \le 8\cR_1} (r-\crr)  \rho^{**}(r) \rd{r},
\end{split}\end{equation}
since $(r-\crr)\chi_{\crr \le r \le 8\cR_1}(r)$ is increasing on $[2\cR_1,8\cR_1]$. Also notice that
\begin{equation}\label{rhoss}
\rho^{**}(2\cR_1) = \sup_{2\cR_1\le r \le 8\cR_1}(\rho^*(r)+\mu_\flat(r)) \le \frac{C}{\cR_1^2},
\end{equation}
by \eqref{rhosmall}. 

Now we claim
\begin{equation}\label{claim1}\begin{split}
& \int_{\crr \le r \le 8\cR_1}(\rho^*(r)+\mu_\flat(r))^m \rd{r} \\ \le & C\Big(\int_{\crr \le r \le 8\cR_1}\rho^{**}(r)^m \rd{r} + \frac{\rho^{**}(2\cR_1)^{m-1}}{\cR_1}\int_{\crr \le r \le 8\cR_1} (r-\crr)  \rho(r) \rd{r}\Big),
\end{split}\end{equation}
if $\varepsilon<1/4$. To see this, we write
\begin{equation}\label{rhoss1}
(\rho^*(r)+\mu_\flat(r))\chi_{2\cR_1 \le r \le 8\cR_1}(r) = \int_0^\infty \Big(\chi_{I_0\cap[2\cR_1,8\cR_1]}(r) + \sum_{j\ge 1,\,\text{flat}} \chi_{I_j\cap[2\cR_1,8\cR_1]}(r)\Big)\rd{h}.
\end{equation}
Clearly $I_0\cap[2\cR_1,8\cR_1]$ stays the same under the non-increasing rearrangement $\rho^{**}$, see the left blue region in Figure \ref{fig_pic9}. For $j\ge 1,\,\text{flat}$, if $\varepsilon<1/4$, then $I_j$ cannot be a subset of $[2\cR_1,8\cR_1]$. We separate into two cases:
\begin{itemize}
\item $2\cR_1\in I_j$. In this case $I_j\cap[2\cR_1,8\cR_1]$ stays the same under the non-increasing rearrangement. 
\item $2\cR_1\notin I_j$. In this case $8\cR_1\in I_j$, and $I_j\cap[2\cR_1,8\cR_1]$ is shifted towards the center by the distance $\tau_j=r_j-R_{j-1}>0$, under the non-increasing rearrangement, see the right blue region in Figure \ref{fig_pic9}. 
\end{itemize}
Therefore
\begin{equation}\label{rhoss2}\begin{split}
\rho^{**}(r) = & \int_0^\infty \Big(\chi_{I_0\cap[2\cR_1,8\cR_1]}(r) + \sum_{j\ge 1,\,\text{flat},\,2\cR_1\in I_j} \chi_{I_j\cap[2\cR_1,8\cR_1]}(r) \\
& + \sum_{j\ge 1,\,\text{flat},\,2\cR_1\notin I_j} \chi_{(I_j-\tau_j)\cap[2\cR_1,8\cR_1]}(r)\Big)\rd{h}.
\end{split}\end{equation}

Notice that for any function $\psi(r) = \int_0^\infty \chi_{S(h)}(r)\rd{h}$ with $S(h) = \{r:\psi(r)\ge h\}$, one has
\begin{equation}
\int \psi(r)^m\rd{r} = \int_0^\infty |S(h)|\cdot mh^{m-1}\rd{h}.
\end{equation}
Therefore, in \eqref{rhoss1} and \eqref{rhoss2}, taking the intersection of every involved interval with $[\crr,8\cR_1]$ and integrating in $r$ gives
\begin{equation}\label{rhossm0}\begin{split}
\int_{\crr \le r \le 8\cR_1}(\rho^*(r)+\mu_\flat(r))^m\rd{r} = & \int_0^\infty \Big(|I_0\cap[\crr,8\cR_1]| + \sum_{j\ge 1,\,\text{flat}} |I_j\cap[\crr,8\cR_1]|\Big)\cdot mh^{m-1}\rd{h} ,
\end{split}\end{equation}
and
\begin{equation}\label{rhossm1}\begin{split}
\int_{\crr \le r \le 8\cR_1}\rho^{**}(r)^m \rd{r} = & \int_0^\infty \Big(|I_0\cap[\crr,8\cR_1]| + \sum_{j\ge 1,\,\text{flat},\,2\cR_1\in I_j} |I_j\cap[\crr,8\cR_1]| \\
& + \sum_{j\ge 1,\,\text{flat},\,2\cR_1\notin I_j} |(I_j-\tau_j)\cap[\crr,8\cR_1]|\Big)\cdot mh^{m-1}\rd{h}.
\end{split}\end{equation}
Also notice that
\begin{equation}\label{rhossm2}\begin{split}
& \frac{\rho^{**}(2\cR_1)^{m-1}}{\cR_1}\int_{\crr \le r \le 8\cR_1} (r-\crr)  \rho(r) \rd{r} \\
\ge & \rho^{**}(2\cR_1)^{m-1}\int_{6\cR_1 \le r \le 8\cR_1}   \rho(r) \rd{r} \\
\ge & \rho^{**}(2\cR_1)^{m-1}\int_{6\cR_1 \le r \le 8\cR_1}   (\rho^*(r)+\mu_\flat(r)) \rd{r} \\
\ge & \int_{6\cR_1 \le r \le 8\cR_1}   (\rho^*(r)+\mu_\flat(r))^m \rd{r} \\
= & \int_0^\infty \Big(|I_0\cap[6\cR_1,8\cR_1]| + \sum_{j\ge 1,\,\text{flat}} |I_j\cap[6\cR_1,8\cR_1]|\Big)\cdot mh^{m-1}\rd{h},
\end{split}\end{equation}
using $\crr\le 4\cR_1$ and the equality in \eqref{rhoss}. 

Therefore, to show \eqref{claim1}, the terms with $I_0$ and $I_j,\,j\ge 1,\,\text{flat},\,2\cR_1\in I_j$ in \eqref{rhossm0} can be controlled by  \eqref{rhossm1}, and  it suffices to show that for any $j\ge 1,\,\text{flat},\,2\cR_1\notin I_j$, there holds
\begin{equation}\label{rhossm3}
|I_j\cap[\crr,8\cR_1]| \le C|I_j\cap[6\cR_1,8\cR_1]|,
\end{equation}
and thus controlled by \eqref{rhossm2}. In fact, since $8\cR_1\in I_j$, we may separate into two cases:
\begin{itemize}
\item If $I_j\cap[\crr,8\cR_1]\subset [6\cR_1,8\cR_1]$, then \eqref{rhossm3} holds with $C=1$, since $|I_j\cap[\crr,8\cR_1]|\le |I_j|$ and $|I_j\cap[6\cR_1,8\cR_1]|=|I_j\cap[\crr,8\cR_1]|$.
\item Otherwise, we have $[6\cR_1,8\cR_1]\subset I_j$ since $8\cR_1\in I_j$. Therefore $|I_j\cap[6\cR_1,8\cR_1]|=2\cR_1$ and $|I_j\cap[\crr,8\cR_1]| \le 8\cR_1-\crr \le 6\cR_1$, and thus \eqref{rhossm3} holds with $C=3$.
\end{itemize} 
Therefore we have proved \eqref{claim1}.

Now we choose $\crr$ by Lemma \ref{lem_r} applied to $\rho^{**}(r)$ on $[2\cR_1,8\cR_1]$, and then
\begin{equation}
\int_{\crr \le r \le 8\cR_1}\rho^{**}(r)^m \rd{r} \le \frac{C}{\cR_1}\rho^{**}(2\cR_1)^{m-1}\int_{\crr \le r \le 8\cR_1} (r-\crr)  \rho^{**}(r) \rd{r} .
\end{equation}
Combining with \eqref{claim1}, we have
\begin{equation}
\int_{\crr \le r \le 8\cR_1}(\rho^*(r)+\mu_\flat(r))^m \rd{r} \le C\frac{\rho^{**}(2\cR_1)^{m-1}}{\cR_1}\int_{\crr \le r \le 8\cR_1} (r-\crr)  \rho(r) \rd{r}.
\end{equation}

{\bf STEP 3}: take $\cR_1$ large to finalize.

Combining STEP 1 and STEP 2, we have 
\begin{equation}\begin{split}
\frac{\rd}{\rd{t}}\Big|_{t=0} E[\rho_t] \le & -\Big(c \lambda(C\cR_1) - C\frac{\rho^{**}(2\cR_1)^{m-1}}{\cR_1}\Big)  \int_{\crr \le r \le 8\cR_1} (r-\crr)  \rho(r) \rd{r} \\ & + C\int_{2\cR_1 \le r \le 8\cR_1}\mu_\sharp(r) \rd{r},
\end{split}\end{equation}
under the assumption \eqref{prop_lcc_2}. By \eqref{lambda} (a consequence of {\bf (A1)}) and \eqref{rhoss},
\begin{equation}
\lambda(C\cR_1) = c\cR_1^{-\alpha},\quad \frac{\rho^{**}(2\cR_1)^{m-1}}{\cR_1} \le C\cR_1^{-(2m-1)}.
\end{equation}
Therefore, with {\bf (A3)}, we can choose $\cR_1$ large enough, so that $C\frac{\rho^{**}(2\cR_1)^{m-1}}{\cR_1} \le c \lambda(C\cR_1)/2$, and obtain the conclusion.

\end{proof}

\begin{proof}[Proof of Proposition \ref{prop_finite2}]

By applying \eqref{prop_lcc_1} to $\rho(t,\cdot)$ (the solution to \eqref{eq0}) for any fixed $t\ge 0$ and using \eqref{lambda} and Lemma \ref{lem_basic}, we obtain
\begin{equation}\label{R1E2}\begin{split}
& \frac{\rd}{\rd{t}}E[\rho(t,\cdot)] \\
\le & -\frac{\left[- c_1 \cR_1^{-\alpha} \int_{\crr(t) \le r \le 8\cR_1} (r-\crr(t))  \rho(t,r) \rd{r} + C_1\int_{2\cR_1\le |\by| \le 8\cR_1}\mu_\sharp(t,\by)\rd{\by}\right]^2}{\cR_1^{-2}\int_{\crr(t) \le r \le 8\cR_1} (r-\crr(t))^2  \rho(t,r) \rd{r}},  \\
\end{split}\end{equation}
provided that  the quantity inside the bracket is negative, and \eqref{prop_lcc_2} holds, i.e.,
\begin{equation}
\int_{2\cR_1\le |\by| \le \cR_3}\mu_\sharp(t,\by)\rd{\by}\le c_2\cR_1^{-\alpha}.
\end{equation}
Here $\crr(t)$ is determined by $\rho(t,\cdot)$, satisfying $2\cR_1\le \crr(t)\le 4\cR_1$. The denominator in \eqref{R1E2} is the cost of the local clustering curve, and $c_1,C_1,c_2$ are constants.  

Then we define the sets $\cT_{1,2,3}\subset [0,\infty)$ as follows:
\begin{itemize}
\item $\cT_1$ contains those $t$ with $\int_{2\cR_1\le |\by| \le 8\cR_1}\mu_\sharp(t,\by)\rd{\by} \ge c_2\cR_1^{-\alpha}$. Proposition \ref{prop_finite1} shows that $|\cT_1| <\infty$. It follows that
\begin{equation}
\int_{t\in \cT_1} \int_{\crr(t) \le r \le 8\cR_1} (r-\crr(t))  \rho(t,r) \rd{r}\rd{t} < \infty.
\end{equation}
\item $\cT_2$ contains those $t\notin \cT_1$ with 
\begin{equation}
\int_{\crr(t) \le r \le 8\cR_1} (r-\crr(t))  \rho(t,r) \rd{r} \le \frac{2C_1\cR_1^\alpha}{c_1}\int_{2\cR_1\le |\by| \le 8\cR_1}\mu_\sharp(t,\by)\rd{\by} .
\end{equation}
  It follows that
\begin{equation}
\int_{t\in \cT_2} \int_{\crr(t) \le r \le 8\cR_1} (r-\crr(t))  \rho(t,r) \rd{r}\rd{t} \le \frac{2C_1\cR_1^\alpha}{c_1}\int_{t\in \cT_2} \int_{2\cR_1\le |\by| \le 8\cR_1}\mu_\sharp(t,\by)\rd{\by}\rd{t} < \infty,
\end{equation}
where the second inequality uses Proposition \ref{prop_finite1}.
\item $\cT_3=[0,\infty)\backslash(\cT_1\cup\cT_2)$. It then follows that \eqref{prop_lcc_2} holds, and the quantity inside the bracket in \eqref{R1E2} is less than $- \frac{c_1 \cR_1^{-\alpha} }{2}\int_{\crr(t) \le r \le 8\cR_1} (r-\crr(t))  \rho(t,r) \rd{r}$, and then
\begin{equation}\begin{split}
\frac{\rd}{\rd{t}}E[\rho(t,\cdot)] 
\le & -c\cR_1^{-2\alpha}\frac{\left[  \int_{\crr(t) \le r \le 8\cR_1} (r-\crr(t))  \rho(t,r) \rd{r} \right]^2}{\cR_1^{-2}\int_{\crr(t) \le r \le 8\cR_1} (r-\crr(t))^2  \rho(t,r) \rd{r}} \\
\le & -c\cR_1^{-2\alpha+1} \int_{\crr(t) \le r \le 8\cR_1} (r-\crr(t))  \rho(t,r) \rd{r},  \\
\end{split}\end{equation}
for $t\in \cT_3$, using the fact that $0\le r-\crr(t) \le C\cR_1$ for any $\crr(t) \le r \le 8\cR_1$. Therefore
\begin{equation}
\int_{t\in \cT_3} \int_{\crr(t) \le r \le 8\cR_1} (r-\crr(t))  \rho(t,r) \rd{r}\rd{t} < \infty.
\end{equation}
\end{itemize}
Combining the three parts gives
\begin{equation}
\int_0^\infty \int_{\crr(t) \le r \le 8\cR_1} (r-\crr(t))  \rho(t,r) \rd{r}\rd{t} < \infty.
\end{equation}
Then the final conclusion follows, by noticing that 
\begin{equation}
\int_{\crr(t) \le r \le 8\cR_1} (r-\crr(t))  \rho(t,r) \rd{r} \ge c\cR_1\int_{7\cR_1 \le r \le 8\cR_1}   \rho(t,r) \rd{r} \ge c \int_{7\cR_1 \le |\bx| \le 8\cR_1}\rho(t,\bx)\rd{\bx},
\end{equation}
since $\crr(t) \le 4\cR_1$.

\end{proof}

\section*{Acknowledgement}

The author would like to thank Jos\'{e} Carrillo, Franca Hoffmann, Dave Levermore and Eitan Tadmor for helpful discussions. The author would also like to thank Jiuya Wang for giving a first proof of Proposition \ref{prop_meow}.

\bibliographystyle{plain}
\bibliography{radial_convergence_2D_bib}

\end{document}